\documentclass[reqno]{amsart}[11pt] 
\usepackage{lineno}

\usepackage[spanish,english]{babel}
\usepackage{amsthm}
\usepackage{graphicx}
\usepackage{float}
\usepackage{mathtools}
\usepackage{animate}
\usepackage{amsmath}	
\usepackage{amsfonts}
\usepackage{amssymb}
\usepackage{latexsym}
\usepackage{pst-all}
\usepackage{pstricks}
\usepackage{tikz}
\usepackage{enumerate}
\usepackage{bm}
\usepackage{algorithm}
\usepackage{algpseudocode}
\usepackage[pdflatex]{epsfig}
\usepackage{cite} 
\usepackage{chngcntr}
\usepackage[colorlinks=true,citecolor=blue]{hyperref}
\newtheorem{theorem}{Theorem}[section]
\newtheorem{prop}[theorem]{Proposition}%
\newtheorem{definition}[theorem]{Definition}%
\newtheorem{lemma}[theorem]{Lemma}%
\usepackage{booktabs}
\newtheorem{remark}{Remark}%
\usepackage{here}
\usepackage{stmaryrd}
\usepackage{algorithm}
\usepackage{algpseudocode}
\usepackage{bm}
\algrenewcommand\algorithmicrequire{\textbf{Input:}}
\algrenewcommand\algorithmicensure{\textbf{Output:}}
\newcommand{\diag}{\operatorname*{diag}}

\numberwithin{equation}{section}

\title[Dynamically consistent FV method for AD]{Dynamically consistent finite volume scheme for a bimonomeric simplified model with inflammation processes for Alzheimer's disease}

\author[Barajas-Calonge]{Juan Barajas-Calonge$^{\mathrm{A}}$}
\author[Sepúlveda]{Mauricio Sepúlveda$^{\mathrm{B}}$}
\author[Torres]{Nicolás Torres$^{\mathrm{C}}$}
\author[Villada]{Luis Miguel Villada$^{\mathrm{D}}$}

\date{\today}

\begin{document}
\begin{abstract}   
A model of progression of Alzheimer’s disease (AD) incorporating the interactions of A$\beta$-monomers, oligomers, microglial cells and interleukins with neurons is considered.  The resulting convection-diffusion-reaction system consists of four partial differential equations (PDEs) and one ordinary differential equation (ODE). We develop a finite volume (FV) scheme for this system, together with non-negativity and a priori bounds for the discrete solution, so that we establish the existence of a discrete solution to the FV scheme. It is shown that the scheme converges to an admissible weak solution of the model. The reaction terms of the system are discretized using a semi-implicit strategy that coincides with a nonstandard discretization of the spatially homogeneous (SH) model. This construction enables us to prove that the FV scheme is dynamically consistent with respect to the spatially homogeneous version of the model. Finally, numerical experiments are presented to illustrate the model and to assess the behavior of the FV scheme.
\end{abstract}

\thanks{$^{\mathrm{A}}$GIMNAP-Departamento de Matem\'{a}tica, Universidad del B\'{\i}o-B\'{\i}o,
 Casilla
5-C, Concepci\'{o}n, Chile. E-Mail: {\tt
 juan.barajas2001@alumnos.ubiobio.cl}}
\thanks{$^{\mathrm{B}}$CI$^{\mathrm{2}}$MA and Departamento de Ingenier\'{\i}a
Matem\'{a}tica, Facultad de Ciencias F\'{i}sicas y Matem\'{a}ticas, Universidad
de Concepci\'{o}n, Casilla 160-C, Concepci\'{o}n, Chile.  E-Mail: {\tt mauricio@ing-mat.udec.cl}}
\thanks{$^{\mathrm{C}}$Université Côte d'Azur, LJAD, 28 Avenue Valrose, Nice, France. E-Mail: {\tt
		nicolas.torres@univ-cotedazur.fr}}
\thanks{$^{\mathrm{D}}$GIMNAP-Departamento de Matem\'{a}tica, Universidad del B\'{\i}o-B\'{\i}o,
	Casilla
	5-C, Concepci\'{o}n, Chile, and   CI$^{\mathrm{2}}$MA, 
	Universidad de Concepci\'{o}n, 
	Casilla 160-C, Concepci\'{o}n, Chile. E-Mail: {\tt
		lvillada@ubiobio.cl}} 
\maketitle

\section{Introduction} 
\subsection{Scope}

Alzheimer's disease (AD) is a progressive neurodegenerative disorder characterized pathologically by the accumulation of certain protein aggregates in the brain. A primary component of these aggregates is the Amyloid-beta (A$\beta$) peptide, which is considered a central player in the disease's pathogenesis. The ``amyloid cascade hypothesis'' posits that an imbalance between the production and clearance of A$\beta$-peptides leads to their aggregation into soluble oligomers and insoluble plaques, triggering a cascade of neurotoxic events leading to irreversible neuronal damage (see, for example, \cite{haass_soluble_2007,sakono_amyloid_2010,sengupta_role_2016,soto_unfolding_2003,cohen_proliferation_2013}).

Dynamics of Amyloid-beta are not solely governed by its intrinsic aggregation kinetics. The brain's innate immune cells, microglia, also play a critical and dual role in the formation of AD through an inflammation reaction in presence of A$\beta$-oligomers, releasing interleukins (cytokines like IL-1) that stimulates neurons to produce more A$\beta$-monomers \cite{forloni_alzheimers_2018,kinney_inflammation_2018,al-ghraiybah_glial_2022,lopategui_cabezas_role_2014}. However, if the concentration of A$\beta$-oligomers is high enough, then a reaction of stress called UPR (Unfolded protein response) \cite{soto_unfolding_2003} is triggered which leads to a decrease of A$\beta$-monomers production, while the rest of oligomers diffuses in the neuronal environment. In this context, two opposed mechanisms of stimulation and inhibition will determine the persistence of AD or not.

Moreover, microglia are not static; they are recruited to sites of emergence through chemotactically directed movement of cells along a chemical gradient. In this case A$\beta$ aggregates themselves, create a concentration field that guides microglial migration. Upon arrival, microglia attempt to encapsulate and degrade the aggregates, forming a protective barrier around the well-known amyloid plaques. In Alzheimer's disease, this protective system becomes impaired and can turn pathological. Chronic exposure to A$\beta$ can lead to microglial dysfunction, reducing their phagocytic efficiency, while pro-inflammatory cytokines exacerbate neural damage.

In this context, we consider a convection-diffusion-reaction system describing the progression of Alzheimer’s disease (AD) incorporating the interactions of A$\beta$-monomers, A$\beta$-oligomers, microglial cells and interleukins with neurons through different mechanisms such as protein polymerization, inflammation processes and neural stress reactions. The governing model, following the approach in \cite{ciuperca2024qualitative,estavoyer2025spatial}, is a strongly coupled nonlinear system of four partial differential equations (PDEs) and one ordinary differential equation (ODE), namely four parabolic equations describing the evolution of the concentrations of A$\beta$-oligomers $u_1$,  A$\beta$-monomers $u_3$, microglial cells $u_4$, and interleukins $u_5$ coupled with and ordinary differential equation modeling the concentration of oligomers in the amyloid plaques $u_2$. The final model is given by the system
\begin{equation}\label{eq:model}
	\begin{aligned}
		\partial_t u_1-d_1 \Delta u_1 &= F_1(\boldsymbol{u}),\\
		\partial_t u_2 &= F_2(\boldsymbol{u}),\\
		\partial_t u_3-d_3 \Delta u_3 &= F_3(\boldsymbol{u}),\\
		\partial_t u_4-d_4 \Delta u_4+ \nabla \cdot (\chi(u_4)\nabla u_1) &= F_4(\boldsymbol{u}),\\
		\partial_t u_5-d_5 \Delta u_5 &= F_5(\boldsymbol{u}),\\
	\end{aligned}
\end{equation}
where the variables are understood as functions of position $\boldsymbol{x}\in \Omega$ and time $t\in [0,T]$ on a bounded domain $\Omega \subset \mathbb{R}^2$, $d_i$ is the diffusion coefficient of $u_i$, $i=1,3,4,5$ and the convective term in the fourth equation represents the chemotaxis of microglial cells in response to the increase of oligomer population. This chemotactic effect leads to the activation of microglial cells in response to the presence of oligomers, triggering an inflammatory reaction accompanied by the production of interleukins. In this work, we consider a sensitivity function $\chi(u)$ such that $\chi(0)=0$ and which also vanishes once the concentration of microglial cells reaches the recruitment threshold $\hat{m}$, equivalently, $\chi(\hat{m})=0$. Biologically, the threshold condition indicates that when the cell density at a specific location in $\Omega$ attains the critical value $\hat{m}$, further accumulation of cells at that point is halted. This phenomenon is often described as the volume-filling effect, also known as overcrowding prevention \cite{hillen2001global,perthame2009existence,andreianov2011finite}. A common choice of the function $\chi$ is
\begin{equation}\label{eq:chifun}
	\chi(u)=\alpha u(\hat{m}-u), \quad u\in [0,\hat{m}].
\end{equation}
The functional responses $F_i$, $i=1,2,\dots,5$ are of the type production-destruction system of equations (PDS). The function $F_3$ includes a stress term $S(u_1, u_5)$ from UPR phenomenon, which indicates that when the concentration of oligomers $u_1$ surrounding the neuron is high, the neuron becomes stressed and ceases the production of $A\beta$ monomers. The function $F_4$ includes a constant rate of proliferation and a logistic growth of the microgial cells while the function $F_5$ incorporates a proliferation which depends on the concentration of oligomers through a Michaelis–Menten function. The functional responses are therefore given by
\begin{equation}\label{eq:RHS}
	\begin{aligned}
		F_1(\boldsymbol{u})&:= r_1 u_3^2-\gamma(u_4) u_1-\tau_0 u_1,\\
		F_2(\boldsymbol{u})&:= \gamma(u_4) u_1 -\tau_p u_2,\\
		F_3(\boldsymbol{u})&:= S(u_1, u_5)-d u_3-r_2 u_1 u_3-r_1 u_3^2,\\
		F_4(\boldsymbol{u})&:=\dfrac{\alpha_1 u_1}{1+\alpha_2 u_1}(\hat{m}-u_4)u_4-\sigma u_4+\lambda_M,\\
		F_5(\boldsymbol{u})&:=\dfrac{\tau_1 u_1}{1+\tau_2 u_1}u_4-\tau_3 u_5,
	\end{aligned}
\end{equation}
where $S(u_1, u_5)=\frac{\tau_S}{1+C u_1^{\nu}}u_5$, $r_1$ is the bi-monomeric polymerization rate, $r_2$ is the polymerization rate of monomers attaching to oligomers, $d$ is the degradation rate of monomers, $\tau_1, \tau_2$ are the growth coefficients of interleukins, $\tau_3$ is the degradation rate of interleukins, $\tau_p$ is the degradation rate of oligomers in the amyloid plaques, $\alpha_1, \alpha_2$ are growth coefficients of microglial cells, $\lambda_M = \lambda_M(\boldsymbol{x})$ is a spatially dependent function modeling the proliferation of microglial cells, $\hat{m}$ is the capacity of microglial cells, $\sigma$ is the degradation rate of microglial cells and the function $\gamma$ characterizes the recruitment rate of oligomers into amyloid plaques and is assumed to satisfy the following properties:
\begin{equation}\label{eq:prop-gamma}
	\begin{aligned}
		&\gamma(u)>0, \quad |\gamma(u)-\gamma(v)|\leq L_{\gamma} |u-v|, \quad \text{ for all }u,v\geq 0,\\
		&\gamma_{\mathrm{min}}\leq \gamma(u)\leq \gamma_{\mathrm{max}},\quad \text{ for all }u\geq 0,\\
	\end{aligned}
\end{equation}
where $L_{\gamma}, \gamma_{\max},$ and $\gamma_{\min}$ are positive constants. Some possible options for $\gamma$ are the constant function $\gamma(u)=\gamma_0$ or the Michaelis–Menten function which is given by $\gamma(u)=\gamma_0+\frac{\gamma_1 u}{1+\gamma_2 u}$, with $\gamma_1,\gamma_2>0$.\\

In addition, we impose homogeneous Neumann boundary conditions on the boundary $\partial \Omega\times(0,T)$, which can be expressed by
\begin{equation}\label{eq:bound-cond}
	\nabla u_i \cdot \boldsymbol{n} =0, \quad i=1,2,3,5, \quad (d_4 \nabla u_4-\chi(u_4) \nabla u_1)\cdot \boldsymbol{n} = 0,
\end{equation}
where $\boldsymbol{n}$ is the unit normal vector to the boundary $\partial \Omega$ pointing outwards from the domain. To close the model, we set the following initial conditions
\begin{equation}\label{eq:ini-cond}
	u_{i}(\boldsymbol{x},0)=u_{i,0}(\boldsymbol{x}), \quad \textrm{ in }\Omega.
\end{equation}

The purpose of this work is to establish the existence of an admissible weak solutions to the initial–boundary value problem \eqref{eq:model}, \eqref{eq:bound-cond}, \eqref{eq:ini-cond}, while ensuring dynamic consistency with the spatially homogeneous (SH) version of \eqref{eq:model} $\boldsymbol{u}_t=\boldsymbol{F}(\boldsymbol{u})$, where $\boldsymbol{F} = (F_1,\dots,F_5)^{\mathrm{T}}$. To this end, we propose an unconditionally stable finite volume (FV) method for the numerical approximation of the system, in which the reaction term is discretized in a novel semi-implicit form which basically follows the NonStandard Finite Difference (NSFD) philosophy \cite{mickens1993nonstandard}). We prove that the limit of the discrete solutions constitutes an admissible weak solution of system \eqref{eq:model}. In addition, we show that our numerical scheme is dynamically consistent with the SH model in the sense that the non-spatial FV scheme preserves the positivity and boundedness of the solution, as well as the equilibrium points and the stability properties of the disease-free equilibrium, when $\gamma$ is taken to be a constant function. Moreover, we present numerical simulations to illustrate the chemotactic behavior and to highlight the role of the chemotactic coefficients in governing the movement of each species, either toward regions of higher concentration or toward regions of lower concentration. Finally, we demonstrate that our FV scheme is capable of reproducing the Turing patterns reported in the literature.
\subsection{Related work}  \label{subsec:related} 
From a mathematical perspective, several works have investigated partial differential equation models related to Alzheimer’s disease (AD). Andrade et al. \cite{andrade2020modeling} proposed a spatially dependent model for the polymerization of A$\beta$–proteins. Ciuperca et al. \cite{ciuperca_alzheimers_2017} studied the formation of A$\beta$-oligomers and fibrils using a continuous-size framework based on the Lifshitz–Slyozov equations. Hao et al. \cite{hao2016mathematical} developed a multi-component model incorporating astrocytes, microglial cells, and peripheral cell populations in addition to A$\beta$–oligomers. For models inspired by the A$\beta$-system considered in \cite{ciuperca2024qualitative}, further developments include optimal control formulations for anti-inflammatory treatments \cite{torres2025optimal,el2025modelling}. Other reaction–diffusion models related to neurodegenerative processes can be found in \cite{matthaus_diffusion_2006,matthaus_spread_2009,bertsch_alzheimers_2016}.

From a numerical perspective, a wide variety of methods have been employed to approximate systems similar to \eqref{eq:model}. In the context of finite volume (FV) schemes, Angelini et al. \cite{angelini2013finite} developed an FV method for general meshes targeting abstract degenerate parabolic convection–reaction–diffusion equations. FV techniques have also been used for numerical simulations in diverse biological and medical applications, including bone growth models \cite{coudiere2013analysis}, early-stage breast cancer progression \cite{foucher2018convergence,alotaibi2024computational}, epidemic transmission dynamics \cite{bendahmane2009convergence}, tumor growth \cite{anaya2009mathematical,anaya2010numerical}, and food-chain interactions \cite{anaya2015numerical,burger2020numerical}, among others.

Additional numerical approaches for biological models comparable to \eqref{eq:model} include Discontinuous Galerkin (DG) schemes \cite{corti2023discontinuous,antonietti2024discontinuous,pederzoli2025coupled} and Virtual Element Methods \cite{wang2025mixed}. However, many of these numerical strategies fail to preserve essential biological properties of the continuous system, such as positivity, boundedness, or the stability of equilibrium states. To address these limitations, several recent works focus on structure-preserving schemes. For Keller–Segel type models \cite{nino2021convergence,lopez2023numerical,lopez2023theoretical}, semi-implicit Euler time discretizations combined with finite element spatial discretizations have been used to guarantee discrete maximum principles for certain variables. For Cahn–Hilliard type systems \cite{acosta2023upwind,acosta2025property,careaga2026invariant}, upwinding DG schemes have been developed to ensure maximum-principle preservation. This line of research also motivates the development of NSFD schemes, which have been extensively used over the past years to approximate ODE systems in a wide range of applications, see \cite{mickens2005dynamic,patidar2016nonstandard,sharma2021review} and the references therein for reviews on this topic. One of the advantages of this discretization approach is that it is unconditionally stable, which constitutes a significant improvement compared with the traditional explicit Euler scheme (see Section \ref{sec:numexa}, Example 1). In addition, these schemes allow us to analytically prove the preservation of important qualitative features of the dynamical system under study, such as positivity, boundedness, equilibrium points and their stability unlike other methods, including the implicit Euler scheme or Runge–Kutta type methods. The NSFD methods have also been exploited in the treatment of one-dimensional PDEs \cite{hernandez2013nonstandard,namjoo2018numerical} and in two-dimensional problems on structured grids, in combination with finite difference schemes \cite{rao2021performance,kumar2022new}.

We emphasize that the models and numerical methodologies discussed above differ from those considered in this work. Our approach combines a traditional finite volume discretization for the spatial operators with a semi-implicit NSFD treatment of the reaction terms. This hybrid strategy ensures not only the classical stability and convergence properties expected within the FV framework, but also preserves the qualitative dynamics of the spatially homogeneous (SH) model, an essential feature for accurately capturing long-term behavior and the emergence of Turing patterns. Moreover, an additional advantage of our approach is its ability to handle complex geometries, making the scheme well suited to simulate realistic biological and biomedical scenarios.
\subsection{Outline of the paper} \label{subsec:outline} 
The remainder of the paper is organized as follows. Section \ref{sec:prelim} provides some preliminaries that include some dynamical properties of the SH model and the definition of an admissible weak solution for the convection-diffusion-reaction model \eqref{eq:model}. In Section \ref{sec:fvm}, we describe the FV method, first recalling in Section \ref{sec:spacetimedisc} the standard admissible-mesh notation from \cite{EYMARD2000713}. In Section \ref{sec:fv-scheme}, we present the FV discretization of equations \eqref{eq:model}, where we approximate the reaction term by a NSFD approach. Since the method is implicit and requires solving nonlinear algebraic systems at each time level, we must establish that the scheme is well defined and admits a solution at each step. This is addressed in Section \ref{sec:exist}, after first demonstrating in Lemma \ref{lemma:maxprin} that any discrete solution generated by a truncated FV scheme remains in the invariant rectangle. These results allow us to prove in Lemma \ref{lemma:exist-truncated} the existence and uniqueness of a solution for the truncated FV scheme, and then in Theorem \ref{thm:exist-fvm}, we use an inductive argument to show the existence of the FVM scheme for \eqref{eq:model}. In Section \ref{sec:de} we show an a priori $L^2$ estimate for the discrete solutions required to prove the convergence of the scheme. Section \ref{sec:conver} focuses on establishing the convergence of the FV scheme as the mesh size tends to zero. Consequently, in Section \ref{sec:compact} we establish compactness results for the family of discrete solutions, and in Section \ref{sec:conv-ana} we verify that any limit point of these solutions is an admissible weak solution of \eqref{eq:model}. In Section \ref{sec:numexa}, we present three numerical tests. Example 1 illustrates the robustness of the NSFD discretization, while Example 2 examines the chemotactic response of microglial cells to an increased oligomer population, and Example 3 investigates the formation of Turing patterns.

\section{Some properties of the continuous model}\label{sec:prelim}
\subsection{Dynamics of the SH model}
Let us focus first on the SH model. In \cite{ciuperca2024qualitative}, the authors proved the positivity and boundedness (Proposition 3.1) of the solutions of this ODE system. Here, we introduce a rectangle $\mathcal{R}$, within which the bounds for each variable are explicit, and we prove that $\mathcal{R}$ is invariant by the semi-flow generated by the system $\boldsymbol{u}_t = \boldsymbol{F}(\boldsymbol{u})$ for $\hat{m}$ large enough. 
\begin{prop}\label{prop:inv-R}
	Let us assume that $\hat{m}\geq \frac{\lambda_M}{\sigma}$ and consider the rectangular region $\mathcal{R} = \prod_{j=1}^{5}[0,\beta_j]\subset \mathbb{R}^5$, where $\beta_1 = \frac{r_1}{\tau_0+\gamma_{\min}}(\frac{\tau_S \tau_1 \hat{m}}{\tau_2 \tau_3 d})^2$, $\beta_2 = \frac{\gamma_{\max} r_1}{\tau_p (\tau_0+\gamma_{\min})}(\frac{\tau_S \tau_1 \hat{m}}{\tau_2 \tau_3 d})^2$, $\beta_3 = \frac{\tau_S \tau_1 \hat{m}}{\tau_2 \tau_3 d}$, $\beta_4 = \hat{m}$, and $\beta_5 = \frac{\tau_1 \hat{m}}{\tau_2 \tau_3}$. Then, the region $\mathcal{R}$ is invariant by the semi-flow generated by the system $\boldsymbol{u}_t=\boldsymbol{F}(\boldsymbol{u})$.
\end{prop}
\begin{proof}
	Since $\boldsymbol{F}$ is Lipschitz continuous in $\mathcal{R}$, according to Theorem 6.4 in \cite{haraux2017simple}, we only need to verify that the domain $\mathcal{R}$ is a contraction set for $\boldsymbol{F}$, that is, for all $i=1,\dots, 5$
	\begin{equation}\label{eq:quasi-pos}
		F_i([u_1,\dots, u_{i-1},0,u_{i+1},\dots, u_5])\geq 0, \quad \forall (u_j)_{j \neq i}\in \prod_{j\neq i}[0,\beta_j],       
	\end{equation}
	and,
	\begin{equation}\label{eq:quasi-bound}
		F_i([u_1,\dots, u_{i-1},\beta_i,u_{i+1},\dots, u_5])\leq 0, \quad \forall (u_j)_{j \neq i}\in \prod_{j\neq i}[0,\beta_j].     
	\end{equation}
	The property \eqref{eq:quasi-pos} follows directly from the definition of $F_i$. Let us focus on showing \eqref{eq:quasi-bound}. We notice the following relations 
	\begin{equation}\label{eq:rel-beta}
		\beta_5 = \frac{\tau_1 }{\tau_2 \tau_3}\beta_4, \, \beta_3 = \frac{\tau_{S}}{d} \beta_5, \, \beta_1 = \frac{r_1}{\tau_0+\gamma_{\min}}\beta_3^2, \, \beta_2 = \frac{\gamma_{\max}}{\tau_p }\beta_1.
	\end{equation}
	Then, by using \eqref{eq:rel-beta}, the fact that $\gamma_{\min}\leq \gamma(s)\leq \gamma_{\max}$, for all $s\geq 0$, and the assumption $\hat{m}\geq \frac{\lambda_M}{\sigma}$ it follows
	\begin{align*}
		F_1([\beta_1,u_2,u_3,u_4, u_5]) &= r_1 u_3^2-\gamma(u_4) \beta_1-\tau_0 \beta_1\leq  r_1 \beta_3^2-(\gamma_{\min}+\tau_0) \beta_1 = 0,\\
		F_2([u_1,\beta_2,u_3,u_4, u_5]) &= \gamma(u_4) u_1 -\tau_p \beta_2\leq \gamma_{\max} \beta_1-\tau_p \beta_2 = 0,\\
		F_3([u_1,u_2,\beta_3,u_4, u_5]) &= \dfrac{\tau_S}{1+C u_1^{\nu}}u_5-d \beta_3-r_2u_1 \beta_3-r_1 \beta_3^2\leq \tau_S \beta_5-d \beta_3 =  0,\\
		F_4([u_1,u_2,u_3,\beta_4, u_5]) &= -\sigma \beta_4+\lambda_M = -\sigma \hat{m}+\lambda_M<0,\\
		F_5([u_1,u_2,u_3,u_4, \beta_5]) &= \dfrac{\tau_1 u_1}{1+\tau_2 u_1}u_4-\tau_3 \beta_5 \leq  \dfrac{\tau_1}{\tau_2} \beta_4-\tau_3 \beta_5=0 .
	\end{align*}
	This concludes the proof of the proposition.
\end{proof}
To analyze more dynamical aspects of the model \eqref{eq:modelhomo}, we follow \cite{ciuperca2024qualitative} and assume that the recruitment rate of oligomers to amyloid plaques $\gamma$ is constant, i.e. we set $\gamma(u)=\gamma_0$, for all $u\geq 0$. This essentially corresponds to assuming an average rate at which oligomers are recruited, while considering that they possess a highly stable structure and that their degradation is negligible, implying that $\tau_0=0$. So we get the simplified spatially homogeneous model
\begin{equation}\label{eq:modelhomo}
	\begin{aligned}
		\frac{du_1}{dt}  &=r_1 u_3^2-\gamma_0 u_1,\\
		\frac{du_2}{dt}&= \gamma_0 u_1 -\tau_p u_2,\\
		\frac{du_3}{dt}  & =\dfrac{\tau_S}{1+C u_1^{\nu}}u_5-du_3-r_2 u_1 u_3-r_1u_3^2,\\
		\frac{du_4}{dt}  &=\dfrac{\alpha_1 u_1}{1+\alpha_2 u_1}(\hat{m}-u_4)u_4-\sigma u_4+\lambda_M,\\
		\frac{du_5}{dt}  & =\dfrac{\tau_1 u_1}{1+\tau_2 u_1}u_4-\tau_3 u_5.
	\end{aligned}
\end{equation}
In Theorem 3.3 of \cite{ciuperca2024qualitative}, the authors also investigate the existence of positive equilibrium points. We summarize their result below in the form of a proposition.
\begin{prop}\label{prop:equi}
	The system \eqref{eq:modelhomo} has a disease free equilibrium point $\mathcal{E}_0=(0,0,0,\lambda_M/\sigma,0)$ and if the parameters satisfy the condition $\sigma \gamma_0 \tau_3<\tau_1\tau_S\lambda_M$, then for $d > 0$ small enough, there exist at least two positive equilibria of the system. If $d > 0$ is large enough, then there are no positive solutions to the system.   
\end{prop}
On the other hand, in Proposition 3.2 they analyzed the local stability of the disease-free equilibrium point $\mathcal{E}_0=(0,0,0,\lambda_M/\sigma,0)$, this result reads
\begin{prop}\label{prop:sta-dis-free}
	For the system \eqref{eq:modelhomo}, the disease-free equilibrium $\mathcal{E}_0$ is locally asymptotically stable for every choice of positive parameters.
\end{prop}
In Section \ref{sec:dc}, we will propose a discrete scheme that preserves the invariance of rectangle $\mathcal{R}$, the equilibrium points of the continuous system and also maintains the local stability conditions of the disease-free equilibrium $\mathcal{E}_0$.
\subsection{Admissible weak solutions}
We observe that the system \eqref{eq:model}-\eqref{eq:RHS} can be written as an abstract semilinear problem of the form $\boldsymbol{u}_t+\mathcal{L}(\boldsymbol{u}) = \boldsymbol{F}(\boldsymbol{u})$, where $\mathcal{L}$ is the operator of spatial derivatives. According to Proposition \ref{prop:inv-R}, the rectangle $\mathcal{R}$ is invariant with respect to the associated system of Ordinary Differential Equations  $\boldsymbol{u}_t = \boldsymbol{F}(\boldsymbol{u})$. Following the approach of \cite{coudiere2013analysis}, we define the following weak admissible solutions to the system \eqref{eq:model} with boundary conditions \eqref{eq:bound-cond} and initial conditions \eqref{eq:ini-cond}. 
\begin{definition}\label{def:weaksol}
	Given functions $\boldsymbol{u}_0 = (u_{1,0},\dots, u_{5,0})^{\mathrm{T}}$ defined a.e. in $\Omega$ such that $\boldsymbol{u}_0\in \mathcal{R}$ a.e. in $\Omega$, $\lambda_M\in L^{\infty}(\Omega)$, and $\hat{m}\geq \frac{\|\lambda_M\|_{L^{\infty}(\Omega)}}{\sigma}$, for all $T>0$ we define a \textit{weak solution} of \eqref{eq:model} as a set of functions $\boldsymbol{u} = (u_1,\dots,u_5)^{\mathrm{T}}$ defined a.e. in $\Omega_{T}:=\Omega \times (0,T)$, such that $\boldsymbol{u}\in \mathcal{R}$ a.e. in $\Omega_{T}$, $\boldsymbol{u}_i\in L^2(0,T;H^1(\Omega))$, $i=1,3,4,5$ and for any test functions $\psi_i\in \mathcal{D}(\overline{\Omega}\times [0,T))$, $i=1,\dots,5$ the function $\boldsymbol{u}$ satisfies the following identities:
	\begin{align}
		-\iint_{\Omega_T} u_1 \partial_t \psi_1 \, \mathrm{d} \boldsymbol{x} \mathrm{d} t &+\iint_{\Omega_T} d_1 \nabla u_1\cdot \nabla \psi_1 \, \mathrm{d} \boldsymbol{x} \mathrm{d} t \label{eq:weak-u1}\\
		&= \iint_{\Omega_T} F_1(\boldsymbol{u})\cdot \psi_1 \, \mathrm{d}\boldsymbol{x} \mathrm{d}t+\int_{\Omega} u_{1,0}(\boldsymbol{x}) \psi_1(\boldsymbol{x},0) \, \mathrm{d}\boldsymbol{x},\notag \\
		-\iint_{\Omega_T} u_2 \partial_t \psi_2 \, \mathrm{d} \boldsymbol{x} \mathrm{d} t &=\iint_{\Omega_T} F_2(\boldsymbol{u})\cdot \psi_2 \, \mathrm{d}\boldsymbol{x} \mathrm{d}t+\int_{\Omega} u_{2,0}(\boldsymbol{x}) \psi_2(\boldsymbol{x},0) \, \mathrm{d}\boldsymbol{x},\label{eq:weak-u2}\\
		-\iint_{\Omega_T} u_3 \partial_t \psi_3 \, \mathrm{d} \boldsymbol{x} \mathrm{d} t &+\iint_{\Omega_T} d_3 \nabla u_3\cdot \nabla \psi_3 \, \mathrm{d} \boldsymbol{x} \mathrm{d} t \label{eq:weak-u3}\\
		&= \iint_{\Omega_T} F_3(\boldsymbol{u})\cdot \psi_3 \, \mathrm{d}\boldsymbol{x} \mathrm{d}t+\int_{\Omega} u_{3,0}(\boldsymbol{x}) \psi_3(\boldsymbol{x},0) \, \mathrm{d}\boldsymbol{x},\notag\\
		-\iint_{\Omega_T} u_4 \partial_t \psi_4\, \mathrm{d} \boldsymbol{x} \mathrm{d} t&+\iint_{\Omega_T} (d_4 \nabla u_4\cdot \nabla \psi_4- \chi(u_4) \nabla u_1) \cdot \nabla \psi_4 \, \mathrm{d} \boldsymbol{x} \mathrm{d} t \label{eq:weak-u4}\\
		&=\iint_{\Omega_T} F_4(\boldsymbol{u})\cdot \psi_4 \, \mathrm{d}\boldsymbol{x} \mathrm{d}t+\int_{\Omega} u_{4,0}(\boldsymbol{x}) \psi_4(\boldsymbol{x},0) \, \mathrm{d}\boldsymbol{x},\notag\\
		-\iint_{\Omega_T} u_5 \partial_t \psi_5\, \mathrm{d} \boldsymbol{x} \mathrm{d} t &+\iint_{\Omega_T} d_5 \nabla u_5\cdot \nabla \psi_5 \, \mathrm{d} \boldsymbol{x} \mathrm{d} t \label{eq:weak-u5}\\
		&= \iint_{\Omega_T} F_5(\boldsymbol{u})\cdot \psi_5 \, \mathrm{d}\boldsymbol{x} \mathrm{d}t+\int_{\Omega} u_{5,0}(\boldsymbol{x}) \psi_5(\boldsymbol{x},0) \, \mathrm{d}\boldsymbol{x}.\notag
	\end{align}
\end{definition}
In Section 4, we demonstrate the existence of a weak solution of \eqref{eq:model} according with Definition \ref{def:weaksol}.
\section{Finite Volume Discretization}\label{sec:fvm}
This section is devoted to constructing approximate solutions of problem (1.1). To this end, we introduce the notion of an admissible finite volume mesh (cf. \cite{EYMARD2000713}).
\subsection{Space and time discretization meshes}\label{sec:spacetimedisc}
Let us begin by describing the discretization of the spatial domain. For this purpose, we assume that the domain $\Omega \subset \mathbb{R}^2$ is polygonal. 
\begin{definition}\label{def:am}
	An admissible mesh for $\Omega$ can be defined as a triplet $(\mathcal{T},\mathcal{E},\mathcal{P})$, where $\mathcal{T}$ is a finite collection of non-overlapping, bounded, convex polygonal subsets of $\Omega$ called control volumes, $\mathcal{E}$ is a finite collection of subsets of $\overline{\Omega}$, each contained in a hyperplane of $\mathbb{R}^2$, having strictly positive one-dimensional measure, and referred to as the edges of the control volumes, and $\mathcal{P} = (\boldsymbol{x}_K)_{K \in \mathcal{T}}$ is a finite collection of points in $\Omega$, known as the centers of the control volumes. This definition is employed to describe the discretization of the domain within the framework of the finite volume method. The triplet $(\mathcal{T},\mathcal{E},\mathcal{P})$ has the following properties:
	\begin{enumerate}
		\item The union of the closures of all control volumes is $\overline{\Omega}$, i.e. $\overline{\Omega} = \bigcup_{K\in \mathcal{T}} \overline{K}$.
		\item For any $K\in \mathcal{T}$, there exists $\mathcal{E}_K\subset \mathcal{E}$ such that $\partial K = \overline{K}-K = \bigcup_{\sigma \mathcal{E}_K} \overline{\sigma}$. Moreover, $\mathcal{E} = \bigcup_{K\in \mathcal{T}}\mathcal{E}_K$.
		\item For any $(K,L)\in \mathcal{T}^2$ with $K\neq L$, either the length of $\overline{K}\cap \overline{L}$ is $0$ or $\overline{K}\cap \overline{L} = \overline{\sigma}$, for some $\sigma \in \mathcal{E}$, which will be denoted by $\sigma_{K|L}$.
		\item The family $\mathcal{P}=(\boldsymbol{x}_K)_{K\in \mathcal{T}}$ is such that $\boldsymbol{x}_K\in \overline{K}$, for all $K\in \mathcal{T}$. In addition, if $\sigma = \sigma_{K|L}$, we assume that $\boldsymbol{x}_K\neq \boldsymbol{x}_L$ and the straight line $\mathcal{D}_{K,L}$ going through $\boldsymbol{x}_K$ and $\boldsymbol{x}_L$ is orthogonal to $K|L$.
		\item For any $\sigma \in \mathcal{E}$ with $\sigma \subset \partial \Omega$, denote by $K$ the control volume such that $\sigma \in \mathcal{E}_K$. If $\boldsymbol{x}_K \notin \sigma$, let $\mathcal{D}_{K,\sigma}$ denote the line through $\boldsymbol{x}_K$ orthogonal to $\sigma$. We assume $\mathcal{D}_{K,\sigma} \neq \emptyset$ and define $\boldsymbol{y}_{\sigma} = \mathcal{D}_{K,\sigma} \cap \sigma$.
	\end{enumerate}
\end{definition}
We introduce the following notations. The size of the mesh $\mathcal{T}$ is defined by 
\begin{equation*}
	\mathrm{size}(\mathcal{T}) = \max_{K\in \mathcal{T}} \mathrm{diam}(K),
\end{equation*}
where $\mathrm{diam}(K) = \max_{\boldsymbol{x},\boldsymbol{y}\in K}\|\boldsymbol{x}-\boldsymbol{y}\|_2$. We denote by $\mathcal{N}(K)$ the set of neighbors of the control volume $K$, i.e., $\mathcal{N}(K) = \{L\in \mathcal{T}: \overline{K}\cap \overline{L} = \overline{\sigma}, \text{ for some }\sigma\in \mathcal{E}\}$; a generic neighbor of $K$ is denoted by $L$. Moreover, we denote by $\boldsymbol{n}_{K|L}$ and $d_{K|L}$ the unit normal vector to $\sigma_{K|L}$ outward from $K$ and the distance $\|\boldsymbol{x}_K-\boldsymbol{x}_L\|_2$, respectively. For any $K\in \mathcal{T}$ and $\sigma\in \mathcal{E}$, we denote by $m(K)$ the $2-$dimensional Lebesgue measure of $K$. If $L\in \mathcal{N}(K)$, then $m(\sigma_{K|L})$ will denote the $1-$dimensional measure of the edge $\sigma_{K|L}$, and finally, the transmissibility through $\sigma_{K|L}$ is defined as $\tau_{K|L}:=\frac{m(\sigma_{K|L})}{d_{K|L}}$.
\begin{figure}[H]
	\centering
	\includegraphics[width=0.4\linewidth]{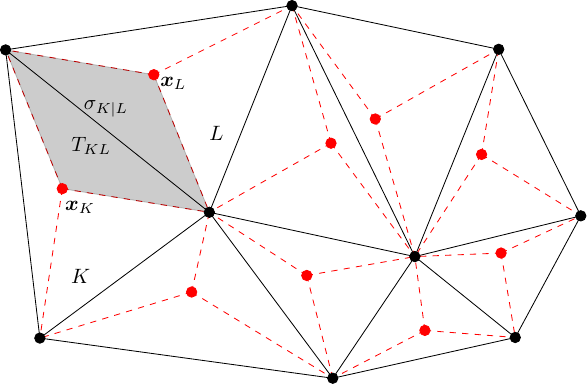}
	\caption{Control volumes, centers and diamonds (in dashed lines).}
	\label{fig:mesh}
\end{figure}
The discrete unknowns are piecewise constant functions on the control volumes $K\in \mathcal{T}$, so we introduce the Hilbert space
\begin{equation}
	H_{\mathcal{T}}(\Omega)=\left\{\phi\in L^2(\Omega):\,  \phi|_{K}\in\mathbb{P}_0(K),\quad  \forall K\in \mathcal{T}\right\}.
\end{equation}
So, every function $u_{\mathcal{T}}\in H_{\mathcal{T}}(\Omega)$ is characterized by its numerical values $(u_K)_{K\in \mathcal{T}}$ such that $u_{\mathcal{T}}|_{K} = u_{K}$ for every $K\in \mathcal{T}$. It is clear that $H_{\mathcal{T}}(\Omega)\subset L^2(\Omega)$ and the usual inner scalar product and norm become
\begin{equation*}
	\big( u_{\mathcal{T}},v_{\mathcal{T}}\big)_{L^2(\Omega)} = \sum_{K\in \mathcal{T}}m(K) u_K v_K, \quad \|u_{\mathcal{T}}\|_{L^2(\Omega)} = \left( \sum_{K\in \mathcal{T}} m(K)|u_K|^2\right)^{1/2}, \quad \forall\, u_{\mathcal{T}},v_{\mathcal{T}}\in H_{\mathcal{T}}(\Omega).
\end{equation*}
We also introduce the discrete semi-norm
\begin{equation}
	|u_{\mathcal{T}}|_{1,\mathcal{T}}^2:= \frac{1}{2}\sum_{K\in \mathcal{T}} \sum_{L\in \mathcal{N}(K)} \tau_{K|L}|u_{L}-u_K|^2,  \quad \forall\, u_{\mathcal{T}} \in H_{\mathcal{T}}(\Omega).
\end{equation}
Moreover, for an interface $\sigma_{K|L}\in \mathcal{E}$ of a control volume $K$, we consider the ``diamond'' $T_{KL}$ constructed by joining the diamond centers $\boldsymbol{x}_K$ and $\boldsymbol{x}_L$ with the extremes of the interval (see Figure \ref{fig:mesh}). The discrete gradient operator $\nabla_{\mathcal{T}} $ of a function $u_{\mathcal{T}}\in H_{\mathcal{T}}(\Omega)$ is a piecewise constant function over the diamonds $T_{KL}$ such that 
\begin{equation}\label{eq:grad-disc}
	\nabla_{\mathcal{T}} u_{\mathcal{T}}|_{T_{KL}} :=
	\begin{cases}
		2\dfrac{u_L-u_K}{d_{K|L}}\boldsymbol{n}_{K|L},& \text{ for }\sigma = \sigma_{K|L} \in \mathcal{E}\setminus \partial \Omega, \\
		0,& \text{ for }\sigma \in \partial K.
	\end{cases}
\end{equation}
For the time discretization of the interval $[0,T]$, we chose a time-step $\Delta t>0$ and $N>0$ such that $(N+1)\Delta t\geq T$, and set $t_n = n \Delta t $ for $n\in \{0,1,\dots, N\}$.  We assume that $\Delta t$ satisfies the following mild restriction
\begin{equation}\label{eq:mild}
	\Delta t <\min\left\{\dfrac{1}{\gamma_{\max}},\dfrac{1}{r_1\beta_3+\tau_S},\dfrac{\alpha_2}{\alpha_2+\hat{m}\alpha_1},\dfrac{\tau_2}{\tau_1}\right\},
\end{equation}
which will be used to prove the existence of a solution to the scheme. We define the size of the space–time discretization by $h:=\max( \mathrm{size}(\mathcal{T}),\Delta t )$. Finally, the complete discrete solution is a time-sequence of piecewise constant functions $(u_{\mathcal{T}}^{n})_{n\in \mathbb{N}}$ such that $u_{\mathcal{T}}^{n}\in H_{\mathcal{T}}(\Omega)$, for all $n$. In addition, we associate a time and space piecewise constant function to this sequence denoted by $u_{h}$ and defined by
\begin{equation}\label{eq:piece-cons}
	u_{h}(\boldsymbol{x},t)=u_K^{n+1}, \quad \text{for a.e. }\quad  (\boldsymbol{x},t)\in K \times (t_n,t_{n+1}),\quad \forall K\in \mathcal{T},\, \forall n\in \{0,1,\dots,N\}.
\end{equation}
\subsection{The finite volume scheme}\label{sec:fv-scheme}
The Finite Volume scheme for \eqref{eq:model} is stated as follows: for all $K\in \mathcal{T}$ and $n\in \{0,1,\dots,N\}$, find $\{u_{i,K}^{n+1}\}_{K\in \mathcal{T}}$, $i=1,\dots,5$ such that
\begin{align}
	m(K)\dfrac{u_{1,K}^{n+1}-u_{1,K}^n}{\Delta t}&-d_1 \sum_{L\in \mathcal{N}(K)}\tau_{K|L} (u_{1,L}^{n+1}-u_{1,K}^{n+1})= m(K) \bar{F}_{1}(\boldsymbol{u}_{K}^{n},\boldsymbol{u}_{K}^{n+1}), \ \label{eq:fvm-1}\\
	\dfrac{u_{2,K}^{n+1}-u_{2,K}^n}{\Delta t}&= \bar{F}_{2}(\boldsymbol{u}_{K}^{n},\boldsymbol{u}_{K}^{n+1}), \label{eq:fvm-2}\\
	m(K)\dfrac{u_{3,K}^{n+1}-u_{3,K}^n}{\Delta t}&-d_3 \sum_{L\in \mathcal{N}(K)}\tau_{K|L} (u_{3,L}^{n+1}-u_{3,K}^{n+1})= m(K) \bar{F}_{3}(\boldsymbol{u}_{K}^{n},\boldsymbol{u}_{K}^{n+1}), \label{eq:fvm-3}\\
	m(K)\dfrac{u_{4,K}^{n+1}-u_{4,K}^n}{\Delta t}&-d_4 \sum_{L\in \mathcal{N}(K)}\tau_{K|L} (u_{4,L}^{n+1}-u_{4,K}^{n+1}) \label{eq:fvm-4}\\
	&+ \sum_{L\in \mathcal{N}(K)} \tau_{K|L} \mathcal{G}\big(u_{4,K}^{n+1},u_{4,L}^{n+1};\delta u_{1,KL}^{n+1}\big) = m(K) \bar{F}_{4}(\boldsymbol{u}_{K}^{n},\boldsymbol{u}_{K}^{n+1}),\notag \\
	m(K)\dfrac{u_{5,K}^{n+1}-u_{5,K}^n}{\Delta t}&-d_5 \sum_{L\in \mathcal{N}(K)}\tau_{K|L} (u_{5,L}^{n+1}-u_{5,K}^{n+1})= m(K) \bar{F}_{5}(\boldsymbol{u}_{K}^{n},\boldsymbol{u}_{K}^{n+1}),\label{eq:fvm-5}
\end{align}
where the right-hand side functions are approximated by a semi-implicit approach in the sense that right-hand side contributions depend on the numerical solution at times $t=t_n$ and $t=t_{n+1}$, i.e., we define:
\begin{equation}\label{eq:rhs-disc}
	\begin{aligned}
		\bar{F}_{1}(\boldsymbol{u}_{K}^{n},\boldsymbol{u}_{K}^{n+1}) &= r_1 (u_{3,K}^n)^2-\gamma(u_{4,K}^n) u_{1,K}^{n+1}-\tau_0 u_{1,K}^{n+1}, \\
		\bar{F}_{2}(\boldsymbol{u}_{K}^{n},\boldsymbol{u}_{K}^{n+1}) &=\gamma(u_{4,K}^n) u_{1,K}^{n}-\tau_p u_{2,K}^{n+1}\\
		\bar{F}_{3}(\boldsymbol{u}_{K}^{n},\boldsymbol{u}_{K}^{n+1}) &=\dfrac{\tau_S}{1+C (u_{1,K}^n)^{\nu}}u_{5,K}^n-du_{3,K}^{n+1}-r_2u_{1,K}^nu_{3,K}^{n+1}-r_1 u_{3,K}^{n}u_{3,K}^{n+1},\\
		\bar{F}_{4}(\boldsymbol{u}_{K}^{n},\boldsymbol{u}_{K}^{n+1}) &=\dfrac{\alpha_1 u_{1,K}^n}{1+\alpha_2 u_{1,K}^n}(\hat{m}-u_{4,K}^{n+1})u_{4,K}^{n}-\sigma u_{4,K}^{n+1}+\lambda_{M,K},\\
		\bar{F}_{5}(\boldsymbol{u}_{K}^{n},\boldsymbol{u}_{K}^{n+1}) &= \dfrac{\tau_1 u_{1,K}^n}{1+\tau_2 u_{1,K}^n}u_{4,K}^n-\tau_3 u_{5,K}^{n+1},
	\end{aligned}
\end{equation}
where $\lambda_{M,K}$ denotes the average of the function $\lambda_M$ over the volume $K$. This discretization is selected to ensure the dynamic consistency of the scheme and coincides with a traditional NSFD approximation of the SH model. Further details are provided in Section \ref{sec:dc}. 

As usual, homogeneous Neumann boundary conditions \eqref{eq:bound-cond} are incorporated implicitly. Specifically, the portions of $\partial K$ that lie in $\partial \Omega$ do not contribute to the sums over $L\in \mathcal{N}(K)$ terms, which effectively enforces a zero–flux condition along the outer edges of the mesh. Here $\delta u_{1,KL}^{n+1}:=u_{1,L}^{n+1}-u_{1,K}^{n+1}$ and $\mathcal{G}$ are numerical flux functions satisfying the following properties: 
\begin{enumerate}
	\item[(H1)]  $\mathcal{G}(\cdot, b;c)$ is nondecreasing for all $a,b\in \mathbb{R}$, and $\mathcal{G}(a, \cdot;c)$ is nonincreasing for all $a,c\in \mathbb{R}$.
	\item[(H2)] $\mathcal{G}(a,b;c) = - \mathcal{G}(b,a;-c)$ for all $a,b,c\in \mathbb{R}$.
	\item[(H3)] $\mathcal{G}(a,a;c)=\chi(a)c$ for all $a,c\in \mathbb{R}$-
	\item[(H4)] There exists a constant $L_{\mathcal{G}}>0$ such that 
	\begin{equation*}
		\big| \mathcal{G}(a,b;c)-\mathcal{G}(\tilde{a},\tilde{b};c)\big|\leq L_{\mathcal{G}}|c|\big(\big|a-\tilde{a}\big|+\big|b-\tilde{b}\big|\big), \quad \forall a,\tilde{a},b,\tilde{b},c\in \mathbb{R}.
	\end{equation*}
\end{enumerate}
In this work, we follow \cite{EYMARD2000713,andreianov2011finite} to construct the numerical flux satisfying assumptions (H1)-(H4). We split the sensitivity function $\chi$ into the non-decreasing part $\chi_{\uparrow}$ and the non-increasing part $\chi_{\downarrow}$, which are given by
\begin{equation*}
	\chi_{\uparrow}(z):= \int_0^{z} \big(\chi'(s)\big)^{+} \, \mathrm{d}s, \quad \chi_{\downarrow}(z):= -\int_0^{z}\big(\chi'(s)\big)^{-} \, \mathrm{d}s.
\end{equation*}
Herein, $s^{+} = \max(s,0)$ and $s^{-} = \max(-s,0)$. Then we take
\begin{equation}
	\mathcal{G}(a,b;c):=c^{+}\big(\chi_{\uparrow}(a)+\chi_{\downarrow}(b)\Big)-c^{-} \big(\chi_{\uparrow}(b)+\chi_{\downarrow}(a)\Big).
\end{equation}
\begin{remark}
	For $\chi(s)=\alpha s(\hat{m}-s),$ the functions $\chi_{\uparrow}$ and $\chi_{\downarrow}$ are given by $\chi_{\uparrow}(a)=\chi\big(\min \big(a,\frac{\hat{m}}{2}\big)\big)$ and $\chi_{\downarrow}(a)=\chi\big(\max \big(a,\frac{\hat{m}}{2}\big)\big)-\chi\big(\frac{\hat{m}}{2}\big)$.
\end{remark}
The discrete problem is composed from the equations \eqref{eq:rhs-disc}-\eqref{eq:fvm-5} and the following discretization of the initial data:
\begin{equation}
	u_{i,K}^{0} = \dfrac{1}{m(K)} \int_K u_{i,0}(\boldsymbol{x}) \, \mathrm{d}\boldsymbol{x}, \quad \forall K\in \mathcal{T},\quad i=1,\dots,5.
\end{equation}

We associate a discrete solution of the scheme at $t = t_{n+1}$ with the vector $\boldsymbol{u}_{h}^{n+1} = (u_{1,h}^{n+1},\dots,u_{5,h}^{n+1})^{\mathrm{T}}$ of the piecewise constant on $\Omega$ functions defined as in \eqref{eq:piece-cons}.
\subsection{Existence of a solution for the finite volume scheme}\label{sec:exist}

In order to prove the existence of an admissible solution of the discrete problem \eqref{eq:fvm-1}-\eqref{eq:fvm-5}, we introduce the following truncated version of \eqref{eq:fvm-4}:
\begin{equation}\label{eq:fvm-trunc}
	\begin{aligned}
		m(K)\dfrac{u_{4,K}^{n+1}-u_{4,K}^n}{\Delta t}&-d_4 \sum_{L\in \mathcal{N}(K)}\tau_{K|L} (u_{4,L}^{n+1}-u_{4,K}^{n+1})\\
		&+ \sum_{L\in \mathcal{N}(K)} \tau_{K|L} \mathcal{\tilde{G}}\big(u_{4,K}^{n+1},u_{4,L}^{n+1};\delta \tilde{u}_{1,KL}^{n+1}\big) = m(K) \bar{F}_{4}(\boldsymbol{u}_{K}^{n},\boldsymbol{u}_{K}^{n+1}),
	\end{aligned}
\end{equation}
with $\tilde{u}_{1,KL}^{n+1}:=\tilde{u}_{1,L}^{n+1}-\tilde{u}_{1,K}^{n+1}$ and $\tilde{u}_i = Z_{[0,\beta_i]}(u_i)$, for $i=1,\dots,5$, where we use the general truncation function $Z_{[a,b]}(r)=\min(b,\max(a,r))$. Moreover, the function $\tilde{\mathcal{G}}$ is obtained from $\mathcal{G}$ by replacing $\chi$ with $\chi \circ Z_{[0,\beta_4]}$, i.e.
\begin{equation*}
	\tilde{\mathcal{G}}(a,b;c) = c^{+}\big(\tilde{\chi}_{\uparrow}(a)+\tilde{\chi}_{\downarrow}(b)\big)-c^{-}\big(\tilde{\chi}_{\uparrow}(b)+\tilde{\chi}_{\downarrow}(a)\big), 
\end{equation*}
where $\tilde{\chi}(s)=\chi(Z_{[0,\beta_4]}(s)) = \chi(s)$ if $0\leq s\leq \beta_4$ and $0$ otherwise. As a result, the function $\tilde{G}$ fulfills conditions (H1)–(H4). In addition, for every $a \in \mathbb{R}$, it holds that
\begin{equation*}
	0\leq \tilde{\chi}_{\uparrow}(a) = \int_0^a \tilde{\chi}'(s)^{+}\, \mathrm{d}s \leq  \int_0^{\beta_4} \tilde{\chi}'(s)^{+}\, \mathrm{d}s = \int_0^a \chi'(s)^{+}\, \mathrm{d}s, \quad 0\leq -\tilde{\chi}_{\downarrow}(a) \leq \int_0^{\beta_4} \chi'(s)^{-}\, \mathrm{d}s.
\end{equation*}
Thus, we have shown the following property
\begin{equation}\label{eq:flux-h5}
	\forall(a,b,c)\in \mathbb{R}^3, \quad \big| \tilde{\mathcal{G}}(a,b;c)\big|\leq |c| \int_0^{\beta_4} \big| \chi'(s)\big| \, \mathrm{d}s.
\end{equation}
Easily, we have that $\tilde{\mathcal{G}}(a,b;c) = \mathcal{G}(a,b;c)$ whenever $a,b \in [0,\beta_4]$, for any $c\in \mathbb{R}$. From now on, we refer to the system given by \eqref{eq:fvm-1}, \eqref{eq:fvm-2}, \eqref{eq:fvm-3}, \eqref{eq:fvm-trunc}, \eqref{eq:fvm-5}, and \eqref{eq:rhs-disc} as the \textit{truncated discrete system.}

\begin{lemma}[Maximum principle]\label{lemma:maxprin}
	Let $\boldsymbol{u}^{0}\in \mathcal{R}$ and $(\boldsymbol{u}_K^{n+1})_{K\in \mathcal{T},n\in \{0,\dots,N\}}$ be any solution of the truncated discrete system. Then $\boldsymbol{u}_K^{n}\in \mathcal{R}$, for all $K\in \mathcal{T}$, for every $n\geq 1$.
\end{lemma}
\begin{proof}
	We proceed by induction on $n$. By hypothesis, we have that $\boldsymbol{u}_K^{0} = \boldsymbol{u}^{0}\in \mathcal{R}$, for all $K\in \mathcal{T}$. So, let us assume that $\boldsymbol{u}_K^{n}\in \mathcal{R}$, for all $K\in \mathcal{T}$. We focus first on \eqref{eq:fvm-trunc}. We notice that $\bar{F}_{4}(\boldsymbol{u}_{K}^{n},\boldsymbol{u}_{K}^{n+1}) = \Psi_4(\boldsymbol{u}_K^{n}) - \Phi_4(\boldsymbol{u}_K^{n}) u_{4,K}^{n+1}$, where $\Psi_4(\boldsymbol{u}_K^{n}) = \frac{\alpha_1 \hat{m} u_{1,K}^n}{1+\alpha_2 u_{1,K}^n} u_{4,K}^{n}+\lambda_M$ and $\Phi_4(\boldsymbol{u}_K^{n}) = \frac{\alpha_1 u_{1,K}^n}{1+\alpha_2 u_{1,K}^n} u_{4,K}^{n}+\sigma$. So, for all $K\in \mathcal{T}$ we have that 
	\begin{equation}\label{eq:maxprin-1}
		\begin{aligned}
			u_{4,K}^{n+1}-\left(\dfrac{u_{4,K}^n+\Delta t \Psi_4(\boldsymbol{u}_K^{n})}{1+\Delta t \Phi_4(\boldsymbol{u}_K^{n})}\right)&=\dfrac{d_4 \Delta t}{m(K)[1+\Delta t \Phi_4(\boldsymbol{u}_K^{n})]} \sum_{L\in \mathcal{N}(K)}\tau_{K|L} (u_{4,L}^{n+1}-u_{4,K}^{n+1})\\
			&\quad-\dfrac{1}{1+\Delta t \Phi_4(\boldsymbol{u}_K^{n})}\sum_{L\in \mathcal{N}(K)} \tau_{K|L} \mathcal{\tilde{G}}\big(u_{4,K}^{n+1},u_{4,L}^{n+1};\delta \tilde{u}_{1,KL}^{n+1}\big). 
		\end{aligned}
	\end{equation}
	We first show that $u_{4,K}^{n+1}\geq 0$, for all $K\in \mathcal{T}$. To do so, let us fix $K\in \mathcal{T}$ such that $u_{4,K}^{n+1}:= \min\{u_{4,L}^{n+1}: L\in \mathcal{T} \}$. Multiplying \eqref{eq:maxprin-1} by $(u_{4,K}^{n+1})^{-}$, we deduce
	\begin{equation}
		\begin{aligned}
			\left(u_{4,K}^{n+1}-\dfrac{u_{4,K}^n+\Delta t \Psi_4(\boldsymbol{u}_K^{n})}{1+\Delta t \Phi_4(\boldsymbol{u}_K^{n})}\right)(u_{4,K}^{n+1})^{-}=S_1+S_2,
		\end{aligned}
	\end{equation}
	where we define:
	\begin{align*}
		S_1&:=\dfrac{d_4 \Delta t}{m(K)[1+\Delta t \Phi_4(\boldsymbol{u}_K^{n})]} \sum_{L\in \mathcal{N}(K)}\tau_{K|L} (u_{4,L}^{n+1}-u_{4,K}^{n+1})(u_{4,K}^{n+1})^{-},\\
		S_2&:=-\dfrac{1}{1+\Delta t \Phi_4(\boldsymbol{u}_K^{n})}\sum_{L\in \mathcal{N}(K)} \tau_{K|L} \mathcal{\tilde{G}}\big(u_{4,K}^{n+1},u_{4,L}^{n+1};\delta \tilde{u}_{1,KL}^{n+1}\big)(u_{4,K}^{n+1})^{-}.
	\end{align*}
	By the choice of $K$ we know that $(u_{4,L}^{n+1}-u_{4,K}^{n+1})(u_{4,K}^{n+1})^{-}\geq 0$ and $\Phi_4(\boldsymbol{u}_K^{n})\geq 0$ because $\boldsymbol{u}_K^{n}\in \mathcal{R}$. So we have that $S_1\geq 0$. Due to the properties (H1)-(H4) of $\mathcal{\tilde{G}}$ and the extension by zero of the continuous function $\chi$ we obtain that
	\begin{align*}
		\mathcal{\tilde{G}}\big(u_{4,K}^{n+1},u_{4,L}^{n+1};\delta \tilde{u}_{1,KL}^{n+1}\big)\leq \mathcal{\tilde{G}}\big(u_{4,K}^{n+1},u_{4,K}^{n+1};\delta \tilde{u}_{1,KL}^{n+1}\big) = \chi(u_{4,K}^{n+1})\delta \tilde{u}_{1,KL}^{n+1}(u_{4,K}^{n+1})^{-} = 0, 
	\end{align*}
	because $\chi(u_{4,K}^{n+1}) = 0$ if $u_{4,K}^{n+1}\leq 0$ and $(u_{4,K}^{n+1})^{-} = 0$ otherwise. Thus, $S_2\geq 0$ and we get
	\begin{align*}
		0\leq \left(u_{4,K}^{n+1}-\dfrac{u_{4,K}^n+\Delta t \Psi_4(\boldsymbol{u}_K^{n})}{1+\Delta t \Phi_4(\boldsymbol{u}_K^{n})}\right)(u_{4,K}^{n+1})^{-}\leq u_{4,K}^{n+1}(u_{4,K}^{n+1})^{-} = -|(u_{4,K}^{n+1})^{-}|^2.
	\end{align*}
	From this, we deduce that $(u_{4,K}^{n+1})^{-}=0$, i.e. $u_{4,K}^{n+1}\geq 0$, for all $K\in \mathcal{T}$. For showing that $u_{4,K}^{n+1}\leq \beta_4$, for all $K\in \mathcal{T}$, let us choose $K\in \mathcal{T}$ such that $u_{4,K}^{n+1}:= \max\{u_{4,L}^{n+1}: L\in \mathcal{T} \}$. We multiply \eqref{eq:fvm-2} by $(u_{4,K}^{n+1}-\beta_4)^{+}$ to obtain that
	\begin{equation}
		\begin{aligned}
			\left(u_{4,K}^{n+1}-\dfrac{u_{4,K}^n+\Delta t \Psi_4(\boldsymbol{u}_K^{n})}{1+\Delta t \Phi_4(\boldsymbol{u}_K^{n})}\right)(u_{4,K}^{n+1}-\beta_4)^{+}=T_1+T_2,
		\end{aligned}
	\end{equation}
	where we define:
	\begin{align*}
		T_1&:=\dfrac{d_4 \Delta t}{m(K)[1+\Delta t \Phi_4(\boldsymbol{u}_K^{n})]} \sum_{L\in \mathcal{N}(K)}\tau_{K|L} (u_{4,L}^{n+1}-u_{4,K}^{n+1})(u_{4,K}^{n+1}-\beta_4)^{+},\\
		T_2&:=-\dfrac{1}{1+\Delta t \Phi_4(\boldsymbol{u}_K^{n})}\sum_{L\in \mathcal{N}(K)} \tau_{K|L} \mathcal{G}\big(u_{4,K}^{n+1},u_{4,L}^{n+1};\delta u_{1,KL}^{n+1}\big)(u_{4,K}^{n+1}-\beta_4)^{+}.
	\end{align*}
	By the choice of $K$ we know that $(u_{4,L}^{n+1}-u_{4,K}^{n+1})(u_{4,K}^{n+1}-\beta_4)^{+}\leq 0$ and $\Phi_4(\boldsymbol{u}_K^{n})\geq 0$ because $\boldsymbol{u}_K^{n}\in \mathcal{R}$. So it follows that $T_1\leq 0$. Due to the properties (H1)-(H4) of $\mathcal{\tilde{G}}$ and the extension by zero of the continuous function $\chi$ we obtain that
	\begin{align*}
		\mathcal{\tilde{G}}\big(u_{4,K}^{n+1},u_{4,L}^{n+1};\delta \tilde{u}_{1,KL}^{n+1}\big)\geq \mathcal{\tilde{G}}\big(u_{4,K}^{n+1},u_{4,K}^{n+1};\delta \tilde{u}_{1,KL}^{n+1}\big) = \chi(u_{4,K}^{n+1})\delta \tilde{u}_{1,KL}^{n+1}(u_{4,K}^{n+1})^{-} = 0, 
	\end{align*}
	because $\chi(u_{4,K}^{n+1}) = 0$ if $u_{4,K}^{n+1}\geq \beta_4$ and $(u_{4,K}^{n+1}-\beta_4)^{+} = 0$ otherwise. Therefore we have that $T_2\leq 0$. Moreover, by using the induction hypothesis $\boldsymbol{u}_K^{n}\in \mathcal{R}$ and the hypothesis $\hat{m}\geq \frac{\|\lambda_M\|_{L^{\infty}(\Omega)}}{\sigma}$ it follows that
	\begin{align*}
		\frac{u_{4,K}^n+\Delta t \Psi_4(\boldsymbol{u}_K^{n})}{1+\Delta t \Phi_4(\boldsymbol{u}_K^{n})}&\leq \dfrac{u_{4,K}^n+\dfrac{\alpha_1 u_{1,K}^n }{1+\alpha_2 u_{1,K}^n}\Delta t \hat{m} u_{4,K}^{n}+\Delta t\sigma \hat{m} }{1+\Delta t\left(\dfrac{\alpha_1 u_{1,K}^n}{1+\alpha_2 u_{1,K}^n} u_{4,K}^n+\sigma \right)} = \dfrac{\left(1+\dfrac{\alpha_1 u_{1,K}^n }{1+\alpha_2 u_{1,K}^n}\Delta t \hat{m} \right)u_{4,K}^n+\Delta t\sigma \hat{m} }{1+\Delta t\left(\dfrac{\alpha_1 u_{1,K}^n}{1+\alpha_2 u_{1,K}^n} u_{4,K}^n+\sigma \right)}\\
		&= f(u_{4,K}^n),
	\end{align*}
	where $f:[0,+\infty)\longrightarrow (0,+\infty)$ is given by,
	\begin{equation}\label{eq:funf}
		f(x):=\dfrac{\left(1+\dfrac{\alpha_1 u_{1,K}^n }{1+\alpha_2 u_{1,K}^n}\Delta t \hat{m} \right)x+\Delta t\sigma \hat{m} }{1+\Delta t\left(\dfrac{\alpha_1 u_{1,K}^n}{1+\alpha_2 u_{1,K}^n} x+\sigma \right)}.
	\end{equation}
	We notice that $f$ is differentiable in $[0,+\infty)$ and its derivative is given by
	{ \begin{align*}
			f'(x)&=\dfrac{1+\Delta t \left[\left(\dfrac{\alpha_1 u_{1,K}^n}{1+\alpha_2 u_{1,K}^n}\right)\hat{m}+\sigma\right]}{\left[1+\Delta t\left(\dfrac{\alpha_1 u_{1,K}^n}{1+\alpha_2 u_{1,K}^n} x+\sigma \right)\right]^2}.
	\end{align*}}
	So, $f'(x)>0$, for all $x\geq 0$. Then, $f$ is an increasing function over $[0,+\infty)$. Now, by the induction hypothesis we know that $u_{4}^{n}\leq \hat{m}$, then $\frac{u_{4,K}^n+\Delta t \Psi_4(\boldsymbol{u}_K^{n})}{1+\Delta t \Phi_4(\boldsymbol{u}_K^{n})}\leq f(u_{4}^n)\leq f(\hat{m})=\hat{m} = \beta_4$. In this way, we obtain that
	\begin{align*}
		0\geq \left(u_{4,K}^{n+1}-\dfrac{u_{4,K}^n+\Delta t \Psi_4(\boldsymbol{u}_K^{n})}{1+\Delta t \Phi_4(\boldsymbol{u}_K^{n})}\right)(u_{4,K}^{n+1}-\beta_4)^{+}\geq (u_{4,K}^{n+1}-\beta_4)(u_{4,K}^{n+1}-\beta_4)^{+}= |(u_{4,K}^{n+1}-\beta_4)^{+}|^2.
	\end{align*}
	From this, we deduce that $(u_{4,K}^{n+1}-\beta_4)^{+}=0$, i.e. $u_{4,K}^{n+1}\leq \beta_4$, for all $K\in \mathcal{T}$. In the same way, we can obtain the bounds $0\leq u_{i,K}^{n+1}\leq \beta_i,$ for all $K\in \mathcal{T}$ and $i=1,2,3,5$. This implies $\boldsymbol{u}_K^{n+1}\in \mathcal{R}$, for all $K\in \mathcal{T}$ and we conclude the proof of the lemma.
\end{proof}
\begin{lemma}[Existence of discrete truncated problem]\label{lemma:exist-truncated}
	Let $\mathcal{T}$ be an admissible discretization of $\Omega$ and $\boldsymbol{u}^{0}\in \mathcal{R}$. Then the truncated discrete problem \eqref{eq:fvm-1}-\eqref{eq:fvm-3}, \eqref{eq:fvm-trunc}, \eqref{eq:fvm-5}, \eqref{eq:rhs-disc} admits at least one solution $\{\boldsymbol{u}_{K}^{n+1}\}_{K\in \mathcal{T},n\in\{0,\dots,N\}}$. 
\end{lemma}
\begin{proof}
	Let us make the proof by induction on $n$, so we assume that $\{\boldsymbol{u}_{K}^{m}\}_{K\in \mathcal{T}}$ exists for $m=1,\dots,n$. Let us introduce the Hilbert space $X_h = [H_{\mathcal{T}}(\Omega)]^5$ of quintuples $\boldsymbol{u}_h^{n+1} = (u_{1,h}^{n+1},\dots,u_{5,h}^{n+1})^{\mathrm{T}}$ of discrete functions on $\Omega$. We define the norm,
	\begin{equation*}
		\|\boldsymbol{u}_h^{n+1}\|_{X_h}^{2}:=\sum_{i=1}^5\Big(|u_{i,h}^{n+1}|_{1,\mathcal{T}}^2+\|u_{i,h}^{n+1}\|_{L^2(\Omega)}^2\Big).
	\end{equation*}
	Let $\boldsymbol{\psi}_h = (\psi_{1,h},\dots,\psi_{5,h})^{\mathrm{T}}$. Multiplying the equations of the truncated discrete problem  by $\psi_{i,h}$ and summing the result over $K\in \mathcal{T}$ we obtain
	\begin{equation}
		\begin{aligned}
			&\dfrac{1}{\Delta t}\Big( \boldsymbol{B}_h(\boldsymbol{u}_h^{n+1},\boldsymbol{\psi}_h^{n+1})-\boldsymbol{B}_h(\boldsymbol{u}_h^{n},\boldsymbol{\psi}_h^{n+1})\Big)\\
			&\quad +\boldsymbol{a}_{1,h}(\boldsymbol{u}_h^{n+1},\boldsymbol{\psi}_h^{n+1})+\boldsymbol{a}_{2,h}(\boldsymbol{u}_h^{n+1},\boldsymbol{\psi}_h^{n+1})-\boldsymbol{B}_h(\bar{F}_{i}(\boldsymbol{u}_{h}^{n},\boldsymbol{u}_{h}^{n+1}),\boldsymbol{\psi}_h^{n+1}) = 0,
		\end{aligned}
	\end{equation}
	where the discrete bilinear forms are given by
	\begin{align*}
		\boldsymbol{B}_h(\boldsymbol{u}_h^{n+1},\boldsymbol{\psi}_h^{n+1})&:=\sum_{K\in \mathcal{T}} m(K) \sum_{i=1}^{5} u_{i,K}^{n+1} \psi_{i,K}^{n+1},\\
		\boldsymbol{a}_{1,h}(\boldsymbol{u}_h^{n+1},\boldsymbol{\psi}_h^{n+1})&:=\frac{1}{2}\sum_{K\in \mathcal{T}}\sum_{L\in \mathcal{N}(K)} \tau_{K|L} \sum_{\substack{i=1\\ i\neq 2}}^{5} d_i(u_{i,L}^{n+1}-u_{i,K}^{n+1})(\psi_{i,L}^{n+1}-\psi_{i,K}^{n+1}),\\
		\boldsymbol{a}_{2,h}(\boldsymbol{u}_h^{n+1},\boldsymbol{\psi}_h^{n+1})&:=-\frac{1}{2}\sum_{K\in \mathcal{T}}\sum_{L\in \mathcal{N}(K)} \tau_{K|L} \mathcal{\tilde{G}}\big(u_{4,K}^{n+1},u_{4,L}^{n+1};\delta \tilde{u}_{1,KL}^{n+1}\big) (\psi_{4,L}^{n+1}-\psi_{4,K}^{n+1}).
	\end{align*}
	Now, we define, by duality, the mapping $\mathbb{P}:X_h\longrightarrow X_h$ given by:
	\begin{equation}
		\begin{aligned}
			\big[\mathbb{P}(\boldsymbol{u}_h^{n+1}),\boldsymbol{\phi}_h\big]&= \dfrac{1}{\Delta t}\Big( \boldsymbol{B}_h(\boldsymbol{u}_h^{n+1},\boldsymbol{\phi}_h)-\boldsymbol{B}_h(\boldsymbol{u}_h^{n},\boldsymbol{\phi}_h)\Big)\\
			&\quad +\boldsymbol{a}_{1,h}(\boldsymbol{u}_h^{n+1},\boldsymbol{\phi}_h)+\boldsymbol{a}_{2,h}(\boldsymbol{u}_h^{n+1},\boldsymbol{\phi}_h)-\boldsymbol{B}_h(\bar{F}_{i}(\boldsymbol{u}_{h}^{n},\boldsymbol{u}_{h}^{n+1}),\boldsymbol{\phi}_h),
		\end{aligned}
	\end{equation}
	for all $\boldsymbol{\phi}_h \in X_h$. It is easy to check that the mapping $\mathbb{P}$ is continuous. Now, we test with the function $\boldsymbol{\phi}_h = \boldsymbol{u}_h^{n+1}$ to get
	\begin{align*}
		\big[\mathbb{P}(\boldsymbol{u}_h^{n+1}),\boldsymbol{u}_h^{n+1} \big] = \mathcal{A}_1+\mathcal{A}_2+\mathcal{A}_3,
	\end{align*}
	where,
	\begin{align*}
		\mathcal{A}_1&= \dfrac{1}{\Delta t}\sum_{K\in \mathcal{T}} m(K) \sum_{i=1}^5 (u_{i,K}^{n+1})^2-\dfrac{1}{\Delta t}\sum_{K\in \mathcal{T}} m(K) \sum_{i=1}^5 u_{i,K}^{n}u_{i,K}^{n+1} + \frac{1}{2}\sum_{K\in \mathcal{T}}\sum_{L\in \mathcal{N}(K)} \tau_{K|L} \sum_{\substack{i=1\\ i\neq 2}}^{5} d_i(u_{i,L}^{n+1}-u_{i,K}^{n+1})^2,\\ 
		\mathcal{A}_2&= -\frac{1}{2}\sum_{K\in \mathcal{T}}\sum_{L\in \mathcal{N}(K)} \tau_{K|L} \mathcal{\tilde{G}}\big(u_{4,K}^{n+1},u_{4,L}^{n+1};\delta \tilde{u}_{1,KL}^{n+1}\big) (u_{4,L}^{n+1}-u_{4,K}^{n+1}),\\
		\mathcal{A}_3&= - \sum_{n=0}^{N}\Delta t\sum_{K\in \mathcal{T}} m(K) \sum_{i\in \mathcal{J}} \bar{F}_{i}(\boldsymbol{u}_{K}^{n},\boldsymbol{u}_{K}^{n+1})u_{i,K}^{n+1}.
	\end{align*}
	Then, by using the definition of $\bar{F}_i$, for $i=1,\dots,5$ given by \eqref{eq:rhs-disc} and Young's inequality we get that
	\begin{align*}
		\mathcal{A}_1&\geq \dfrac{1}{\Delta t} \sum_{i=1}^5 \|u_{i,h}^{n+1}\|_{L^2(\Omega)}^2-\dfrac{1}{2\Delta t} \sum_{i=1}^5 \|u_{i,h}^{n+1}\|_{L^2(\Omega)}^2-\dfrac{1}{2\Delta t}\sum_{i=1}^5 \|u_{i,h}^{n}\|_{L^2(\Omega)}^2\\
		&=\dfrac{1}{2\Delta t} \sum_{i=1}^5 \|u_{i,h}^{n+1}\|_{L^2(\Omega)}^2-\dfrac{1}{2\Delta t}\sum_{i=1}^5 \|u_{i,h}^{n}\|_{L^2(\Omega)}^2.
	\end{align*}
	In addition, we employ the property \eqref{eq:flux-h5} to obtain that
	\begin{align*}
		\mathcal{\tilde{G}}\big(u_{4,K}^{n+1},u_{4,L}^{n+1};\delta \tilde{u}_{1,KL}^{n+1}\big)&\leq |\delta \tilde{u}_{1,KL}^{n+1}| \int_0^{\beta_4} \big| \chi'(s)\big| \, \mathrm{d}s \leq C_1,
	\end{align*}
	where $C_1:=2\beta_1 \int_0^{\beta_4} \big| \chi'(s)\big| \, \mathrm{d}s$. Then, by employing Cauchy-Schwarz inequality we get
	\begin{align*}
		\mathcal{A}_2&\geq -\frac{1}{2}\sum_{K\in \mathcal{T}}\sum_{L\in \mathcal{N}(K)} \tau_{K|L} C_1 |u_{4,L}^{n+1}-u_{4,K}^{n+1}| \geq  -C_2|u_{4,h}^{n+1}|_{1,\mathcal{T}}\geq -C_2 \|\boldsymbol{u}_{h}^{n+1}\|_{X_h},
	\end{align*}
	where $C_2$ is a positive constant. Moreover, we have that
	\begin{align*}
		\mathcal{A}_3&= \sum_{K\in \mathcal{T}} m(K)\Big\{ [\gamma(u_{4,K}^n)+\tau_0](u_{1,K}^{n+1})^2+\tau_p(u_{2,K}^{n+1})^2 \\
		&\quad +[d+r_2u_{1,K}^n+r_1 u_{3,K}^n](u_{3,K}^{n+1})^2 +  \left(\dfrac{\alpha_1 u_{1,K}^n u_{4,K}^{n}}{1+\alpha_2 u_{1,K}^n}+\sigma\right)(u_{4,K}^{n+1})^2+\tau_3 (u_{5,K}^{n+1})^2 \Big\}\\
		&\quad  -\sum_{K\in \mathcal{T}} m(K) \Bigg( r_1 (u_{3,K}^n)^2 u_{1,K}^{n+1}+\gamma(u_{4,K}^n) u_{1,K}^{n} u_{2,K}^{n+1} +\dfrac{\tau_S}{1+C (u_{1,K}^n)^{\nu}} u_{5,K}^n u_{3,K}^{n+1}\\
		&\quad +\dfrac{\alpha_1 \hat{m} u_{1,K}^n}{1+\alpha_2 u_{1,K}^n}u_{4,K}^{n}u_{4,K}^{n+1} +\lambda_{M,K} u_{4,K}^{n+1}+\dfrac{\tau_1 u_{1,K}^n}{1+\tau_2 u_{1,K}^n}u_{4,K}^{n}u_{5,K}^{n+1}\Bigg)\\
		&\geq - \frac{1}{2}\sum_{K\in \mathcal{T}} m(K)\Bigg\{r_1\beta_3 (u_{3,K}^n)^2+r_1\beta_3(u_{1,K}^{n+1})^2 +\gamma_{\max}(u_{1,K}^n)^2+\gamma_{\max}(u_{2,K}^{n+1})^2 +\tau_{S}(u_{5,K}^n)^2+\tau_{S}(u_{3,K}^{n+1})^2\\&\qquad+\left(\frac{\hat{m}\alpha_1}{\alpha_2}+\frac{\tau_1}{\tau_2}\right)(u_{4,K}^{n})^2+\left(\frac{\hat{m}\alpha_1}{\alpha_2}+1\right)(u_{4,K}^{n+1})^2+\|\lambda_M\|_{L^{\infty}(\Omega)}^2+\frac{\tau_1}{\tau_2} (u_{5,K}^{n+1})^2\Bigg\}\\
		&\geq -\frac{r_1\beta_3}{2}\|u_{3,h}^n\|_{L^2(\Omega)}^2 -\left(\frac{r_1\beta_3+\tau_S}{2}\right)\|u_{3,h}^{n+1}\|_{L^2(\Omega)}^2-\frac{\gamma_{\max}}{2}\|u_{1,h}^n\|_{L^2(\Omega)}^2-\frac{\gamma_{\max}}{2}\|u_{2,h}^{n+1}\|_{L^2(\Omega)}^2\\
		&\quad -\dfrac{\tau_S}{2}\|u_{5,h}^{n}\|_{L^2(\Omega)}^2-\frac{1}{2}\left(\frac{\hat{m}\alpha_1}{\alpha_2}+\frac{\tau_1}{\tau_2}\right)\|u_{4,h}^{n}\|_{L^2(\Omega)}^2-\frac{1}{2}\left(\frac{\hat{m}\alpha_1}{\alpha_2}+1\right)\|u_{4,h}^{n+1}\|_{L^2(\Omega)}^2\\
		&\quad -\frac{1}{2}\|\lambda_M\|_{L^{\infty}(\Omega)}^2 m(\Omega)-\frac{\tau_1}{\tau_2}\|u_{5,h}^{n+1}\|_{L^2(\Omega)}^2.
	\end{align*}
	According to the mild restriction \eqref{eq:mild} we get that
	\begin{align*}
		\big[\mathbb{P}(\boldsymbol{u}_h^{n+1}),\boldsymbol{u}_h^{n+1}\big]&\geq \frac{1}{2 \Delta t}\|u_{1,h}^{n+1}\|_{L^2(\Omega)}^2 +\frac{1}{2}\left(\dfrac{1}{\Delta t}-\gamma_{\max} \right)\|u_{2,h}^{n+1}\|_{L^2(\Omega)}^2 +\frac{1}{2}\left[\dfrac{1}{\Delta t}-(r_1\beta_3+\tau_S)\right]\|u_{3,h}^{n+1}\|_{L^2(\Omega)}^2\\
		&\quad+\frac{1}{2}\left[\dfrac{1}{\Delta t}-\left(\frac{\hat{m}\alpha_1}{\alpha_2}+1\right)\right]\|u_{4,h}^{n+1}\|_{L^2(\Omega)}^2+\frac{1}{2}\left(\dfrac{1}{\Delta t}-\frac{\tau_1}{\tau_2}\right)\|u_{5,h}^{n+1}\|_{L^2(\Omega)}^2 \\
		&\quad + \frac{1}{2 \Delta t}\sum_{i=1}^5 |u_{i,h}^{n+1}|_{1,\mathcal{T}}^2
		-\dfrac{1}{2\Delta t}\sum_{i=1}^5 \|u_{i,h}^{n}\|_{L^2(\Omega)}^2-C_2 \|\boldsymbol{u}_{h}^{n+1}\|_{X_h}-C_3\sum_{i=1}^5 \|u_{i,h}^{n}\|_{L^2(\Omega)}^2\\
		&\quad -\frac{1}{2}\|\lambda_M\|_{L^{\infty}(\Omega)}^2 m(\Omega)\\
		&\geq C_4  \|\boldsymbol{u}_{h}^{n+1}\|_{X_h}^2 - C_5 \|\boldsymbol{u}_{h}^{n}\|_{X_h}^2-C_2 \|\boldsymbol{u}_{h}^{n+1}\|_{X_h} -\frac{1}{2}\|\lambda_M\|_{L^{\infty}(\Omega)}^2 m(\Omega),
	\end{align*}
	where $C_i>0$, $i=1,\dots,5$ are constants depending only on $\Delta t$ and on the parameters of the model. We then observe that the right-hand side of the inequality above is a second-order polynomial in $\|\boldsymbol{u}_{h}^{n+1}\|_{X_h}^2$ with a positive leading coefficient. Hence, there exists $k >0$ large enough such that $\big[\mathbb{P}(\boldsymbol{u}_h^{n+1}), \boldsymbol{u}_h^{n+1}\big] > 0$ whenever $\|\boldsymbol{u}_{h}^{n+1}\|_{X_h} =k$. By using Lemma 1.4 in \cite{temam2001navier}, we conclude that there exists $\boldsymbol{u}_{h}^{n+1}\in X_h$ such that $\mathbb{P}(\boldsymbol{u}_{h}^{n+1})=0$, which completes the proof.
\end{proof}

\begin{theorem}[Existence of an admissible solution for all time]\label{thm:exist-fvm}
	Let $\mathcal{T}$ be an admissible discretization of $\Omega$ and $\boldsymbol{u}^{0}\in \mathcal{R}$. Then, the discrete problem \eqref{eq:fvm-1}-\eqref{eq:fvm-5} has a solution $(\boldsymbol{u}_K^{n+1})_{K\in \mathcal{T},n\in \{0,\dots,N\}}$ which is physically admissible, i.e. $\boldsymbol{u}_K^{n+1}\in \mathcal{R}$, for all $K\in \mathcal{T}$, $n\in \{0,\dots,N\}$.
\end{theorem}
\begin{proof}
	Let us prove the theorem by induction on $n$. We see that the solution exists for $t_0$ and is given by $\boldsymbol{u}_K^{0} = \boldsymbol{u}^0|_K$, for all $K\in \mathcal{T}$, which is admissible by hypothesis. Assume that an admissible solution exists for $t=t_n$, then by Lemma \ref{lemma:exist-truncated}, there exists a solution $(\boldsymbol{u}_K^{n+1})_{K\in \mathcal{T}}$ to the truncated discrete problem and by employing Lemma \ref{lemma:maxprin} it follows that $\boldsymbol{u}_K^{n+1}\in \mathcal{R}$, for all $K\in \mathcal{T}$. But, since it is admissible then for all $K\in \mathcal{T}$ and $L\in \mathcal{N}(K)$, we have that $\tilde{u}_{i,K}^{n+1} = u_{i,K}^{n+1}$, $i=1,\dots,5$, and $\mathcal{\tilde{G}}\big(u_{4,K}^{n+1},u_{4,L}^{n+1};\delta \tilde{u}_{1,KL}^{n+1}\big) = \mathcal{G}\big(u_{4,K}^{n+1},u_{4,L}^{n+1};\delta u_{1,KL}^{n+1}\big)$. Therefore $(\boldsymbol{u}_K^{n+1})$ is an admissible solution of the original system \eqref{eq:fvm-1}-\eqref{eq:fvm-5}. The proof is completed for all times $t_n$ ($n \in \mathbb{N}$) by induction.
\end{proof}

\subsection{Discrete estimates}\label{sec:de}
In this section, following arguments analogous to those in \cite{EYMARD2000713,andreianov2011finite,coudiere2013analysis}, we establish the discrete estimates required for proving the convergence of scheme \eqref{eq:fvm-1}-\eqref{eq:fvm-5} towards a weak solution of \eqref{eq:model}. Let us consider the index set $\mathcal{J}:=\{1,3,4,5\}$.
\begin{lemma}\label{lemma:ae}
	Let $(\boldsymbol{u}_K^{n+1})_{K\in \mathcal{T},n\in \{0,\dots,N\}}$ be a solution of \eqref{eq:fvm-1}-\eqref{eq:fvm-5}. Then is a constant $\tilde{C}_1>0$ depending on the $L^2-$norm of the initial condition $\|\boldsymbol{u}_0\|_{L^2(\Omega)}$ and the parameters of the model, such that 
	\begin{align}
		&\dfrac{1}{2}\sum_{n=0}^{N}\Delta t \sum_{K\in \mathcal{T}}\sum_{L\in \mathcal{N}(K)} \tau_{K|L} \sum_{i\in \mathcal{J}}  (u_{i,L}^{n+1}-u_{i,K}^{n+1})^2\leq \tilde{C}_1. \label{eq:ae-bound-2}
	\end{align}
\end{lemma}
\begin{proof}
	We start the proof by multiplying \eqref{eq:fvm-1}-\eqref{eq:fvm-5} by $\Delta t u_{i,K}^{n+1}$ and adding the result over $i\in \mathcal{J}, K\in \mathcal{T},$ and $n\in \{0,\dots,N\}$ to obtain the identity $X_1+X_2+X_3+X_4=0$, where
	\begin{equation*}
		\begin{aligned}
			X_1 &:= \sum_{n=0}^{N}\sum_{K\in \mathcal{T}} m(K)\sum_{i\in \mathcal{J}} (u_{i,K}^{n+1}-u_{i,K}^n)u_{i,K}^{n+1},\\ 
			X_2 &:= -\sum_{n=0}^{N}\Delta t\sum_{K\in \mathcal{T}} \sum_{L\in \mathcal{N}(K)} \tau_{K|L} \sum_{i\in \mathcal{J}} d_i (u_{i,L}^{n+1}-u_{i,K}^{n+1})u_{i,K}^{n+1},\\
			X_3 &:=   \sum_{n=0}^{N}\Delta t\sum_{K\in \mathcal{T}} \sum_{L\in \mathcal{N}(K)} \tau_{K|L} \mathcal{G}\big(u_{4,K}^{n+1},u_{4,L}^{n+1};\delta u_{1,KL}^{n+1}\big) u_{4,K}^{n+1},\\
			X_4 &:= - \sum_{n=0}^{N}\Delta t\sum_{K\in \mathcal{T}} m(K) \sum_{i\in \mathcal{J}} \bar{F}_{i}(\boldsymbol{u}_{K}^{n},\boldsymbol{u}_{K}^{n+1})u_{i,K}^{n+1}.
		\end{aligned}
	\end{equation*}
	By using the inequality ``$a(a-b)\geq \frac{1}{2}(a^2-b^2)$ for $a,b\in \mathbb{R}$'' we get a lower bound for $X_1$ given by:
	\[X_1\geq \dfrac{1}{2}\sum_{n=0}^{N}\sum_{K\in \mathcal{T}} m(K) \sum_{i\in \mathcal{J}} \Big((u_{i,K}^{n+1})^2-(u_{i,K}^n)^2\Big)=\dfrac{1}{2}\sum_{K\in \mathcal{T}} m(K) \sum_{i\in \mathcal{J}} \Big((u_{i,K}^{N+1})^2-(u_{i,K}^0)^2\Big).\]
	Now, by using the identity $2a(a-b) = a^2-b^2+(a-b)^2$ and gathering the edges we can write
	\begin{align*}
		X_2 &= \dfrac{\Delta t}{2} \sum_{n=0}^{N}\sum_{K\in \mathcal{T}} \sum_{L\in \mathcal{N}(K)} \tau_{K|L} \sum_{i\in \mathcal{J}} d_i \Big((u_{i,L}^{n+1})^2-(u_{i,K}^{n+1})^2+(u_{i,L}^{n+1}-u_{i,K}^{n+1})^2\Big)\\
		&= \dfrac{\Delta t}{2} \sum_{n=0}^{N}\sum_{K\in \mathcal{T}} \sum_{L\in \mathcal{N}(K)} \tau_{K|L} \sum_{i\in \mathcal{J}} d_i (u_{i,L}^{n+1}-u_{i,K}^{n+1})^2.
	\end{align*}
	Then employing summation by parts we get
	\begin{align*}
		X_3 &=  \sum_{n=0}^{N}\Delta t\sum_{K\in \mathcal{T}} \sum_{L\in \mathcal{N}(K)} \tau_{K|L} \mathcal{G}\big(u_{4,K}^{n+1},u_{4,L}^{n+1};\delta u_{1,KL}^{n+1}\big) u_{4,K}^{n+1}\\
		&=-\frac{ 1}{2} \sum_{n=0}^{N}\Delta t\sum_{K\in \mathcal{T}} \sum_{L\in \mathcal{N}(K)} \tau_{K|L} \mathcal{G}\big(u_{4,K}^{n+1},u_{4,L}^{n+1};\delta u_{1,KL}^{n+1}\big)(u_{4,L}^{n+1}-u_{4,K}^{n+1}).
	\end{align*}
	We recall that $\delta u_{1,KL}^{n+1} = u_{1,L}^{n+1}-u_{1,K}^{n+1}$; using in addition the assumptions (H2)-(H4) together with the boundedness of $u_{4,K}^{n+1}$, $K\in \mathcal{T}, n\in \{0,\dots,N\}$, and Young's inequality ``$ab\leq \frac{a^2}{2}+\frac{b^2}{2}$, for $a,b>0$'' we deduce an upper bound for $X_3$ given by
	\begin{align*}
		|X_3|&\leq \frac{1}{2} \sum_{n=0}^{N} \Delta t\sum_{K\in \mathcal{T}} \sum_{L\in \mathcal{N}(K)} \tau_{K|L} \big|\mathcal{G}\big(u_{4,K}^{n+1},u_{4,L}^{n+1};\delta u_{1,KL}^{n+1}\big)\big||u_{4,L}^{n+1}-u_{4,K}^{n+1}|\\
		&\leq \frac{ 1}{2} \sum_{n=0}^{N}\Delta t\sum_{K\in \mathcal{T}} \sum_{L\in \mathcal{N}(K)} \tau_{K|L} L_{\mathcal{G}} |u_{1,L}^{n+1}-u_{1,K}^{n+1}|\big(|u_{4,K}^{n+1}|+|u_{4,L}^{n+1}|\big) |u_{4,L}^{n+1}-u_{4,K}^{n+1}|\\
		&\leq  L_{\mathcal{G}} \beta_4 \sum_{n=0}^{N}\Delta t\sum_{K\in \mathcal{T}} \sum_{L\in \mathcal{N}(K)} \tau_{K|L}|u_{1,L}^{n+1}-u_{1,K}^{n+1}| |u_{4,L}^{n+1}-u_{4,K}^{n+1}|\\
		&\leq  \dfrac{L_{\mathcal{G}} \beta_4}{2} \sum_{n=0}^{N}\Delta t\sum_{K\in \mathcal{T}} \sum_{L\in \mathcal{N}(K)} \tau_{K|L}|u_{1,L}^{n+1}-u_{1,K}^{n+1}|^2+ \dfrac{L_{\mathcal{G}} \beta_4}{2} \sum_{n=0}^{N}\Delta t\sum_{K\in \mathcal{T}} \sum_{L\in \mathcal{N}(K)} \tau_{K|L}|u_{4,L}^{n+1}-u_{4,K}^{n+1}|^2.
	\end{align*}
	Finally, from Lemma \ref{lemma:maxprin} and Young's inequality we get
	\begin{align*}
		X_4 =& - \sum_{n=0}^{N}\Delta t\sum_{K\in \mathcal{T}} m(K) \Bigg\{ \left( r_1 (u_{3,K}^n)^2-\gamma(u_{4,K}^n) u_{1,K}^{n+1}-\tau_0 u_{1,K}^{n+1}\right) u_{1,K}^{n+1}\\
		&\quad +\left(\dfrac{\tau_S}{1+C (u_{1,K}^n)^{\nu}}u_{5,K}^n-du_{3,K}^{n+1}-r_2u_{1,K}^nu_{3,K}^{n+1}-r_1 u_{3,K}^n u_{3,K}^{n+1}\right)u_{3,K}^{n+1}\\
		&\quad+\left(\dfrac{\alpha_1 u_{1,K}^n}{1+\alpha_2 u_{1,K}^n}(\hat{m}-u_{4,K}^{n+1})u_{4,K}^{n}-\sigma u_{4,K}^{n+1}+\lambda_M\right)u_{4,K}^{n+1}+\left. \left(\dfrac{\tau_1 u_{1,K}^n}{1+\tau_2 u_{1,K}^n}u_{4,K}^n-\tau_3 u_{5,K}^{n+1}\right)u_{5,K}^{n+1}\right\}\\
		=& \sum_{n=0}^{N}\Delta t\sum_{K\in \mathcal{T}} m(K)\Big\{ [\gamma(u_{4,K}^n)+\tau_0](u_{1,K}^{n+1})^2+[d+r_2u_{1,K}^n+r_1 u_{3,K}^n](u_{3,K}^{n+1})^2\\
		&\quad +  \left(\dfrac{\alpha_1 u_{1,K}^n u_{4,K}^{n}}{1+\alpha_2 u_{1,K}^n}+\sigma\right)(u_{4,K}^{n+1})^2+\tau_3 (u_{5,K}^{n+1})^2 \Big\}\\
		& - \sum_{n=0}^{N}\Delta t\sum_{K\in \mathcal{T}} m(K) \Bigg( r_1 (u_{3,K}^n)^2 u_{1,K}^{n+1} +\dfrac{\tau_S}{1+C (u_{1,K}^n)^{\nu}} u_{5,K}^n u_{3,K}^{n+1}+\dfrac{\alpha_1 \hat{m} u_{1,K}^n}{1+\alpha_2 u_{1,K}^n}u_{4,K}^{n}u_{4,K}^{n+1}\\
		&\quad  +\lambda_{M,K} u_{4,K}^{n+1}+\dfrac{\tau_1 u_{1,K}^n}{1+\tau_2 u_{1,K}^n}u_{4,K}^{n}u_{5,K}^{n+1}\Bigg)\\
		&\geq -\dfrac{1}{2}\sum_{n=0}^{N}\Delta t\sum_{K\in \mathcal{T}} m(K)\Bigg\{r_1\beta_3 (u_{3,K}^n)^2+r_1\beta_3(u_{1,K}^{n+1})^2 +\tau_{S}(u_{5,K}^n)^2+\tau_{S}(u_{3,K}^{n+1})^2\\
		&\qquad+\left(\frac{\hat{m}\alpha_1}{\alpha_2}+\frac{\tau_1}{\tau_2}\right)(u_{4,K}^{n})^2+\left(\frac{\hat{m}\alpha_1}{\alpha_2}+1\right)(u_{4,K}^{n+1})^2+\lambda_{M,K}^2+\frac{\tau_1}{\tau_2} (u_{5,K}^{n+1})^2\Bigg\}\\
		&\geq-\dfrac{1}{2}\sum_{n=0}^{N}\Delta t\sum_{K\in \mathcal{T}} m(K) \|\lambda_M\|_{L^{\infty}(\Omega)}^2 -C_1 \sum_{K\in \mathcal{T}} m(K)\sum_{i\in \mathcal{J}} (u_{i,K}^0)^2-\dfrac{C_2}{2}\sum_{n=0}^{N}\Delta t\sum_{K\in \mathcal{T}} m(K)\sum_{i\in \mathcal{J}} (u_{i,K}^{n+1})^2
	\end{align*}
	where $C_i>0$ are constants.
	Collecting all the previous inequalities we arrive at
	\begin{equation}\label{eq:ae-1}
		\begin{aligned}
			\dfrac{1}{2}&\sum_{K\in \mathcal{T}} m(K) \sum_{i\in \mathcal{J}} (u_{i,K}^{N+1})^2 + C_3 \sum_{n=0}^{N}\Delta t\sum_{K\in \mathcal{T}} \sum_{L\in \mathcal{N}(K)} \tau_{K|L} \sum_{i\in \mathcal{J}} (u_{i,L}^{n+1}-u_{i,K}^{n+1})^2\\
			&\leq C_4 +\frac{C_2\Delta t}{2}\sum_{n=0}^{N}\sum_{K\in \mathcal{T}} m(K)\sum_{i\in \mathcal{J}} (u_{i,K}^{n+1})^2.
		\end{aligned}
	\end{equation}
	By combining Lemma \ref{lemma:maxprin} with \eqref{eq:ae-1}, one can deduce \eqref{eq:ae-bound-2}.
\end{proof}
\section{Convergence}\label{sec:conver}
In this section we will use the a priori estimates obtained in Section \ref{sec:de} to obtain the convergence of the FV method to a weak solution in the sense of Definition \ref{def:weaksol} of the system \eqref{eq:model}. To do so, let us consider a sequence of admissible meshes $(\mathcal{T}_m)_{m\geq 1}$ of $\Omega$ such that $\mathrm{size}(\mathcal{T}_m)\to 0$, as $m\to +\infty$ and let $(N_m)_{m\geq 1}$ an increasing sequence of integers, so we obtain a sequence of time steps $(\Delta t_m)_{m\geq 1}$ such that $\Delta t_m\to 0$, as $m\to +\infty$. During this section we will employ the notation $\boldsymbol{u}_{h_m} = (u_{1,h_m},\dots,u_{5,h_m})^{\mathrm{T}}$ for a piecewise constant solution of \eqref{eq:fvm-1} defined as \eqref{eq:piece-cons}. 
\subsection{Compactness argument}\label{sec:compact}
Let us define the translated space $Q_{T-\tau}:=\Omega\times (0,T-\tau)$, for all $\tau \in (0,T).$
\begin{lemma}[Time-translate estimates]\label{lemma:time-translate}
	Let $\Delta t_0>0$ small enough. Given a time-step $\Delta t_m \leq \Delta t_0$, then there exists a constant $C>0$ independent of $m$ and $\tau$ such that
	\begin{equation}\label{eq:time-translate}
		\iint_{Q_{T-\tau}} \big| u_{i,h_m}(\boldsymbol{x},t+\tau)-u_{i,h_m}(\boldsymbol{x},t)\big|^2 \, \mathrm{d} \boldsymbol{x}\, \mathrm{d}t\leq C(\tau+\Delta t_m), \quad i=1,3,4,5,
	\end{equation}
	for all $\tau \in (0,T)$.
\end{lemma}
\begin{proof}
	We focus on proving the estimate \eqref{eq:time-translate} for $i=4$; the proof is analogous for $i=1,3,5$. Let us introduce the quantity 
	\begin{equation}
		\mathcal{C}_m(t) = \int_{\Omega} \big|u_{4,h_m}(\boldsymbol{x},t+\tau)-u_{4,h_m}(\boldsymbol{x},t) \big|^2 \, \mathrm{d}\boldsymbol{x}, \quad \text{ for all }t\in (0,T-\tau).
	\end{equation}
	We set $n_0(t) = [t/\Delta t_m]$ and $n_1(t) = [(t+\tau)/\Delta t_m]$, where $[x] = n$ for $x\in [n,n+1),$ $n\in \mathbb{N}$. We can then  rewrite $\mathcal{C}_m(t)$ as
	\begin{align*}
		\mathcal{C}_m(t) &= \sum_{K\in \mathcal{T}_m} m(K) \big(u_{4,K}^{n_1(t)}- u_{4,K}^{n_0(t)}\big)^2\leq  \sum_{K\in \mathcal{T}_m} \Big( \big(u_{4,K}^{n_1(t)}- u_{4,K}^{n_0(t)}\big)\times \sum_{t\leq n \Delta t_m<t+\tau} m(K)(u_{4,K}^{n+1}-u_{4,K}^n) \Big).
	\end{align*}
	Using the equation \eqref{eq:fvm-4}, Lemma \ref{lemma:maxprin}, summation by parts, along with the weighted Young inequality, we get
	\begin{align*}
		\mathcal{C}_m(t)&\leq \sum_{t\leq n \Delta t_m<t+\tau} \Delta t_m \sum_{K\in \mathcal{T}_m} \big(u_{4,K}^{n_1(t)}- u_{4,K}^{n_0(t)}\big) \Bigg(\sum_{L\in \mathcal{N}(K)} \tau_{K|L}\Big[ d_4 (u_{4,L}^{n+1}-u_{4,K}^{n+1}) +\mathcal{G}\big(u_{4,K}^{n+1},u_{4,L}^{n+1};\delta u_{1,KL}^{n+1}\big)\Big]\Bigg)\\
		&\quad + \sum_{t\leq n \Delta t_m<t+\tau} \Delta t_m \sum_{K\in \mathcal{T}_m} m(K)\Big(u_{4,K}^{n_1(t)}- u_{4,K}^{n_0(t)} \Big)\Big(\dfrac{\alpha_1 u_{1,K}^n}{1+\alpha_2 u_{1,K}^n}(\hat{m}-u_{4,K}^{n+1})u_{4,K}^{n} -\sigma u_{4,K}^{n+1}+\lambda_M\Big)\\
		&\leq \dfrac{1}{2}\sum_{t\leq n \Delta t_m<t+\tau} \Delta t_m \sum_{K\in \mathcal{T}_m} \sum_{L\in \mathcal{N}(K)} \tau_{K|L} \Bigg\{ \Big[  d_4 (u_{4,L}^{n+1}-u_{4,K}^{n+1})+\mathcal{G}\big(u_{4,K}^{n+1},u_{4,L}^{n+1};\delta u_{1,KL}^{n+1}\big)\Big]\\
		&\quad \times \Big[(u_{4,K}^{n_1(t)}-u_{4,L}^{n_1(t)})-(u_{4,K}^{n_0(t)}-u_{4,L}^{n_0(t)}) \Big]\Bigg\}\\
		&\quad + \sum_{t\leq n \Delta t_m<t+\tau} \Delta t_m \sum_{K\in \mathcal{T}_m} m(K)\Big(u_{4,K}^{n_1(t)}- u_{4,K}^{n_0(t)} \Big)\Big(\dfrac{\alpha_1 u_{1,K}^n}{1+\alpha_2 u_{1,K}^n}(\hat{m}-u_{4,K}^{n+1})u_{4,K}^{n} -\sigma u_{4,K}^{n+1}+\lambda_M\Big)\\
		&\leq \mathcal{C}_{1,m}(t)+\mathcal{C}_{2,m}(t)+\mathcal{C}_{3,m}(t)+\mathcal{C}_{4,m}(t)+\mathcal{C}_{5,m}(t),
	\end{align*}
	where we have defined,
	\begin{equation}
		\begin{aligned}
			\mathcal{C}_{1,m}(t) &= d_4\sum_{t\leq n \Delta t_m<t+\tau} \Delta t_m \sum_{K\in \mathcal{T}_m}\sum_{L\in \mathcal{N}(K)}\tau_{K|L}  \big|u_{4,L}^{n+1}-u_{4,K}^{n+1}\big|^2,\\
			\mathcal{C}_{2,m}(t) &= \Big(\frac{d_4}{2}+L_{\mathcal{G}}\beta_4\Big)\sum_{t\leq n \Delta t_m<t+\tau} \Delta t_m \sum_{K\in \mathcal{T}_m}\sum_{L\in \mathcal{N}(K)}\tau_{K|L}\big|u_{4,L}^{n_1(t)}-u_{4,K}^{n_1(t)}\big|^2,\\
			\mathcal{C}_{3,m}(t) &= \Big(\frac{d_4}{2}+L_{\mathcal{G}}\beta_4\Big)\sum_{t\leq n \Delta t_m<t+\tau} \Delta t_m \sum_{K\in \mathcal{T}_m}\sum_{L\in \mathcal{N}(K)}\tau_{K|L}\big|u_{4,L}^{n_0(t)}-u_{4,K}^{n_0(t)}\big|^2,\\
			\mathcal{C}_{4,m}(t) &= 2L_{\mathcal{G}}\beta_4 \sum_{t\leq n \Delta t_m<t+\tau} \Delta t_m \sum_{K\in \mathcal{T}_m}\sum_{L\in \mathcal{N}(K)}\tau_{K|L}  \big|u_{1,L}^{n+1}-u_{1,K}^{n+1}\big|^2,\\
			\mathcal{C}_{5,m}(t) &= \Big[\dfrac{\alpha_1}{\alpha_2} \hat{m}(2\hat{m}+\sigma)+\lambda_M\Big] \sum_{t\leq n \Delta t_m<t+\tau} \Delta t_m \sum_{K\in \mathcal{T}_m} m(K)\big| u_{4,K}^{n_1(t)}-u_{4,K}^{n_0(t)}\big|.
		\end{aligned}
	\end{equation}
	We now introduce the characteristic function $\zeta$ defined by $\zeta(n,t_1,t_2) = 1$, if $t_1<(n+1) \Delta t_m\leq t_2$ and $\zeta(n,t_1,t_2)=0$ otherwise. Then, for any sequence $(a^n)_{n\in \mathbb{N}}$ of non-negative numbers we have that
	\begin{equation}\label{eq:time-tras-ineq-1}
		\int_{0}^{T-\tau} \sum_{t\leq n \Delta t_m<t+\tau} a^n\, \mathrm{d}t\leq \sum_{n=0}^{\big[\frac{T}{\Delta t_m}\big]} a^n \int_0^{T-\tau} \zeta(n,\tau,t+\tau)\, \mathrm{d}t\leq \tau \sum_{n=0}^{\big[\frac{T}{\Delta t_m}\big]} a^n,
	\end{equation}
	and for any $\xi \in [0,\tau]$
	\begin{equation}\label{eq:time-tras-ineq-2}
		\int_0^{T-\tau} \sum_{t\leq n \Delta t_m<t+\tau} a^{[(t+\xi)/\Delta t_m]}\, \mathrm{d}t \leq \tau \sum_{n=0}^{\big[\frac{T}{\Delta t_m}\big]} a^n.
	\end{equation}
	From \eqref{eq:time-tras-ineq-1} we deduce that
	\begin{align*}
		\int_0^{T-\tau} \mathcal{C}_{1,m}(t)\, \mathrm{d} t &\leq  \sum_{n=0}^{N_m} \Delta t_m \int_0^{T-\tau} \zeta(n,t,t+\tau) \sum_{K\in \mathcal{T}_m}\sum_{L\in \mathcal{N}(K)}\tau_{K|L}  \big|u_{4,L}^{n+1}-u_{4,K}^{n+1}\big|^2\, \mathrm{d}t\\
		&\leq \tau  \sum_{n=0}^{N_m} \Delta t_m \sum_{K\in \mathcal{T}_m}\sum_{L\in \mathcal{N}(K)}\tau_{K|L}  \big|u_{4,L}^{n+1}-u_{4,K}^{n+1}\big|^2.
	\end{align*}
	In view of the discrete estimate \eqref{eq:ae-bound-2}, there exists a constant $C>0$ such that
	\begin{equation}\label{eq:time-tras-ineq-3}
		\int_0^{T-\tau} \mathcal{C}_{1,m}(t)\, \mathrm{d} t  \leq \tau C.
	\end{equation}
	Following the same lines yields
	\begin{equation}\label{eq:time-tras-ineq-4}
		\int_0^{T-\tau} \mathcal{C}_{4,m}(t)\, \mathrm{d} t  \leq \tau C,
	\end{equation}
	for some constant $C>0$. Next, we use \eqref{eq:time-tras-ineq-2} with $\xi = \tau$ for $\mathcal{A}_{2,m}(t)$, and with  $\xi = 0$ for $\mathcal{A}_{3,m}(t)$ to obtain
	\begin{align*}
		\int_0^{T-\tau} \mathcal{C}_{2,m}(t)\, \mathrm{d} t \leq \tau \sum_{n=0}^{N_m} \Delta t_m \sum_{K\in \mathcal{T}_m}\sum_{L\in \mathcal{N}(K)}\tau_{K|L}  \big|u_{4,L}^{n+1}-u_{4,K}^{n+1}\big|^2,
	\end{align*}
	and,
	\begin{align*}
		\int_0^{T-\tau} \mathcal{C}_{3,m}(t)\, \mathrm{d} t \leq \tau \sum_{n=0}^{N_m} \Delta t_m \sum_{K\in \mathcal{T}_m}\sum_{L\in \mathcal{N}(K)}\tau_{K|L}  \big|u_{4,L}^{n+1}-u_{4,K}^{n+1}\big|^2.
	\end{align*}
	Then, we employ \eqref{eq:ae-bound-2} to deduce that
	\begin{equation}
		\int_0^{T-\tau} \mathcal{C}_{2,m}(t)\, \mathrm{d} t\leq \tau C, \quad \int_0^{T-\tau} \mathcal{C}_{3,m}(t)\, \mathrm{d} t \leq \tau C.
	\end{equation}
	Finally, if $\hat{C} = \frac{\alpha_1}{\alpha_2} \hat{m}(2\hat{m}+\sigma)+\lambda_M$, then by using \eqref{eq:time-tras-ineq-2} with $\xi = \tau$ and $\xi = 0$, and the Lemma \ref{lemma:maxprin} we get
	\begin{equation}
		\begin{aligned}
			\int_0^{T-\tau} \mathcal{C}_{5,m}(t)\, \mathrm{d} t &\leq \hat{C} \int_0^{T-\tau} \sum_{t\leq n \Delta t_m<t+\tau} \Delta t_m \sum_{K\in \mathcal{T}_m} m(K)\Big(\big| u_{4,K}^{n_1(t)}\big|+\big| u_{4,K}^{n_0(t)}\big|\Big) \, \mathrm{d} t\\
			&\leq \tau\hat{C}  \sum_{n=0}^{N_m} \Delta t_m \sum_{K\in \mathcal{T}_m} m(K)\Big(\big| u_{4,K}^{n+1}\big|+\big| u_{4,K}^{n+1}\big|\Big)\\
			&\leq \tau 2 \beta_4 \hat{C}   \sum_{n=0}^{N_m} \Delta t_m \sum_{K\in \mathcal{T}_m} m(K)\\
			&\leq \tau C,
		\end{aligned}
	\end{equation}
	for some constant $C>0$. This ends the proof of the time translate estimate.
\end{proof}
\begin{lemma}[Space translate estimates]\label{lemma:space-translate}
	Let $\Delta t_0>0$ small enough. Given a time-step $\Delta t_m \leq \Delta t_0$, then there exists a constant $C>0$ independent of $m$ and $\tau$ such that, for all $\boldsymbol{y}\in \mathbb{R}^2$,
	\begin{equation}\label{eq:space-translate}
		\int_0^{T} \int_{\Omega_{\boldsymbol{y}}} \big|u_{i,h_m}(\boldsymbol{x}+\boldsymbol{y},t)-u_{i,h_m}(\boldsymbol{x},t) \big|^2 \, \mathrm{d} \boldsymbol{x}\, \mathrm{d}t \leq C \| \boldsymbol{y}\|_{2}\big( \|\boldsymbol{y}\|_{2}+2(K_{\Omega}-1)h_m\big),
	\end{equation}
	for $i=1,3,4,5$, where $\Omega_{\boldsymbol{y}} = \{\boldsymbol{x}\in \Omega: \boldsymbol{x}+\boldsymbol{y}\in \Omega\}$ and $K_{\Omega}$ is the numbers of sides of $\Omega$.
\end{lemma}
\begin{proof}
	To establish this result, we follow the approach of Lemma~9.3 in Eymard et al.~\cite{EYMARD2000713}, to obtain
	\begin{align*}
		\int_0^{T} \int_{\Omega_{\boldsymbol{y}}}& \big|u_{i,h_m}(\boldsymbol{x}+\boldsymbol{y},t)-u_{i,h_m}(\boldsymbol{x},t) \big|^2 \, \mathrm{d} \boldsymbol{x}\, \mathrm{d}t \\
		&\leq \| \boldsymbol{y}\|_{2}\big( \|\boldsymbol{y}\|_{2}+2(K_{\Omega}-1)h_m\big) \sum_{n=0}^{N_m} \Delta t_m \sum_{K\in \mathcal{T}_m}\sum_{L\in \mathcal{N}(K)}\tau_{K|L}  \big|u_{i,L}^{n+1}-u_{i,K}^{n+1}\big|^2,
	\end{align*}
	for $i=1,3,4,5$. By using the estimate \eqref{eq:ae-bound-2}, one can obtain the space translate estimate \eqref{eq:space-translate}. This ends the proof of the lemma.
\end{proof}
\subsection{Convergence analysis}\label{sec:conv-ana}
\begin{theorem}[Convergence Towards an Admissible Weak Solution]\label{thm:conv}
	Assume that $\boldsymbol{u}_0 = (u_{1,0},\dots,u_{5,0})^{\mathrm{T}}$ is such that $u_{i,0}\in H^1(\Omega)$, for $i=1,3,4,5$ and $\boldsymbol{u}_0(\boldsymbol{x})\in \mathcal{R}$, for a.e. $\boldsymbol{x}\in \Omega$. Let $(\boldsymbol{u}_{h_m})_{m\geq 1}$ be a family of solutions such that $h_m\to 0$, as $m\to \infty$. There exists $u_{i}\in L^2(0,T;H^1(\Omega))$, $i=1,3,4,5$ and $u_2\in L^2(\Omega_T)$ such that, up to a subsequence
	\begin{enumerate}
		\item[\textit{a)}] $u_{i,h_m} \longrightarrow u_{i}$ strongly in $L^p(\Omega_T)$ and a.e. in $\Omega_T$ as $m\to \infty$, for all $1\leq p<\infty$, $i=1,3,4,5$.
		\item[\textit{b)}] $u_{2,h_m} \rightharpoonup u_{2}$ weakly in $L^2(\Omega_T)$ as $m\to \infty$.
		\item[\textit{c)}] $\nabla_{\mathcal{T}_m} u_{i,h_m} \rightharpoonup \nabla u_{i}$, weakly in $[L^2(\Omega_T)]^2$ as $m\to \infty$, for $i=1,3,4,5$.
		\item[\textit{d)}] Moreover, the limit $\boldsymbol{u}=(u_{1},\dots,u_{5})^{\mathrm{T}}$ is a weak solution (in the sense of Definition \ref{def:weaksol}) of the problem \eqref{eq:model}.
	\end{enumerate}
\end{theorem}
\begin{proof}
	We start by showing the convergence \textit{a)} for $i=4$. We apply Kolmogorov’s compactness criterion (see \cite{EYMARD2000713}, Theorem 14.1) as a tool to analyze the sequence $(u_{4,h_m})_{m\in \mathbb{N}}$. To do so, we define $N = 3$, $q= 2$, and $p(u_{4,h_m})=\hat{u}_{4,h_m}$, where $\hat{u}_{4}$ is defined by $\hat{u}_{4,h_m} = u_{4,h_m}$ within $\Omega_T$ and $\hat{u}_{4,h_m}=0$ outside of $\Omega_T$. The first condition of Kolmogorov’s compactness criterion is guaranteed by definition of $\hat{u}_{4}$ and the second condition follows from Lemma \ref{lemma:ae}. Thus, to complete the proof, it remains to verify the third condition of Kolmogorov’s theorem. Now, for any $\boldsymbol{y}\in \mathbb{R}^2$ and $\tau \in \mathbb{R}$, the triangle inequality yields
	\begin{align*}
		\big\|\hat{u}_{4,h_m}(\cdot+\boldsymbol{y},\cdot+\tau)-\hat{u}_{4,h_m}(\cdot,\cdot)\big\|_{L^2(\mathbb{R}^{N})}&\leq \big\|\hat{u}_{4,h_m}(\cdot+\boldsymbol{y},\cdot+\tau)-\hat{u}_{4,h_m}(\cdot+\boldsymbol{y},\cdot)\big\|_{L^2(\mathbb{R}^{N})}\\
		&\,\, +\|\hat{u}_{4,h_m}(\cdot+\boldsymbol{y},\cdot)-\hat{u}_{4,h_m}(\cdot,\cdot)\big\|_{L^2(\mathbb{R}^{N})}.
	\end{align*}
	Thanks to the translate estimates given by Lemma \ref{lemma:time-translate} and \ref{lemma:space-translate} we have that $\big\|\hat{u}_{4,h_m}(\cdot+\boldsymbol{y},\cdot+\tau)-\hat{u}_{4,h_m}(\cdot,\cdot)\big\|_{L^2(\mathbb{R}^{N})}\to 0$, as $\boldsymbol{y}\to 0$ and $\tau\to 0$. This guarantees the compactness of the sequence $(u_{4,h_m})_{m}$ in $L^2(\Omega_T)$. Then, by using the Kolmogorov theorem, there exists a subsequence still denoted by $(u_{4,h_m})_{m}$, and a function $u_{4}\in L^2(\Omega_T)$ such that $u_{4,h_m}\longrightarrow u_{4}$ strongly in $L^2(\Omega_T)$ and a.e. in $\Omega_T$. Since this sequence is bounded in $L^{\infty}(\Omega_T)$ this convergence also holds in $L^{p}(\Omega_T),$ for all $p\in [1,+\infty)$. Following the same approach, and using the translation estimates \eqref{eq:time-translate}, \eqref{eq:space-translate} together with the $L^{\infty}$ bound of $u_{i,h_m}$ for $i=1,3,5$, we obtain, after extracting a subsequence, that $u_{i,h_m}\longrightarrow u_{i}$ strongly in $L^p(\Omega_T)$ and almost everywhere in $\Omega_T$, for all $p\in [1,\infty)$ and $i=1,3,5$. This concludes the proof of \textit{(a)}.
	
	We now proceed to demonstrate the second convergence, \textit{(b)}. Due to Lemma \ref{lemma:maxprin} the sequence $(u_{2,h_m})_{m}$ is uniformly bounded in $L^2(\Omega_T)$. Thus, this sequence converges weakly in $L^2(\Omega_T)$, up to a possibly unlabeled subsequence, to a function $u_2$ in $L^2(\Omega_T)$.
	
	We establish now the item \textit{c)}. Let $i\in \{1,3,4,5\}$ fixed. By using the discrete estimate \eqref{eq:ae-bound-2}, we can establish that the sequence $\Big(\nabla_{\mathcal{T}_m} u_{i,h_m}\Big)_m$ is uniformly bounded in $[L^2(\Omega_T)]^2$. Therefore, up to a possibly
	unlabeled subsequence, this sequence converges weakly to a function $p_i^{*}\in [L^2(\Omega_T)]^2$. Then, we identify $\nabla u_i$ by $p_i^{*}$ by using the following convergence result (see, e.g., \cite{chainais2003finite,bendahmane2014convergence})
	\begin{align*}
		\int_0^{T} \int_{\Omega} \Big( \nabla_{\mathcal{T}_m} u_{i,h_m}\cdot \phi_i + u_{i,h_m} \nabla \cdot \phi_i\Big)\, \mathrm{d}\boldsymbol{x}\, \mathrm{d}t\longrightarrow 0, \quad \text{ as }m\to +\infty, \quad \forall \phi_i \in [\mathcal{D}(\Omega_T)]^2.
	\end{align*}
	This ends the proof of \textit{c)}.
	
	Finally, we establish \textit{(d)}, that is, we identify the function $\boldsymbol{u} = (u_1,\dots,u_5)^{\mathrm{T}}$, obtained from the arguments above, as a weak solution of the continuous problem \eqref{eq:model} in the sense of Definition \ref{def:weaksol}. Let us focus on the proof of convergence for $i=2,4$. The convergence of the rest of the equations is similar.
	
	\subsubsection{Convergence of the discrete plaque oligomers equation}\label{sec:conv-u2} 
	We consider $\psi\in \mathcal{D}(\overline{\Omega}\times [0,T))$ and denote by $\psi_K^{n} = \psi(\boldsymbol{x}_K,t_n),$ for all $K\in \mathcal{T}$ and $n\in \{0,\dots, N_m\}$. Multiplying the equation \eqref{eq:fvm-2} by $\Delta t_m \psi_K^{n+1}$ and summing over $K\in \mathcal{T}_m$ and $n\in \{0,\dots,N_m\}$ yields $A_m=B_m+C_m$, where
	\begin{equation}
		\begin{aligned}
			A_m&=\sum_{n=0}^{N_m} \sum_{K\in \mathcal{T}_m}m(K)(u_{2,K}^{n+1}-u_{2,K}^{n}) \psi_K^{n+1},\\
			B_m&=\sum_{n=0}^{N_m} \Delta t_m \sum_{K\in \mathcal{T}_m}m(K) \gamma(u_{4,K}^{n}) u_{1,K}^{n} \psi_K^{n+1},\\
			C_m&=-\tau_p \sum_{n=0}^{N_m} \Delta t_m \sum_{K\in \mathcal{T}_m}m(K) u_{2,K}^{n+1} \psi_K^{n+1}.
		\end{aligned}
	\end{equation}
	Then, by using summation by parts in time and noticing that $\psi_K^{N_m+1} = 0,$ for all $K\in \mathcal{T}_m$ we get
	\begin{equation}
		\begin{aligned}\label{eq:conv-1}
			A_m &= - \sum_{n=0}^{N_m} \sum_{K\in \mathcal{T}_m}m(K) u_{2,K}^{n+1} (\psi_{K}^{n+1}-\psi_K^{n})-\sum_{K\in \mathcal{T}_m}m(K) u_{2,K}^{0} \psi_K^{0}\\
			&=-\iint_{\Omega_T} u_{2,h_m}(\boldsymbol{x},t) \partial_t \psi(\boldsymbol{x}_K,t) \, \mathrm{d}\boldsymbol{x}\, \mathrm{d}t -  \int_{\Omega} u_{2,0}(\boldsymbol{x}) \psi(\boldsymbol{x}_K,0) \, \mathrm{d}\boldsymbol{x}:= -A_m^{(1)}-A_m^{(2)}. 
		\end{aligned}
	\end{equation}
	Then, using the regularity properties of $\partial_t \psi$ we obtain
	\begin{align*}
		\left|A_m^{(1)}-\iint_{\Omega_T} u_{2}(\boldsymbol{x},t) \partial_t \psi(\boldsymbol{x},t) \, \mathrm{d}\boldsymbol{x}\, \mathrm{d}t\right|&\leq  \iint_{\Omega_T}\big|u_{2,h_m}(\boldsymbol{x},t) \partial_t \psi(\boldsymbol{x}_K,t)-u_2(\boldsymbol{x},t)\partial_t \psi(\boldsymbol{x},t)\big| \, \mathrm{d}\boldsymbol{x}\, \mathrm{d}t \\
		& \leq \|u_{2,h_m}\|_{L^{\infty}(\Omega_T)} \iint_{\Omega_T} \big| \partial_t \psi(\boldsymbol{x}_K,t)-\partial_t \psi(\boldsymbol{x},t)\big| \, \mathrm{d}\boldsymbol{x}\, \mathrm{d}t \\
		&\qquad+\iint_{\Omega_T} \big|u_{2,h_m}(\boldsymbol{x},t)\partial_t \psi(\boldsymbol{x},t)-u_2(\boldsymbol{x},t)\partial_t \psi(\boldsymbol{x},t)\big|\, \mathrm{d}\boldsymbol{x}\, \mathrm{d}t\\
		&\leq \iint_{\Omega_T} \big|u_{2,h_m}(\boldsymbol{x},t)\partial_t \psi(\boldsymbol{x},t)-u_2(\boldsymbol{x},t)\partial_t \psi(\boldsymbol{x},t)\big|\, \mathrm{d}\boldsymbol{x}\, \mathrm{d}t\\
		&\quad + C \|u_{2,h_m}\|_{L^{\infty}(\Omega_T)} h_m,
	\end{align*}
	for some constant $C>0$. Thus, using the regularity of the function $\psi$, the boundedness of the sequence $(u_{2,h_m})_m$ in $L^{\infty}(\Omega_T)$, and the weak convergence in $L^2(\Omega_T)$ of $u_{2,h_m}$ towards $u_2$  it follows that
	\begin{equation}\label{eq:conv-2}
		A_{m}^{(1)} \longrightarrow \iint_{\Omega_T} u_{2}(\boldsymbol{x},t) \partial_t \psi(\boldsymbol{x},t) \, \mathrm{d}\boldsymbol{x}\, \mathrm{d}t, \quad \text{ as }m\to +\infty.
	\end{equation}
	In the same way, by using the Lipschitz continuity of the function $\psi$, we observe that,
	\begin{align*}
		\left|A_m^{(2)}-\int_{\Omega} u_{2,0}(\boldsymbol{x}) \psi(\boldsymbol{x},0) \, \mathrm{d}\boldsymbol{x}\right|&\leq \int_{\Omega} |u_{2,0}(\boldsymbol{x})||\psi(\boldsymbol{x}_K,0)-\psi(\boldsymbol{x},0)| \,  \mathrm{d}\boldsymbol{x}\\
		& \leq \|u_{2,0}\|_{L^{\infty}(\Omega_T)}\int_{\Omega}|\psi(\boldsymbol{x}_K,0)-\psi(\boldsymbol{x},0)|\,  \mathrm{d}\boldsymbol{x}\\
		&\leq C \|u_{2,0}\|_{L^{\infty}(\Omega_T)} h_m,
	\end{align*}
	for some constant $C>0$. This yields
	\begin{equation}\label{eq:conv-3}
		A_{m}^{(2)} \longrightarrow \int_{\Omega} u_{2,0}(\boldsymbol{x}) \psi(\boldsymbol{x},0) \, \mathrm{d}\boldsymbol{x},\quad \text{ as }m\to +\infty.
	\end{equation}
	From \eqref{eq:conv-1}, \eqref{eq:conv-2}, and \eqref{eq:conv-3} it follows that
	\begin{equation}\label{eq:conv-am}
		A_m \longrightarrow -\iint_{\Omega_T} u_{2}(\boldsymbol{x},t) \partial_t \psi(\boldsymbol{x},t) \, \mathrm{d}\boldsymbol{x}\, \mathrm{d}t-\int_{\Omega} u_{2,0}(\boldsymbol{x}) \psi(\boldsymbol{x},0) \, \mathrm{d}\boldsymbol{x},\quad \text{ as }m\to +\infty.
	\end{equation}
	Now, we focus on $B_m$. We define $\psi_{h_m}(\boldsymbol{x},t) = \psi_{K}^{n}$, for all $\boldsymbol{x}\in K$ and $t\in [t_n,t_{n+1})$ and we observe that
	\begin{align*}
		B_m &= \iint_{\Omega_T} \gamma\big(u_{4,h_m}(\boldsymbol{x},t-\Delta t_m)\big) u_{1,h_m}(\boldsymbol{x},t-\Delta t_m) \psi_{h_m}(\boldsymbol{x},t) \, \mathrm{d}\boldsymbol{x}\, \mathrm{d}t,
	\end{align*}
	where we have defined $u_{i,h_m}(\boldsymbol{x},t-\Delta t_m)=u_{i,h_m}(\boldsymbol{x},t)$, for all $t\in [0,\Delta t_m]$, $i=1,4$. Then, by employing triangle and Cauchy-Schwarz inequalities, Lemma \ref{lemma:maxprin}, and regularity of the function $\psi$ we arrive at
	{\begin{equation}\label{eq:conv-4}
			\begin{aligned}
				\left|B_m \right.&\left.- \iint_{\Omega_T} \gamma(u_4(\boldsymbol{x},t))u_1(\boldsymbol{x},t) \psi(\boldsymbol{x},t)\, \mathrm{d}\boldsymbol{x}\, \mathrm{d}t\right|\\
				&\leq \iint_{\Omega_T} \big|\gamma\big(u_{4,h_m}(\boldsymbol{x},t-\Delta t_m)\big) \big|\big|u_{1,h_m}(\boldsymbol{x},t-\Delta t_m)\psi_{h_m}(\boldsymbol{x},t) - u_1(\boldsymbol{x},t) \psi(\boldsymbol{x},t)\big|\mathrm{d}\boldsymbol{x}\, \mathrm{d}t\\
				&\quad +  \iint_{\Omega_T}  \big|\gamma\big(u_{4,h_m}(\boldsymbol{x},t-\Delta t_m)-\gamma\big(u_{4}(\boldsymbol{x},t)\big)\big| |u_1(\boldsymbol{x},t) \psi(\boldsymbol{x},t)|\,\mathrm{d}\boldsymbol{x}\, \mathrm{d}t\\
				&\leq \|\gamma(u_{4,h_m})\|_{L^{\infty}(\Omega)} \|\psi_{h_m}\|_{L^2(\Omega_T)}\left(\iint_{\Omega_T}|u_{1,h_m}(\boldsymbol{x},t-\Delta t_m)-u_1(\boldsymbol{x},t)|^2\, \mathrm{d}\boldsymbol{x}\, \mathrm{d}t \right)^{1/2}\\
				&\quad + \|\gamma(u_{4,h_m})\|_{L^{\infty}(\Omega)} \|u_1\|_{L^2(\Omega_T)} \|\psi_{h_m}-\psi\|_{L^2(\Omega_T)}\\
				&\quad +\|u_1\|_{L^2(\Omega)}\|\psi\|_{L^{\infty}(\Omega)} \left(\iint_{\Omega_T}|\gamma\big(u_{4,h_m}(\boldsymbol{x},t-\Delta t_m)\big)-\gamma\big(u_4(\boldsymbol{x},t)\big)|^2\, \mathrm{d}\boldsymbol{x}\, \mathrm{d}t \right)^{1/2}.
			\end{aligned}
	\end{equation}}
	Let us introduce the terms
	\begin{align*}
		B_m^{(1)}&:=\left(\iint_{\Omega_T}|u_{1,h_m}(\boldsymbol{x},t-\Delta t_m)-u_1(\boldsymbol{x},t)|^2\, \mathrm{d}\boldsymbol{x}\, \mathrm{d}t \right)^{1/2},\\
		B_m^{(2)}&:= \left(\iint_{\Omega_T}|\gamma\big(u_{4,h_m}(\boldsymbol{x},t-\Delta t_m)\big)-\gamma\big(u_4(\boldsymbol{x},t)\big)|^2\, \mathrm{d}\boldsymbol{x}\, \mathrm{d}t \right)^{1/2}.
	\end{align*}
	By using the time-translate estimate \eqref{eq:time-translate} we get
	\begin{equation}\label{eq:conv-5}
		\begin{aligned}
			B_m^{(1)}&\leq \left(\iint_{\Omega_T}|u_{1,h_m}(\boldsymbol{x},t-\Delta t_m)-u_{1,h_m}(\boldsymbol{x},t)|^2\, \mathrm{d}\boldsymbol{x}\, \mathrm{d}t \right)^{1/2}\\
			&\quad +\left(\iint_{\Omega_T}|u_{1,h_m}(\boldsymbol{x},t)-u_1(\boldsymbol{x},t)|^2\, \mathrm{d}\boldsymbol{x}\, \mathrm{d}t \right)^{1/2} \leq C \sqrt{\Delta t_m}+\|u_{1,h_m}-u_1\|_{L^2(\Omega)}.
		\end{aligned}
	\end{equation}
	In the same way, by employing the time-translate estimate \eqref{eq:time-translate} and the Lipschitz continuity property \eqref{eq:prop-gamma} of the function $\gamma$, one gets
	\begin{equation}
		\begin{aligned}\label{eq:conv-6}
			B_m^{(2)}&\leq \left(\iint_{\Omega_T}|\gamma\big(u_{4,h_m}(\boldsymbol{x},t-\Delta t_m)\big)-\gamma\big(u_{4,h_m}(\boldsymbol{x},t)\big)|^2\, \mathrm{d}\boldsymbol{x}\, \mathrm{d}t \right)^{1/2}\\
			&\quad +\left(\iint_{\Omega_T}|\gamma\big(u_{4,h_m}(\boldsymbol{x},t)\big)-\gamma\big(u_4(\boldsymbol{x},t)\big)|^2\, \mathrm{d}\boldsymbol{x}\, \mathrm{d}t \right)^{1/2} \leq \gamma_1 C \sqrt{\Delta t_m}+\gamma_1\|u_{4,h_m}-u_4\|_{L^2(\Omega_T)}
		\end{aligned}
	\end{equation}
	Therefore, using the regularity of $\psi$, together with the boundedness of $(\gamma(u_{4,h_m}))_m$ in $L^{\infty}(\Omega_T)$, the boundedness of $u_1$ in $L^2(\Omega_T)$, and the strong convergence of $u_{i,h_m}$ to $u_i$ in $L^2(\Omega_T)$ for $i=1,4$, we deduce from \eqref{eq:conv-4}, \eqref{eq:conv-5}, and \eqref{eq:conv-6} that
	\begin{equation}\label{eq:conv-bm}
		B_m \longrightarrow \iint_{\Omega_T} \gamma\big(u_4(\boldsymbol{x},t)\big)u_1(\boldsymbol{x},t) \psi(\boldsymbol{x},t)  \, \mathrm{d}\boldsymbol{x}\, \mathrm{d}t,\quad \text{ as }m\to +\infty.
	\end{equation}
	Next, for the term $C_m$ we observe that
	\begin{equation*}
		C_m = -\tau_p\iint_{\Omega_T} u_{2,h_m}(\boldsymbol{x},t) \psi_{h_m}(\boldsymbol{x},t) \mathrm{d}\boldsymbol{x}\, \mathrm{d}t.
	\end{equation*}
	By following the estimates for $B_m$ and employing the regularity properties of $\psi$, together with the weak convergence of the sequence $(u_{2,h_m})_m$ in $L^2(\Omega_T)$, we can deduce that
	\begin{equation}\label{eq:conv-cm}
		C_m \longrightarrow -\tau_p \iint_{\Omega_T} u_2(\boldsymbol{x},t) \psi(\boldsymbol{x},t)  \, \mathrm{d}\boldsymbol{x}\, \mathrm{d}t,\quad \text{ as }m\to +\infty.
	\end{equation}
	From \eqref{eq:conv-am}, \eqref{eq:conv-bm}, and \eqref{eq:conv-cm} it follows that $u_2$ satisfies \eqref{eq:weak-u2}.
	\subsubsection{Convergence of the discrete microglial cells equation}
	In the same way, let $\psi\in \mathcal{D}(\overline{\Omega}\times [0,T))$ and denote by $\psi_K^{n} = \psi(\boldsymbol{x}_K,t_n),$ for all $K\in \mathcal{T}$ and $n\in \{0,\dots, N_m\}$. Multiplying the equation \eqref{eq:fvm-4} by $\Delta t_m \psi_K^{n+1}$ and summing over $K\in \mathcal{T}_m$ and $n\in \{0,\dots,N_m\}$ we get $A_m'+D_m'+C_m' = D_m'$, where
	\begin{equation}
		\begin{aligned}
			A_m'&=\sum_{n=0}^{N_m} \sum_{K\in \mathcal{T}_m}m(K)(u_{4,K}^{n+1}-u_{4,K}^{n}) \psi_K^{n+1},\\
			B_m'&=-d_4 \sum_{n=0}^{N_m} \Delta t_m \sum_{K\in \mathcal{T}_m}\sum_{L\in \mathcal{N}(K)}\tau_{K|L} (u_{4,L}^{n+1}-u_{4,K}^{n+1})  \psi_K^{n+1},\\
			C_m'&=\sum_{n=0}^{N_m} \Delta t_m\sum_{L\in \mathcal{N}(K)} \tau_{K|L} \mathcal{G}\big(u_{4,K}^{n+1},u_{4,L}^{n+1};\delta u_{1,KL}^{n+1}\big)  \psi_K^{n+1},\\
			D_m'&=\sum_{n=0}^{N_m} \Delta t_m\sum_{K\in \mathcal{T}_m} m(K) \bar{F}_{4}(\boldsymbol{u}_{K}^{n},\boldsymbol{u}_{K}^{n+1})\psi_K^{n+1}.
		\end{aligned}
	\end{equation}
	Regarding the time-evolution term $A_m'$, we apply the same approach as in Section \ref{sec:conv-u2}, which yields the following convergence
	\begin{equation}\label{eq:conv-amprima}
		A_m' \longrightarrow -\iint_{\Omega_T} u_{4}(\boldsymbol{x},t) \partial_t \psi(\boldsymbol{x},t) \, \mathrm{d}\boldsymbol{x}\, \mathrm{d}t-\int_{\Omega} u_{4,0}(\boldsymbol{x}) \psi(\boldsymbol{x},0) \, \mathrm{d}\boldsymbol{x},\quad \text{ as }m\to +\infty.
	\end{equation}
	In addition, by using the Lipschitz continuity properties of the function $\bar{\boldsymbol{F}}_4$ and the same arguments as in Section \ref{sec:conv-u2} for the terms $B_m$ and $C_m$, we can prove the following convergence
	\begin{equation}\label{eq:conv-dmprima}
		D_m' \longrightarrow \iint_{\Omega_T} {F}_{4}(\boldsymbol{u}(\boldsymbol{x},t)) \psi(\boldsymbol{x},t) \, \mathrm{d}\boldsymbol{x} \, \mathrm{d}t, \quad \text{ as }m\to +\infty.
	\end{equation}
	Therefore, we only need to handle the diffusion term $B_m'$ and the convection term $C_m'$. For the diffusive part we integrate by parts and having into account that $\tau_{K|L}=2m(T_{KL})/d_{K|L}^2$ we use the definition of the discrete gradient \eqref{eq:grad-disc} to obtain 
	\begin{equation}\label{eq:conv-7}
		\begin{aligned}
			B_m' &= d_4 \sum_{n=0}^{N_m} \Delta t_m \sum_{K\in \mathcal{T}_m}\sum_{L\in \mathcal{N}(K)}\tau_{K|L} (u_{4,L}^{n+1}-u_{4,K}^{n+1}) (\psi_L^{n+1}-\psi_K^{n+1})\\
			&=\dfrac{d_4}{2} \sum_{n=0}^{N_m} \Delta t_m \sum_{K\in \mathcal{T}_m}\sum_{L\in \mathcal{N}(K)} m(T_{KL}) \left[2\left(\dfrac{u_{4,L}^{n+1}-u_{4,K}^{n+1}}{d_{K|L}}\right)\cdot \boldsymbol{n}_{K|L}\right] \left[2\left(\dfrac{\psi_{L}^{n+1}-\psi_{K}^{n+1}}{d_{K|L}}\right)\cdot \boldsymbol{n}_{K|L}\right]\\
			&=\dfrac{d_4}{2} \sum_{n=0}^{N_m} \Delta t_m  \int_{\Omega} \nabla_{\mathcal{T}_m} u_{4,\mathcal{T}_m}^{n+1}\cdot \nabla_{\mathcal{T}_m} \psi_{\mathcal{T}_m}^{n+1}\, \mathrm{d}\boldsymbol{x}.
		\end{aligned}
	\end{equation}
	Now, let us define $\boldsymbol{x}_{KL} = \theta \boldsymbol{x}_K+(1-\theta)\boldsymbol{x}_L$, for $0<\theta<1$ some point of the segment $\llbracket\boldsymbol{x}_K,\boldsymbol{x}_L\rrbracket$ such that
	\begin{equation*}
		\psi(\boldsymbol{x}_{L},t_{n+1})-\psi(\boldsymbol{x}_{K},t_{n+1}) = d_{K|L} \nabla \psi (\boldsymbol{x}_{KL},t_{n+1})\cdot \boldsymbol{n}_{K|L}.
	\end{equation*}
	Then we have that
	\begin{equation*}
		\nabla_{\mathcal{T}_m} \psi_{\mathcal{T}_m}^{n+1} = 2\left(\dfrac{\psi_{L}^{n+1}-\psi_K^{n+1}}{d_{K|L}}\right)\cdot \boldsymbol{n}_{K|L} = 2\left(\dfrac{\psi(\boldsymbol{x}_L,t_{n+1})-\psi(\boldsymbol{x}_K,t_{n+1})}{d_{K|L}}\right)\cdot \boldsymbol{n}_{K|L} = 2\nabla \psi(\boldsymbol{x}_{KL},t_{n+1}).
	\end{equation*}
	Let us define $(\nabla \psi)_{h_m}(\boldsymbol{x},t) = \nabla \psi(\boldsymbol{x}_{KL},t_{n+1}),$ for all $(\boldsymbol{x},t)\in T_{KL}\times(t_n,t_{n+1})$. Then from \eqref{eq:conv-7} it follows that
	\begin{equation}
		B_m' = d_4 \iint_{\Omega_T} \nabla_{\mathcal{T}_m} u_{4,\mathcal{T}_m}(\boldsymbol{x},t)\cdot (\nabla \psi)_{h_m}(\boldsymbol{x},t)\, \mathrm{d}\boldsymbol{x}.
	\end{equation}
	Now, by employing the weak convergence \textit{c)} of discrete gradient of the microglial cells $u_{4,h_n}$ and the regularity properties of the test function $\psi$ allow us to deduce that
	\begin{equation}\label{eq:conv-bmprima}
		B_m'\longrightarrow d_4 \iint_{\Omega_T} \nabla u_4(\boldsymbol{x},t) \cdot \nabla \psi(\boldsymbol{x},t) \, \mathrm{d}\boldsymbol{x}\, \mathrm{d} t.
	\end{equation}
	Finally, we focus on the convergence of the convective term. We adopt similar ideas to \cite{andreianov2011finite,alotaibi2024computational} so we first prove that $|C_m-C_m^{*}|\to 0$, as $m\to +\infty$, where
	\begin{equation*}
		C_m^{*} = -\iint_{\Omega_T} \chi(\underbar{$u$}_{4,h_m}(\boldsymbol{x},t)) \nabla_{\mathcal{T}_m} u_{1,h_m}(\boldsymbol{x},t) \cdot (\nabla \psi)_{h_m}(\boldsymbol{x},t)\, \mathrm{d}\boldsymbol{x}\, \mathrm{d}t,
	\end{equation*}
	where $\underbar{$u$}_{4,h_m}$ is defined by $\underbar{$u$}_{4,h_m}(\boldsymbol{x},t) := u_{4,KL}^{n+1} := \min(u_{4,K}^{n+1},u_{4,L}^{n+1}),$ for $(\boldsymbol{x},t)\in T_{KL}\times(t_n,t_{n+1}).$ From the assumptions (H2)-(H4) of the numerical flux function $\mathcal{G}$, we get
	\begin{equation*}
		\begin{aligned}
			\Big|\mathcal{G}\big(u_{4,K}^{n+1},u_{4,L}^{n+1};\delta u_{1,KL}^{n+1}\big) - \chi(u_{4,KL}^{n+1})\delta u_{1,KL}^{n+1}\Big| &=  \Big|\mathcal{G}\big(u_{4,K}^{n+1},u_{4,L}^{n+1};\delta u_{1,KL}^{n+1}\big) - \mathcal{G}\big(u_{4,K}^{n+1},u_{4,K}^{n+1};\delta u_{1,KL}^{n+1}\big)\Big|\\
			&\leq C\big|\delta u_{1,KL}^{n+1}\big| \big|u_{4,L}^{n+1}-u_{4,K}^{n+1}\big|.
		\end{aligned}
	\end{equation*}
	Define now  $\bar{u}_{4,h_m}$ is defined by $\bar{u}_{4,h_m}(\boldsymbol{x},t) := u_{4,KL}^{n+1} := \max(u_{4,K}^{n+1},u_{4,L}^{n+1}),$ for $(\boldsymbol{x},t)\in T_{KL}\times(t_n,t_{n+1}).$ Then, from the inequality above one gets
	\begin{equation*}
		|C_m-C_m^{*}|\leq C \iint_{\Omega_T} \big| \bar{u}_{4,h_m}(\boldsymbol{x},t)-\underbar{$u$}_{4,h_m}(\boldsymbol{x},t)\big| \big|\nabla_{\mathcal{T}_m} u_{1,h_m}(\boldsymbol{x},t) \cdot (\nabla \psi)_{h_m}(\boldsymbol{x},t)\big|\, \mathrm{d}\boldsymbol{x}\, \mathrm{d}t. 
	\end{equation*}
	Thanks to the apriori estimate \eqref{eq:ae-bound-2} we notice that $| \bar{u}_{4,h_m}-\underbar{$u$}_{4,h_m}|\longrightarrow 0$, a.e. in $\Omega_T$.  By using Cauchy-Schwarz inequality with the uniform bound for $\nabla_{\mathcal{T}_m} u_{1,h_m}$ given by \eqref{eq:ae-bound-2}, it follows that $|C_m-C_m^{*}|\longrightarrow 0$ as $m\to +\infty$. In addition, we have that $\underbar{$u$}_{4,h_m}\leq u_{4,h_m}\leq \bar{u}_{4,h_m}$ and $u_{4,h_m}\longrightarrow u$ a.e. on $\Omega_T$. Then, by the continuity of the sensitivity function $\chi$, we obtain that $\chi\big(\underbar{$u$}_{4,h_m}\big) \longrightarrow \chi(u)$ a.e. on $\Omega_T$ and in $L^p(\Omega_T)$, for $1\leq p<+\infty$. Consequently, by the weak convergence \textit{c)}, we get that
	\begin{equation}\label{eq:conv-cmprima}
		C_m\longrightarrow -\iint_{\Omega_T} \chi(u_4(\boldsymbol{x},t)) \nabla u_1(\boldsymbol{x},t) \cdot \nabla \psi(\boldsymbol{x},t)\, \mathrm{d}\boldsymbol{x}\, \mathrm{d}t.
	\end{equation}
	From \eqref{eq:conv-amprima},\eqref{eq:conv-dmprima}, \eqref{eq:conv-bmprima}, and \eqref{eq:conv-cmprima}, it follows that $u_4$ satisfies \eqref{eq:weak-u4}. This concludes the proof.
\end{proof}
\section{Dynamical consistency with the SH model}\label{sec:dc}
In this section, we show that the FV scheme \eqref{eq:rhs-disc}–\eqref{eq:fvm-5} is dynamically consistent with the SH model \eqref{eq:modelhomo}, in the sense that key properties of the continuous system, such as boundedness, equilibrium points, and the local stability of the disease-free equilibrium, are preserved by the discrete scheme.
\subsection{Nonstandard discretization}
To describe this discretization approach let us consider a general Cauchy problem of the form
\begin{equation}\label{eq:general-ode}
	\begin{aligned}
		\dfrac{d\boldsymbol{u}}{dt}&=\boldsymbol{f}(\boldsymbol{u}), \textrm{ in }[0,T]\\
		\boldsymbol{u}(t_0)&=\boldsymbol{u}^0,
	\end{aligned}
\end{equation}
where $T$ is a positive real number, $\boldsymbol{u}=(u_1,u_2, \dots, u_N)^{\mathrm{T}}:[0,T]\to \mathbb{R}^N$, and the function $\boldsymbol{f}=(f_1,f_2,\dots,f_n)^{\mathrm{T}}:\mathbb{R}^N\to \mathbb{R}^N$ is differentiable at $\boldsymbol{u}^0\in \mathbb{R}^n$. First, we employ the discretization of the temporal domain $[0,T]$ described in Section \ref{sec:spacetimedisc}, i.e. we set $t_n=n \Delta t $, where $\Delta t=T/M$ is the size step, and $M$ is a fixed positive integer. A numerical scheme with step size $\Delta t$ that approximate the solution $\boldsymbol{u}(t_n)$ of the system \eqref{eq:general-ode}
can be write as
\begin{equation}\label{eq:general-approx}
	\boldsymbol{\mathcal{D}}_{\Delta t}(\boldsymbol{u}^n)=\boldsymbol{F}_{\Delta t}(\boldsymbol{f};\boldsymbol{u}^n),
\end{equation}
where, $\boldsymbol{\mathcal{D}}_{\Delta t}(\boldsymbol{u}^n)\approx \dfrac{d\boldsymbol{u}}{dt}\Big|_{t=t_n}$, $\boldsymbol{u}^n\approx \boldsymbol{u}(t_n)$, and $\boldsymbol{F}_{\Delta t}(\boldsymbol{f};\boldsymbol{u}^n)$ approximates the right side of the system (\ref{eq:general-ode}).
\begin{definition}\label{def:NSFD}
	The scheme give by (\ref{eq:general-approx}) is called a Non-Standard Finite Difference method (NSFD) if at least one of the following conditions is satisfied,
	\begin{enumerate}
		\item $\boldsymbol{\mathcal{D}}_{\Delta t}(\boldsymbol{u}^n)=\dfrac{\boldsymbol{u}^{n+1}-\boldsymbol{u}^n}{\varphi(\Delta t)}$, where $\varphi$ is a non-negative real valued function which satisfies, \[\varphi(\Delta t)=\Delta t+\mathcal{O}(\Delta t^{2}).\] Some examples of this kind of functions are: 
		\begin{equation}\label{eq6}
			\varphi (\Delta t) = \frac{1-e^{-\lambda \Delta t}} {\lambda}, \, \lambda> 0, \, \, \varphi (\Delta t) = e^{\Delta t} -1, \text { see \cite{mickens1993nonstandard,mickens2000applications}.}
		\end{equation}
		\item $\boldsymbol{F}_{\Delta t}(\boldsymbol{f};\boldsymbol{u}^n)=\boldsymbol{G}(\boldsymbol{u}^n,\boldsymbol{u}^{n+1},\Delta t)$, where $\boldsymbol{G}(\boldsymbol{u}^n,\boldsymbol{u}^{n+1},\Delta t)$ is a non-local approximation of the right side of the system (\ref{eq:general-ode}).
	\end{enumerate}
\end{definition}
Now, considering the FV scheme \eqref{eq:rhs-disc}-\eqref{eq:fvm-5} but neglecting the discretization of the spatial operators, we can approximate the system of equations \eqref{eq:modelhomo} as 
\begin{equation}\label{eq:nsfd1}
	\begin{aligned}
		\frac{u_1^{n+1}-u_1^n}{\Delta t} &= r_1 (u_3^n)^2-\gamma(u_4^n) u_1^{n+1}-\tau_0 u_1^{n+1}, \\
		\frac{u_2^{n+1}-u_2^n}{\Delta t}&=\gamma(u_4^n) u_1^{n}-\tau_p u_2^{n+1}\\
		\frac{u_3^{n+1}-u_3^n}{\Delta t} &=\dfrac{\tau_S}{1+C (u_1^n)^{\nu}}u_5^n-du_3^{n+1}-r_2u_1^nu_3^{n+1}-r_1 u_3^n u_3^{n+1},\\
		\frac{u_4^{n+1}-u_4^n}{\Delta t}&=\dfrac{\alpha_1 u_1^n}{1+\alpha_2 u_1^n}(\hat{m}-u_4^{n+1})u_4^{n}-\sigma u_4^{n+1}+\lambda_M,\\
		\frac{u_5^{n+1}-u_5^n}{\Delta t}&= \dfrac{\tau_1 u_1^n}{1+\tau_2 u_1^n}u_4^n-\tau_3 u_5^{n+1}.
	\end{aligned}
\end{equation}
We notice that this discretization coincides with the traditional NSFD approach described in Definition \ref{def:NSFD} with $\varphi(\Delta t) = \Delta t$, so from now on we will refer to the scheme as a ``NSFD method''.
\subsection{Properties of the NSFD method}
We first demonstrate that the NSFD method preserves the invariance of the rectangle $\mathcal{R}$. This is stated in the following proposition.
\begin{prop}\label{lemma:nsfd-inv}
	Let $\hat{m}\geq \frac{\lambda_M}{\sigma}$, then the NSFD method \eqref{eq:nsfd1} is dynamically consistent with respect to the invariance of the rectangle $\mathcal{R}$  for all the values of the step size $\Delta t$, that is $\boldsymbol{u}^n\in \mathcal{R}$ for all $n\geq 1$ whenever $\boldsymbol{u}^0\in \mathcal{R}$.
\end{prop}
\begin{proof}
	We proceed by induction on $n$. By hypothesis $\boldsymbol{u}^0\in \mathcal{R}$, so let us assume that $\boldsymbol{u}^n\in \mathcal{R}$. Rearranging the scheme \eqref{eq:nsfd1} in explicit form we obtain that,
	\begin{equation}\label{eq:nsfd2}
		\begin{aligned}
			u_1^{n+1}&= \dfrac{u_1^n+\Delta t r_1 (u_3^n)^2}{1+\Delta t(\gamma(u_4^n)+\tau_0)},\,\, u_2^{n+1}=\dfrac{u_2^n+\Delta t\gamma(u_4^n) u_1^{n}}{1+\Delta t \tau_p }, \,\, u_3^{n+1}= \dfrac{u_3^{n}+ \dfrac{\Delta t\tau_S}{1+C (u_1^n)^{\nu}}u_5^n}{1+\Delta t(d+r_2u_1^n+r_1 u_3^n)},\\
			u_4^{n+1}&=\dfrac{u_4^n+\dfrac{\alpha_1 u_1^n }{1+\alpha_2 u_1^n}\Delta t \hat{m} u_4^{n}+\Delta t\lambda_M }{1+\Delta t\left(\dfrac{\alpha_1 u_1^n}{1+\alpha_2 u_1^n} u_4^n+\sigma \right)}, \, \, u_5^{n+1}=\dfrac{u_5^n+\Delta t \dfrac{\tau_1 u_1^n}{1+\tau_2 u_1^n }u_4^n}{1+\Delta t \tau_3}.
		\end{aligned}
	\end{equation}
	From these equations and for the fact that $\gamma(s)>0$, for all $s\geq 0$, it is clear that $u_{i}^{n+1}\geq 0$, for $i=1,\dots,5$. We need to prove now that $u_{i}^{n+1}\leq \beta_i$, for $i=1,\dots,5$. Let us focus first on $u_4^{n+1}$. From \eqref{eq:nsfd2} and the hypothesis $\hat{m}\geq \frac{\lambda_M}{\sigma}$ it follows that
	\begin{align*}
		u_{4}^{n+1} \leq \dfrac{u_4^n+\dfrac{\alpha_1 u_1^n }{1+\alpha_2 u_1^n}\Delta t \hat{m} u_4^{n}+\Delta t\sigma \hat{m} }{1+\Delta t\left(\dfrac{\alpha_1 u_1^n}{1+\alpha_2 u_1^n} u_4^n+\sigma \right)} = \dfrac{\left(1+\dfrac{\alpha_1 u_1^n }{1+\alpha_2 u_1^n}\Delta t \hat{m} \right)u_4^n+\Delta t\sigma \hat{m} }{1+\Delta t\left(\dfrac{\alpha_1 u_1^n}{1+\alpha_2 u_1^n} u_4^n+\sigma \right)}= f(u_4^n),
	\end{align*}
	where $f:[0,+\infty)\longrightarrow (0,+\infty)$ is the increasing function \eqref{eq:funf} defined in the proof of Lemma \ref{lemma:maxprin}. Now, by the induction hypothesis we know that $u_{4}^{n}\leq \hat{m}$, then $u_4^{n+1}\leq f(u_{4}^n)\leq f(\hat{m})=\hat{m}$. In this way we prove that $u_4^{n+1}\leq \beta_4$. Now, from \eqref{eq:rel-beta} and \eqref{eq:nsfd2}, the induction hypothesis $\boldsymbol{u}^n\in \mathcal{R}$ and the fact that $\gamma_{\min}\leq \gamma(s)\leq \gamma_{\max}$, for all $s\geq 0$ we see  that,
	\begin{align*}
		u_1^{n+1}&\leq \dfrac{\beta_1+\Delta t r_1 \beta_3^2}{1+\Delta t (\gamma_{\min}+\tau_0)} = \dfrac{\beta_1+\Delta t (\gamma_{\min}+\tau_0) \beta_1}{1+\Delta t (\gamma_{\min}+\tau_0)}=\beta_1,\\
		u_2^{n+1}&\leq \dfrac{\beta_2+\Delta t\gamma_{\max} \beta_1}{1+\Delta t \tau_p } = \dfrac{\beta_2+\Delta t\tau_p \beta_2}{1+\Delta t \tau_p} = \beta_2,\\
		u_3^{n+1}&\leq \dfrac{\beta_3+ \Delta t\tau_S \beta_5}{1+\Delta t d} = \dfrac{\beta_3+ \Delta t d \beta_3}{1+\Delta t d} = \beta_3,\\
		u_5^{n+1}&\leq \dfrac{\beta_5+\Delta t \dfrac{\tau_1}{\tau_2 }\beta_4}{1+\Delta t \tau_3} = \dfrac{\beta_5+\Delta t \tau_3\beta_5}{1+\Delta t \tau_3}=\beta_5.
	\end{align*}
\end{proof}
Next, we show that the sets of equilibria described in Proposition~\ref{prop:equi} are maintained by the discrete NSFD method. To this end, we set $\gamma(u)=\gamma_0$, for all $u$ and $\tau_0=0$. The result reads as follows.
\begin{prop}
	The NSFD method \eqref{eq:nsfd1} is dynamically consistent
	with respect to the equilibrium points of the model \eqref{eq:model} for all the values of the step size $\Delta t$.
\end{prop}
\begin{proof}
	In order to find the equilibrium points of the NSFD scheme \eqref{eq:nsfd1} we need to solve the system $\boldsymbol{u}^{n+1} = \boldsymbol{u}^n$, i.e
	$u_i^{n+1}=u_i^{n}$, for all $i=1,\dots,5$. We start by expressing the system \eqref{eq:nsfd2} in the form
	\begin{equation}\label{eq:nsfd3}
		\begin{aligned}
			u_i^{n+1}&=u_i^{n}+\Delta t G_i(\boldsymbol{u}^n) F_i(\boldsymbol{u}^n), \quad i=1,\dots,5,
		\end{aligned}
	\end{equation}
	where 
	\begin{equation}
		\begin{aligned}
			G_1(\boldsymbol{u}^n)&= \dfrac{1 }{1+\Delta t \gamma_0}, \, \, G_2(\boldsymbol{u}^n)=\dfrac{1}{1+\Delta t \tau_p }, \, \, G_3(\boldsymbol{u}^n)= \dfrac{1}{1+\Delta t(d+r_2u_1^n+r_1 u_3^n)},\\
			G_4(\boldsymbol{u}^n)&=\dfrac{1 }{1+\Delta t\left(\dfrac{\alpha_1 u_1^n}{1+\alpha_2 u_1^n} u_4^n+\sigma \right)}, \, \, G_5(\boldsymbol{u}^n)=\dfrac{1}{1+\Delta t \tau_3}.
		\end{aligned}
	\end{equation}
	By using \eqref{eq:nsfd3} we get that $G_i(\boldsymbol{u}^n) F_i(\boldsymbol{u}^n)=0$. Being $G_i$ positive functions, it follows that $F_i(\boldsymbol{u}^n) =0$, for $i=1,\dots,5$. Therefore, the NSFD system \eqref{eq:nsfd2} and the continuous-time model \eqref{eq:modelhomo} have the same sets of equilibria.
\end{proof}
We now analyze the local stability of the disease-free equilibrium point $\mathcal{E}_0$ for the NSFD method. We state some previous lemmas.
\begin{lemma}\label{lemma:jacobian}
	Let $\mathcal{E} \in \mathbb{R}^5$ be an equilibrium point, and $\mathcal{J}^{D}(\mathcal{E})$ be the Jacobian matrix evaluated at $\mathcal{E}$ for the discrete system. Then, the following identity holds
	\begin{equation}\label{eq:jac-dis}
		\mathcal{J}_{ij}^{D}(\mathcal{E}) = \delta_{ij}+G_i(\mathcal{E})	\mathcal{J}_{ij}^{C}(\mathcal{E}),
	\end{equation}
	for all $i,j$, where $\mathcal{J}^{C}(\mathcal{E})$ is the Jacobian matrix of the continuous problem evaluated at $\mathcal{E}$.
\end{lemma}
\begin{proof}
	If $\mathcal{E}$ is an equilibrium point, then $F_i(\mathcal{E})=0$, for $i=1,\dots,5$. So, from equations \eqref{eq:nsfd3} we obtain
	\begin{equation*}\label{eq:jac-dis-2}
		\mathcal{J}_{ij}^{D}(\mathcal{E}) = \delta_{ij}+G_i(\mathcal{E})\dfrac{\partial F_i}{\partial u_j} (\mathcal{E})+\dfrac{\partial G_i}{\partial u_j} (\mathcal{E})F_i(\mathcal{E}) = \delta_{ij}+G_i(\mathcal{E})	\mathcal{J}_{ij}^{C}(\mathcal{E}).
	\end{equation*}
	This concludes the proof.
\end{proof}

\begin{lemma}\label{lemma:disc-stability}
	Consider the non-linear system $X_{t+1} = \Psi(X_t)$, where $\Psi:\mathbb{R}^n \to \mathbb{R}^n$ is a $C^1$-diffeomorphism with a fixed point, $X_0$. Then a steady-state equilibrium, $X_0$, is locally asymptotically stable if and only if the moduli of all eigenvalues of the Jacobian matrix, $\mathcal{J}^{D}(X_0)$, are smaller than one.
\end{lemma}
The following result is obtained, in agreement with Proposition~\ref{prop:sta-dis-free}. 
\begin{prop}
	The NSFD method \eqref{eq:nsfd1} is dynamically consistent with respect to the stability of the disease-free equilibrium $\mathcal{E}_0=(0,0,0,\lambda_M/\sigma,0)$ of the model \eqref{eq:modelhomo}, that is, $\mathcal{E}_0$ is locally asymptotically stable for every choice of positive parameters and for all step size values $\Delta t$.
\end{prop}
\begin{proof}
	The jacobian matrix of the continuous system \eqref{eq:modelhomo} at $\mathcal{E}_0$ is
	\begin{equation*}
		\mathcal{J}^{C}(\mathcal{E}_0) = 
		\begin{pmatrix}
			-\gamma_0&0&0&0&0\\
			\gamma_0&-\tau_p&0&0&0\\
			0&0&-d&0&\tau_S\\
			\alpha_1\left(\hat{m}-\frac{\lambda_M}{\sigma}\right)\frac{\lambda_M}{\sigma}&0&0&-\sigma&0\\
			\tau_1\frac{\lambda_M}{\sigma}&0&0&0&-\tau_3
		\end{pmatrix}.
	\end{equation*}
	In addition, we observe that 
	\begin{equation}
		\begin{aligned}
			G_1(\mathcal{E}_0)&=\dfrac{\Delta t }{1+\gamma_0\Delta t}, \,\,  G_2(\mathcal{E}_0) = \dfrac{\Delta t}{1+\tau_p \Delta t  }, \, \, G_3(\mathcal{E}_0) = \dfrac{\Delta t}{1+d\Delta t},\\
			G_4(\mathcal{E}_0)&=\dfrac{\Delta t}{1+\sigma \Delta t },\, \, G_5(\mathcal{E}_0)=\dfrac{\Delta t}{1+\Delta t \tau_3}.
		\end{aligned}
	\end{equation}
	Then by Lemma \ref{lemma:jacobian} we get that the Jacobian of the discrete system \eqref{eq:nsfd2} evaluated at  disease-free equilibrium  $\mathcal{E}_0$ is 
	\begin{equation*}
		\mathcal{J}^{D}(\mathcal{E}_0) = 
		\begin{pmatrix}
			1-\frac{ \gamma_0 \Delta t}{1+\gamma_0\Delta t}&0&0&0&0\\
			\gamma_0&1-\frac{\tau_p\Delta t }{1+\tau_p \Delta t  },&0&0&0\\
			0&0&1-\frac{d \Delta t}{1+d\Delta t}&0&\tau_S\\
			\alpha_1\left(\hat{m}-\frac{\lambda_M}{\sigma}\right)\frac{\lambda_M}{\sigma}&0&0&1-\frac{\sigma \Delta t}{1+\sigma \Delta t }&0\\
			\tau_1\frac{\lambda_M}{\sigma}&0&0&0&1-\frac{\tau_3\Delta t}{1+\Delta t \tau_3}
		\end{pmatrix},
	\end{equation*}
	which has eigenvalues $$\lambda_1=\frac{1}{1+\gamma_0\Delta t}, \,\, \lambda_2= \frac{1}{1+\tau_p \Delta t  }  \,\, \lambda_3= \frac{1}{1+d \Delta t  }, \,\, \lambda_4=  \frac{1}{1+\sigma \Delta t  }, \,\, \lambda_5=  \frac{1}{1+\tau_3\Delta t }.$$
	We observe that $|\lambda_i|<1$, for $i=1,\dots,5$. So, according to Lemma \ref{lemma:disc-stability} the equilibrium $\mathcal{E}_0$ is locally asymptotically stable.
\end{proof}
\section{Numerical examples}\label{sec:numexa}
The main objective of this section is to present a series of examples that demonstrate the robustness of the constructed FV scheme. Throughout all the tests, certain model parameters are kept fixed while others are varied in order to explore different dynamical scenarios and to compare the results with those reported in the existing literature. In Table \ref{tab:par}, we specify the fixed and variable parameters along with their corresponding units. These values are taken from the works of Pujo-Menjouet \textit{et al.} \cite{andrade2020modeling,ciuperca2024qualitative,estavoyer2025spatial}. The meshes used in the 2D examples were generated with the Gmsh software \cite{geuzaine2009gmsh}, employing a Frontal-Delaunay algorithm that produces meshes satisfying the orthogonality condition stated in Definition \ref{def:am}.
\begin{table}[t]
	\centering
	\begin{tabular}{c|c|c||c|c|c}
		\toprule
		Parameter& Value & Units& Parameter& Value & Units\\
		\midrule
		$r_1$      &    0.1    & $L\,(\mathrm{mol})^{-1}(\mathrm{month})^{-1}$ & $\alpha_1 $&   1   & $L^2\,(\mathrm{mol})^{-2}(\mathrm{month})^{-1}$ \\
		$r_2$      &    0.1    & $L\,(\mathrm{mol})^{-1}(\mathrm{month})^{-1}$ & $\alpha_2 $&   1   & $L\,(\mathrm{mol})^{-1}$\\
		$d$        & variable  &            $(\mathrm{month})^{-1}$            & $\lambda_M$& 0.001 & $\mathrm{mol}\,L^{-1}\,(\mathrm{months})^{-1}$\\
		$\gamma_0$ &    0.05   &            $(\mathrm{month})^{-1}$            & $\hat{m}$  &   1 & $\mathrm{mol}\,L^{-1}$\\
		$\tau_1$   &      1    &            $L\,(\mathrm{mol})^{-1}$           & $\sigma$   & 0.001 & $(\mathrm{months})^{-1}$\\
		$\tau_2$   &      1    &            $L\, (\mathrm{month})^{-1}$        & $D_1$      & 1 & $m^2\,(\mathrm{months})^{-1}$\\
		$\tau_3$   &      1    &            $(\mathrm{month})^{-1}$            & $D_I$      & 1 & $m^2\,(\mathrm{months})^{-1}$\\
		$\tau_p$   &   0.03    &            $(\mathrm{month})^{-1}$            & $\nu_1$    & 1 & $m^2\,(\mathrm{months})^{-1}$\\
		$\tau_S$   &      1    &            $(\mathrm{month})^{-1}$            & $\nu_2$    & 1 & $m^2\,(\mathrm{months})^{-1}$\\
		$C$        &      1    &            $L^n\, (\mathrm{mol})^{-n}$        & $\alpha$   & variable & $m^2\,L\,(\mathrm{mol})^{-1}$\\ 
		$\nu$        &      2    &                      $-$                      &&&\\ \bottomrule
	\end{tabular}
	\caption{Parameters used in the numerical examples, along with their corresponding values and units.}
	\label{tab:par}
\end{table}
\subsection{Preliminaries} Let us detailed the numerical implementation of FV scheme \eqref{eq:fvm-1}-\eqref{eq:rhs-disc}. Given an admissible mesh $\mathcal{T}$, let us denote by $N_e$ the numbers of control volumes of $\mathcal{T}$ and by $\mathbf{u}_i^n\in \mathbb{R}^{N_e}$, $i=1,\dots,5$, the vectors of unknowns at time $t=t^n$. To state the algorithm, we define the following vectors of $\mathbb{R}^{N_e}$ which are evaluated at time $t=t_n$:
\begin{equation*}
	\boldsymbol{\gamma}^n:=\big(\gamma(u_{4,K}^n\big)_{K=1}^{N_e},\,\, \mathbf{a}^n:= \left(\dfrac{\tau_S}{1+C(u_{1,K}^n)^{\nu}}\right)_{K=1}^{N_e}, \,\, \mathbf{c}^n:=\left(\dfrac{\alpha_1 u_{1,K}^n}{1+\alpha_2 u_{1,K}^n}\right)_{K=1}^{N_e},\,\, \mathbf{b}^n:=\left(\dfrac{\tau_1 u_{1,K}^n}{1+\tau_2 u_{1,K}^n}\right)_{K=1}^{N_e},
\end{equation*}
and the constant vector $\boldsymbol{\lambda}_M:=(\lambda_{M,K})_{K=1}^{N_e}$. We also introduce the Hadamard product $\mathbf{u}*\mathbf{v}$, which is defined for vectors $\mathbf{u},\mathbf{v}\in \mathbb{R}^{N_e}$ as $(\mathbf{u}*\mathbf{v})_j = u_jv_j,$ for $j=1,\dots,N_e$, we denote by $\mathrm{diag}(\mathbf{v})$ the diagonal matrix with the vector $\mathbf{v}\in \mathbb{R}^{N_e}$ in the diagonal, by $\mathbf{I}\in \mathbb{R}^{N_e\times N_e}$ the identity matrix of size $N_e\times N_e$, and by $\mathbf{1}\in \mathbb{R}^{N_e}$ a vector of ones of size $N_e$. Let $\mathcal{L}\in \mathbb{R}^{N_e\times N_e}$ be the discretization of the Laplacian operator and $\mathcal{C}(\boldsymbol{u}_4;\boldsymbol{u}_1)$ the discretization of the numerical flux approximating the term $\nabla \cdot (\chi(u_4) \nabla u_1)$. Within this notation, the equation \eqref{eq:fvm-1} for the $A\beta$-oligomers becomes
\begin{equation*}
	\dfrac{u_{1,K}^{n+1}-u_{1,K}^n}{\Delta t}-d_1 (\mathcal{L} \mathbf{u}_1^{n+1})_K =  r_1 (u_{3,K}^n)^2-\gamma(u_{4,K}^n) u_{1,K}^{n+1}-\tau_0 u_{1,K}^{n+1},
\end{equation*}
which is linear in $\mathbf{u}_1^{n+1}$. Then, by rearranging the terms in vector form, we obtain
\begin{equation}\label{eq:disc-fvm1}
	[\mathbf{I}+\Delta t d_1 \mathcal{L}+\mathrm{diag}(\boldsymbol{\gamma}^n+\tau_0 \mathbf{1})]\mathbf{u}_1^{n+1} = \mathbf{u}_1^{n}+\Delta t r_1 \big(\mathbf{u}_3^n *\mathbf{u}_3^n\big). 
\end{equation}
We also notice that equations \eqref{eq:fvm-3} for oligomers and \eqref{eq:fvm-5} for interleukins, are linear in $\mathbf{u}_3^{n+1}$ and  $\mathbf{u}_5^{n+1}$, respectively, so we can write
\begin{equation}\label{eq:disc-fvm35}
	\begin{aligned}
		[\mathbf{I}+\Delta t d_3 \mathcal{L}+\mathrm{diag}(d+r_2\mathbf{u}_1^n+r_1\mathbf{u}_3^n)]\mathbf{u}_3^{n+1} &= \mathbf{u}_3^{n}+\Delta t \big(\mathbf{a}^n* \mathbf{u}_5^n\big),\\
		[\mathbf{I}+\Delta t d_5 \mathcal{L}+\Delta t \tau_3\mathbf{I}]\mathbf{u}_3^{n+1} &= \mathbf{u}_5^{n}+\Delta t \big(\mathbf{b}^n* \mathbf{u}_4^n\big).
	\end{aligned}
\end{equation}
The second equation \eqref{eq:fvm-2} for the Amyloids plaques can be put directly in vector form as
\begin{equation}\label{eq:disc-fvm2}
	[(1+\Delta \tau_p)\mathbf{I}] \mathbf{u}_2^{n+1} = \mathbf{u}_2^n+\Delta t\big( \boldsymbol{\gamma}^n* \mathbf{u}_1^n\big).
\end{equation}
For the nonlinear system \eqref{eq:fvm-4} of microglial cells, we employ a Newton–Raphson solver. To describe it, let us set $\mathbf{M}:=\mathbf{u}_4^{n+1}$. Given $\mathbf{u}_1^{n+1}$, the nonlinear system \eqref{eq:fvm-4} can be put in vector form as
\begin{equation}\label{eq:disc-fvm4}
	\mathbf{M}-\mathbf{u}_4^{n}-\Delta td_4 \mathcal{L}\,\mathbf{M} +\Delta t \mathcal{C}(\mathbf{M};\mathbf{u}_1^{n+1})-\bar{\mathbf{F}}_4(\mathbf{u}^n,\mathbf{M}) = 0,
\end{equation}
where 
\begin{equation*}
	\bar{\mathbf{F}}_4(\mathbf{u}^n,\mathbf{M}) = \diag(\mathbf{s}^n)(\hat{m}\mathbf{1}-\mathbf{M})-\sigma \mathbf{M}+\boldsymbol{\lambda}_M, \quad \mathbf{s}^n:=\mathbf{c}^n* \mathbf{u}_4^n.
\end{equation*}
We define the nonlinear residual $\mathbf{R}$ as the left-hand side of \eqref{eq:disc-fvm4}. So we look for $\mathbf{u}_4^{n+1}$ such that $\mathbf{R}\mathbf(\mathbf{u}_4^{n+1})=0$. We can also compute the Jacobian of the approximation by
\begin{equation}\label{eq:jac-nr}
	J(\mathbf{M})= \dfrac{\partial \mathbf{R}}{\partial \mathbf{M}} = \mathbf{I}-\Delta d_4 \mathcal{L}+ \Delta t\dfrac{\partial \mathcal{C}}{\partial \mathbf{M}}(\mathbf{M};\mathbf{u}_1^{n+1})+\Delta t \mathrm{diag}(\mathbf{s}^n+\sigma \mathbf{1}).
\end{equation}
Summarizing the above, we can formulate an algorithm for the coupled FV scheme, which is detailed in Algorithm 1.
\begin{algorithm}[t]\label{eq:algo}
	\caption{Semi-implicit FV scheme }
	\begin{algorithmic}[1]
		\Require $\mathbf{u}_1^0,\mathbf{u}_2^0,\mathbf{u}_3^0,\mathbf{u}_4^0,\mathbf{u}_5^0$, matrix $\mathcal{L}$, parameters, tolerance     $\varepsilon$
		\State $\mathbf{u}_i^n\gets \mathbf{u}_i^0$, $i=1,\dots,5$,  \, $n \gets 0$
		\State 
		$\Delta t\gets \mathrm{CFL}\cdot \min\left\{\frac{1}{\gamma_{\max}},\frac{1}{r_1\beta_3+\tau_S},\frac{\alpha_2}{\alpha_2+\hat{m}\alpha_1},\frac{\tau_2}{\tau_1}\right\}$
		\While{$t<T$} 
		\State Solve linear systems \eqref{eq:disc-fvm1}, \eqref{eq:disc-fvm2}, \eqref{eq:disc-fvm35}:
		$\mathbf{A}_i \mathbf{x}_i
		= \mathbf{b}_i,  \, i=1,2,3,5.$
		\State $\mathbf{u}_i^{n+1}\gets \mathbf{x}_i$, $i=1,2,3,5$
		\State Initialize $\mathbf{u}_4^{(0)} \gets \mathbf{u}_4^{\,n}$
		
		\For{$j = 0,1,2,\dots$}
		
		\State Build the residual from \eqref{eq:disc-fvm4}
		
		\State $\mathbf{R}(\mathbf{u}_4^{(j)}) \gets
		\mathbf{u}_4^{(j)} - \mathbf{u}_4^{\,n}
		- \Delta t\,d_4 \mathcal{L}\,\mathbf{u}_4^{(j)}
		+ \Delta t\,\mathcal{C}\big(\mathbf{u}_4^{(j)};\mathbf{u}_1^{\,n+1}\big)
		- \Delta t\,\bar{\mathbf{F}}_4\big(\mathbf{u}^{\,n},\mathbf{u}_4^{(j)}\big)$
		
		\State Compute the residual norm $\|\mathbf{R}(\mathbf{u}_4^{(j)})\|$
		
		\If{$\|\mathbf{R}(\mathbf{u}_4^{(j)})\| < \varepsilon$}
		\State \textbf{break}
		\EndIf
		
		\State Build the Jacobian $J(\mathbf{u}_4^{(j)})$ from \eqref{eq:jac-nr}
		
		\State Solve the linear system:
		$J(\mathbf{u}_4^{(j)})\,\delta \mathbf{u}_4^{(j)}
		= -\mathbf{R}(\mathbf{u}_4^{(j)})$
		
		\State Update: $\mathbf{u}_4^{(j+1)}\gets \mathbf{u}_4^{(j)} + \delta \mathbf{u}_4^{(j)}.$
		
		\EndFor
		
		\State Set $\mathbf{u}_4^{\,n+1} \gets \mathbf{u}_4^{(j+1)}$
		\State $t\gets t+ \Delta t$, \,  $n\gets n+1$
		\EndWhile
		\Ensure $\mathbf{u}_1^N,\mathbf{u}_2^N,\mathbf{u}_3^N,\mathbf{u}_4^N,\mathbf{u}_5^N$
	\end{algorithmic}
\end{algorithm}
For all examples, we implemented this numerical scheme using the MATLAB software. For the solution of the linear systems \eqref{eq:disc-fvm1}, \eqref{eq:disc-fvm35}, \eqref{eq:disc-fvm2} we make use of the backslash MATLAB function. 
\begin{figure}[t]
	\centering
	\includegraphics[scale=0.25]{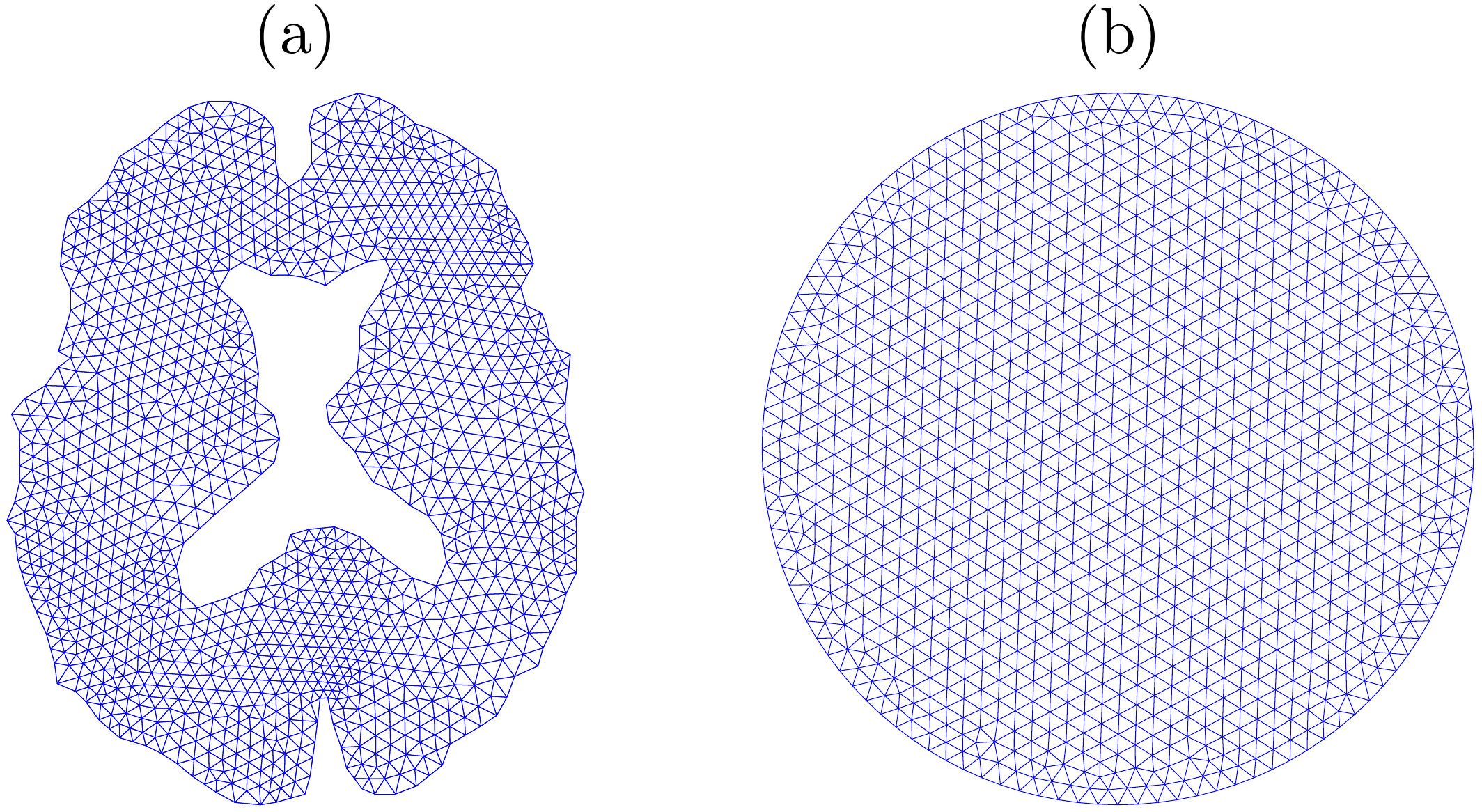}
	\caption{(a) Brain-shaped domain and reference mesh with $2658$ triangles. (b) Circular domain and reference mesh with $2358$ triangles.}\label{fig:meshes}
\end{figure}
\subsection{Example 1. Robustness of NSFD discretization.} 
In this example, we focus on the NSFD scheme \eqref{eq:nsfd2}, which approximates the solution of the spatially homogeneous model \eqref{eq:modelhomo} studied in \cite{ciuperca2024qualitative}. The goal is to test the robustness of the method with respect to the time step size. To do so, we consider a sequence of $\Delta t \in \{0.5,1.3,2\}$ and perform simulations until $T=200$. We compare the stability of the numerical solution obtained by NSFD with the traditional explicit Euler scheme. 
In this test, we consider $d=0.15$ and the rest of the parameters as in Table \ref{tab:par}. Within these parameters, the invariant rectangle $\mathcal{R}$ is defined by 
\begin{equation}\label{eq:rtangle1}
	\mathcal{R} = [0,88.1]\times[0,148.14]\times [0, 6.66]\times [0,1]\times [0, 1],
\end{equation}
and the condition $\hat{m}\geq \lambda_M/\sigma$ is fulfilled. We set an initial condition given by $\boldsymbol{u}^{0} = (0.0004,0,0.003,1,0.4)^{\mathrm{T}}$ belonging to $\mathcal{R}$. According to \cite{ciuperca2024qualitative}, this configuration leads the system to converge toward a positive equilibrium. In Figure \ref{fig:exa1}, we present the numerical results obtained using the NSFD scheme \eqref{eq:nsfd2} (right column) and the Euler scheme (left column). For the first time step, $\Delta t = 0.5$, we observe (first row) that both schemes remain stable and the numerical solutions lie within the invariant rectangle $\mathcal{R}$. However, for larger time steps, $\Delta t = 1.3$ and $\Delta t = 2$ (second and third rows, respectively), spurious oscillations appear in the 
{explicit}
Euler scheme, clearly violating the preservation of the invariant region of $\mathcal{R}$. In contrast, our NSFD scheme remains stable with respect to the time step $\Delta t$, in the sense that it preserves the dynamical properties of the continuous system regardless of the chosen time step.
\begin{figure}[ht!]
	\centering
	\includegraphics[scale=0.5]{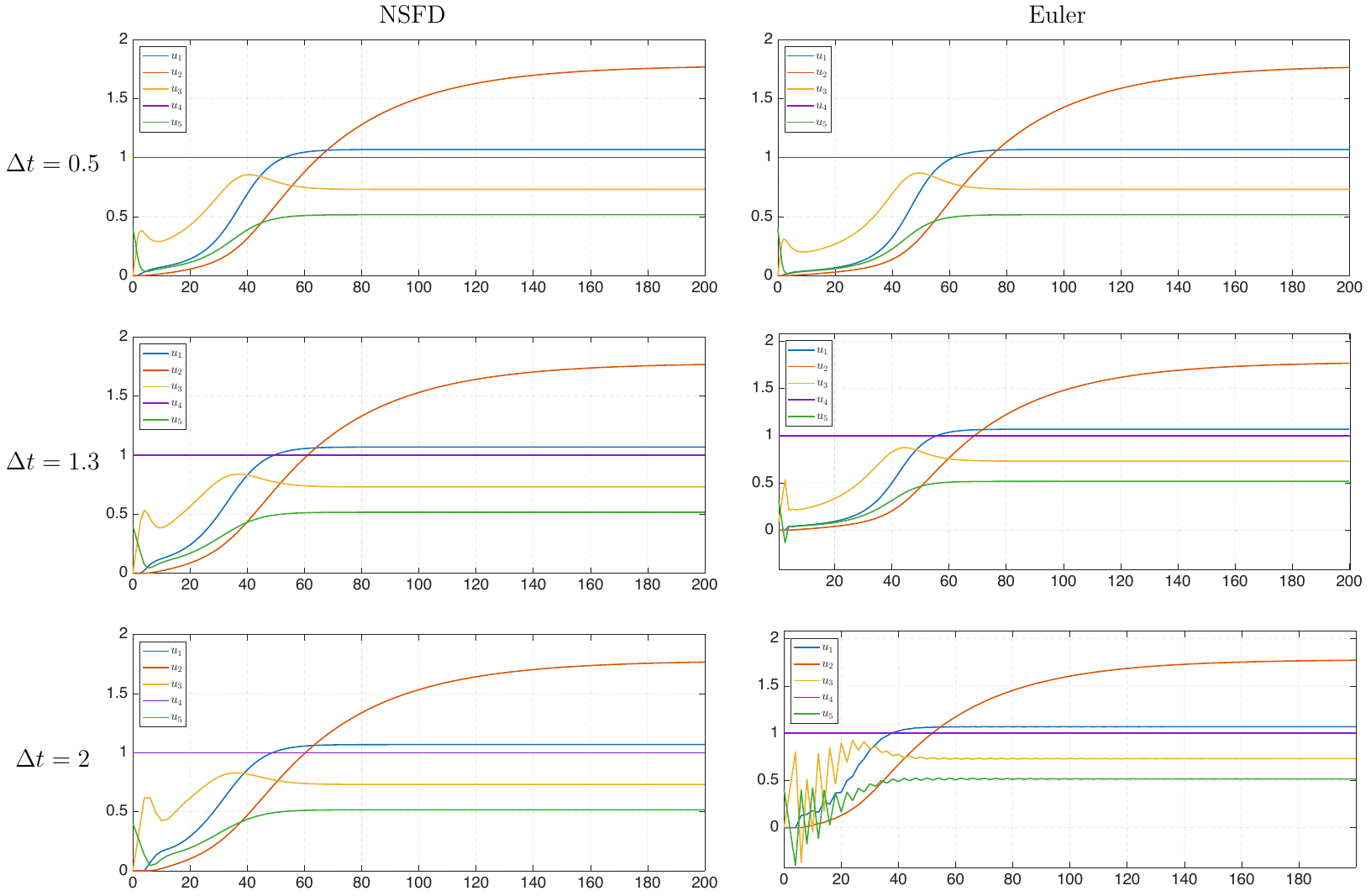}
	\caption{Numerical solutions for the five populations $u_i(t)$, $i=1,\dots,5$ for the SH model \eqref{eq:modelhomo}, with $0\leq t\leq 200$ computed with NSFD (first column) and Euler (second column) numerical schemes with step sizes $\Delta t\in \{0.5,1.3,2\}$.}\label{fig:exa1}
\end{figure}

\subsection{Example 2. Chemotaxis effect}  In this example, we study the chemotactic response of microglial cells to an increased population of oligomers. This chemotactic effect leads to the activation of microglial cells in the presence of oligomers, triggering an inflammatory reaction characterized by the production of interleukins. Furthermore, we highlight the ability of our scheme to operate effectively on complex geometrical domains. To this end, we consider a brain-shaped domain, as shown in Figure \ref{fig:meshes} (a), with a triangulation consisting of $12289$ elements satisfying the conditions in Definition \ref{def:am}. We employ $d=0.15$, $\alpha=24$ and the rest parameters of the parameters as in Table \ref{tab:par}. The initial condition $\boldsymbol{u}^0$ introduces small random perturbations in the oligomers $u_1$ around the value $0.1$, while the remaining variables are initialized with spatially homogeneous or localized distributions. In particular, $u_4$ presents two Gaussian peaks centered at points $\boldsymbol{x}_1$ and $\boldsymbol{x}_2$ in $\Omega$ (see first column of Figure \ref{fig:exa2}), representing localized concentrations of microglial cells within the domain, i.e. we set $\boldsymbol{u}^{0}=(u_{1,K}^{0},\dots,u_{5,K}^{0})_{K\in \mathcal{T}}$  is set as $u_{1,K}^{0}=0.1+ 0.01(1-2 r_K)$, $r_K\sim \mathcal{U}(0,1)$, $u_{2,K}^0 =0, u_{3,K}^0 = 0.1$, $$u_{4,K}^0 = 0.5+0.05\exp\left(-\dfrac{\|\boldsymbol{x}_K-\boldsymbol{x}_1\|^2}{40}\right)+0.05\exp\left(-\dfrac{\|\boldsymbol{x}_K-\boldsymbol{x}_2\|^2}{40}\right),$$ $u_{5,K} = 0.2$, for each $K\in \mathcal{T}$.The simulation is carried out up to $T = 50$, and the numerical results are presented in Figure \ref{fig:exa2} for specific intermediate times. We observe that the concentration of microglial cells increases in correspondence with the spatial distribution of oligomers. Furthermore, the invariance of the rectangle \eqref{eq:rtangle1} is preserved by the FV scheme \eqref{eq:fvm-1}-\eqref{eq:fvm-5}.
\begin{figure}[ht!]
	\centering
	\includegraphics[scale=0.18]{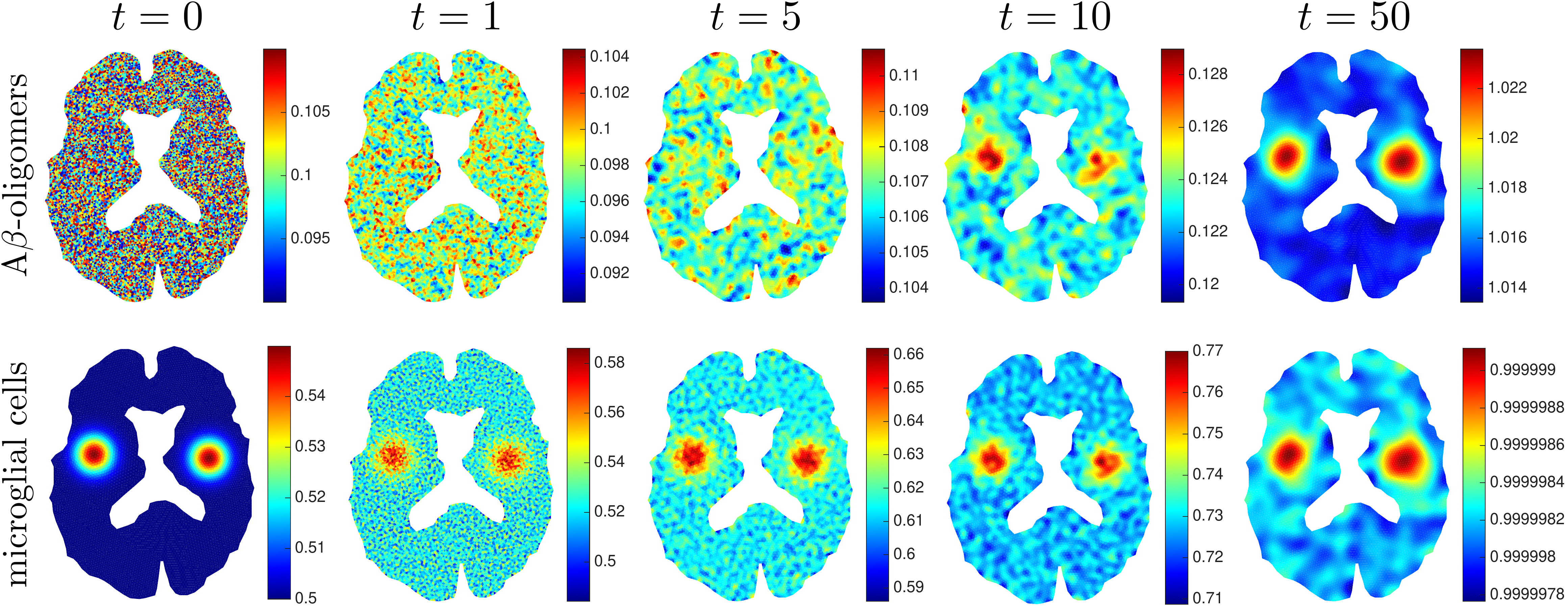}
	\caption{Numerical solutions of the concentrations of $A\beta$-oligomers $u_1(\boldsymbol{x},t)$ (first row) and microglial cells $u_4(\boldsymbol{x},t)$ (second row) for model \eqref{eq:model} over the time interval $0 \leq t \leq 50$. Snapshots are shown at $t = 0$ (initial condition), $t = 1$, $t = 5$, $t = 10$, and $t = 50$.}\label{fig:exa2}
\end{figure}
\subsection{Example 3. Turing pattern formation.}  In this example, we aim to study the formation of Turing patterns reported in \cite{estavoyer2025spatial} for model \eqref{eq:model}, with the sensitivity function given by $\chi(u) = \alpha u$. We remark that the convergence proof for this particular choice of $\chi$ can be established in a manner analogous to the proof presented in this work for $\chi$ defined by \eqref{eq:chifun}. We aim to investigate whether the FV scheme can reproduce the two types of patterns reported in \cite{estavoyer2025spatial}, namely stripe and dot patterns. To do so, we set a circular domain as shown in Figure \ref{fig:meshes} (b), with a triangulation consisting of $9237$ elements satisfying the conditions in Definition \ref{def:am}. We set initial conditions as a perturbation of the positive equilibrium point $\mathcal{E}^{*} = (1.0686,1.7739,0.7310,1.0,0.5166)$, i.e. we set $\boldsymbol{u}_K^0 = \mathcal{E}^{*}+0.001(1-2r_K)$, $r_K\sim \mathcal{U}(0,1)$, for all $K\in \mathcal{T}$. Notice that, in this example we have that $\boldsymbol{u}^0\notin \mathcal{R}$, therefore, we cannot expect the numerical solution produced by our FV scheme to remain within $\mathcal{R}$, as illustrated in Figures \ref{fig:exa3a}-\ref{fig:exa3b}. We set $d=0.15$ and $\alpha=24$ (stripes pattern, Figure \ref{fig:exa3a}) and $\alpha=40$ (dot pattern, Figure \ref{fig:exa3b}); the remaining parameters are taken from Table \ref{tab:par} and we simulate until $T=2000$. 
\begin{figure}[ht!]
	\centering
	\includegraphics[scale=0.31]{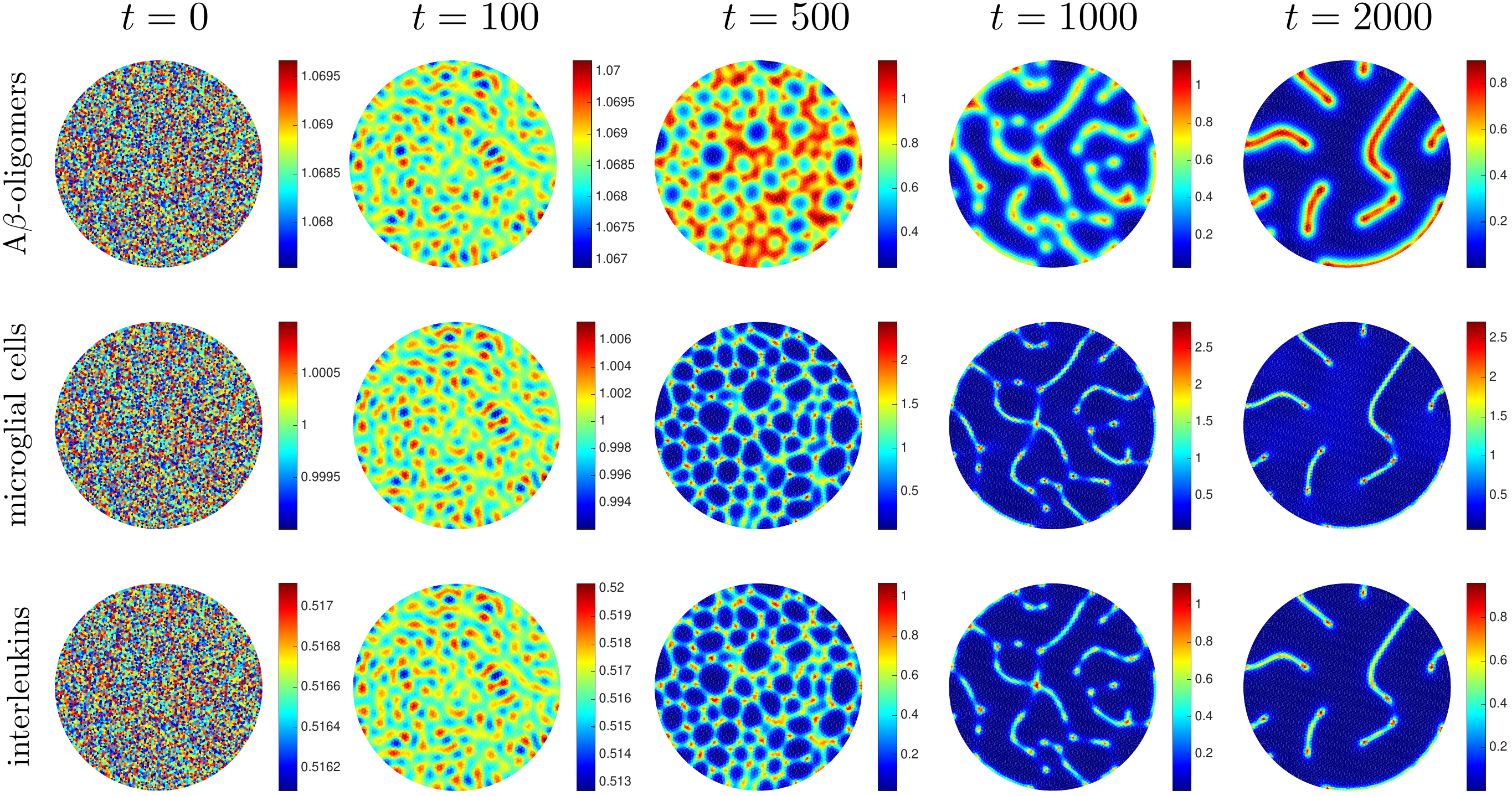}
	\caption{Numerical solutions of the concentrations of $A\beta$-oligomers $u_1(\boldsymbol{x},t)$ (first row), microglial cells $u_4(\boldsymbol{x},t)$ (second row), and interleukins $u_5(\boldsymbol{x},t)$ (third row) for model \eqref{eq:model} over the time interval $0 \leq t \leq 2000$. Snapshots are shown at $t = 0$ (initial condition), $t = 100$, $t = 500$, $t = 1000$, and $t = 2000$.}\label{fig:exa3a}
\end{figure}
\begin{figure}[ht!]
	\centering
	\includegraphics[scale=0.31]{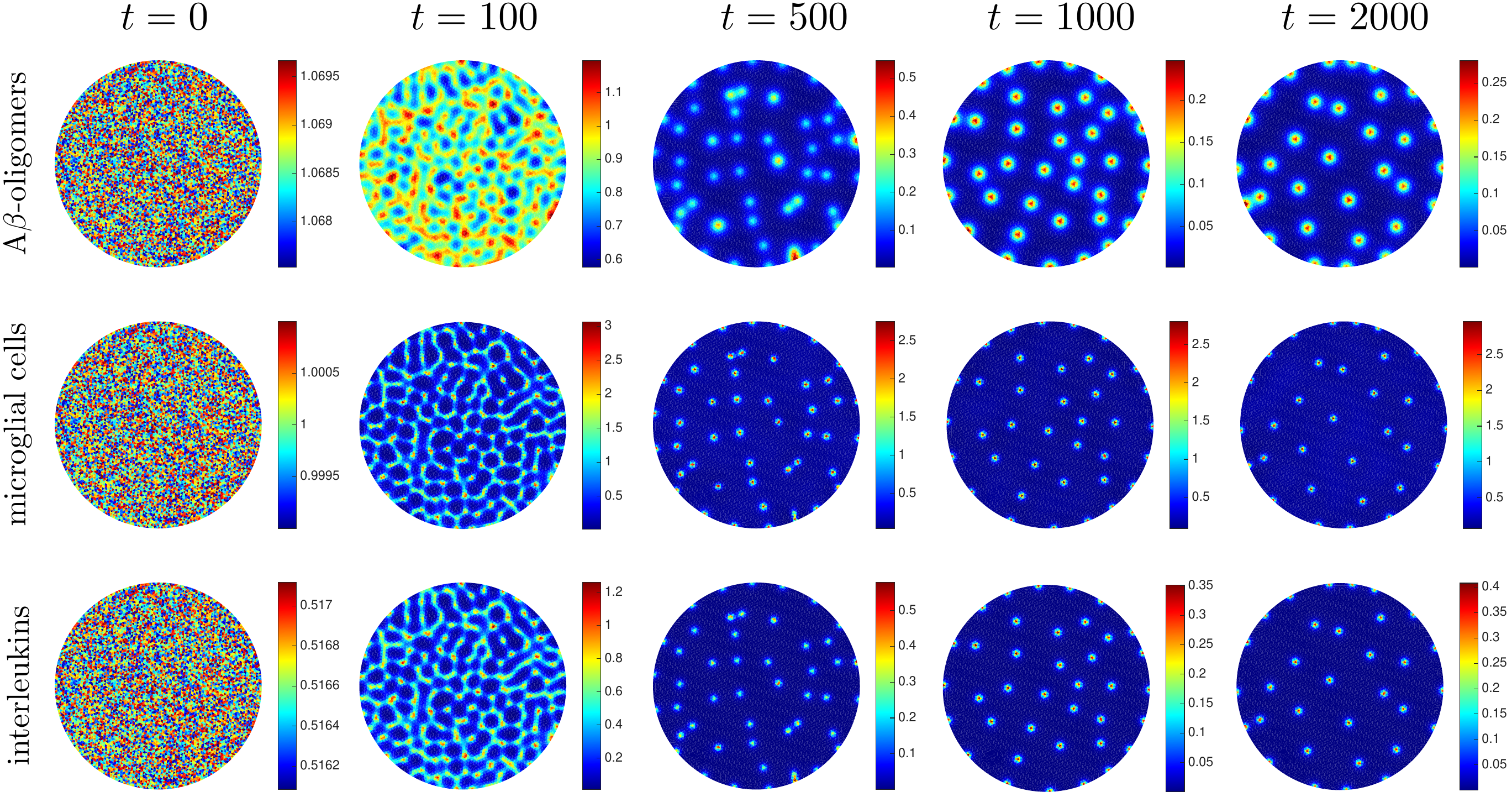}
	\caption{Numerical solutions of the concentrations of $A\beta$-oligomers $u_1(\boldsymbol{x},t)$ (first row), microglial cells $u_4(\boldsymbol{x},t)$ (second row), and interleukins $u_5(\boldsymbol{x},t)$ (third row) for model \eqref{eq:model} over the time interval $0 \leq t \leq 2000$. Snapshots are shown at $t = 0$ (initial condition), $t = 100$, $t = 500$, $t = 1000$, and $t = 2000$.}\label{fig:exa3b}
\end{figure}
\section{Conclusions}  \label{sec:conc} 
In this work, we have developed and analyzed a finite volume scheme for a reaction–diffusion–chemotaxis model describing the interactions among A$\beta$-monomers, A$\beta$-oligomers, microglial cells, interleukins, and neurons in the progression of Alzheimer’s disease. The first contribution of this paper is to present a modified version of the model addressed in \cite{ciuperca2024qualitative} by incorporating a more general sensitivity function $\chi$. This modification yields a more realistic description of microglial chemotactic movement toward A$\beta$-oligomer concentrations, as the sensitivity naturally vanishes once the microglial population reaches its recruitment threshold $\hat{m}$. Next, we propose a novel semi-implicit FV scheme to approximate the resulting system and then, by employing a priori $L^2$-estimates and compactness arguments, we show that a sequence of discrete solutions produced by this scheme converges to an admissible weak solution of the full chemotaxis-inflammation model, as stated in Theorem \ref{thm:conv}. The FV method incorporates nonstandard finite difference (NSFD) techniques for reaction terms, ensuring that the discrete dynamics reproduce the qualitative behavior of the SH ordinary differential system. We show in Section \ref{sec:dc} that the method preserves: positivity and boundedness of populations (i.e. the invariance of the rectangle $\mathcal{R}$), the equilibrium points, and the local stability of the disease-free equilibrium.

In Section \ref{sec:numexa}, through a series of numerical tests, the FV scheme demonstrates robustness with respect to geometry of the domain, time-step size, and parameter variations. Moreover, the numerical results successfully capture microglial chemotaxis toward A$\beta$-oligomer concentrations, consistent with biological mechanisms and existing literature \cite{sakono_amyloid_2010,forloni_alzheimers_2018}. Furthermore, the method is shown to reproduce the stripe and dot Turing patterns reported previously for this model \cite{estavoyer2025spatial}, demonstrating that the discretization is not only stable but also able to capture complex spatial morphologies.

Finally, we are interested in performing a numerical study of a more general model in which a concentration of A$\beta$-proto-oligomers of different sizes is considered \cite{ciuperca2024qualitative}. Furthermore, we seek to establish high-order numerical schemes preserving the dynamic of the continuous model.
\section*{Acknowledgments}
JBC is supported by the National Agency for Research and Development, ANID-Chile through the Scholarship Program, Becas Doctorado Nacional 2022, 21221387 and for PID2023-146836NB-I00, granted by MCIN/ AEI /10.13039/501100011033. NT is supported by the project EUR SPECTRUM with the initiative IDEX of Université de Côte d'Azur. MS, NT and LMV are supported by the Center for Mathematical Modeling (CMM) BASAL fund FB210005 for center of excellence from ANID-Chile.  MS is supported by the Fondecyt project 1220869.


\bibliographystyle{unsrt}
\bibliography{bibliography-bvst}







\end{document}